\documentclass[]{amsart}
\usepackage{mathrsfs}
\usepackage{color}
\usepackage{amsmath}
\usepackage{amsfonts}
\usepackage{color}
\usepackage{amssymb}

\newtheorem{Theorem}{Theorem}[section]
\newtheorem{Corollary}[Theorem]{Corollary}
\newtheorem{Example}[Theorem]{Example}
\newtheorem{Lemma}[Theorem]{Lemma}
\newtheorem{Proposition}[Theorem]{Proposition}
\newtheorem{Remark}[Theorem]{Remark}
\newtheorem{Definition}[Theorem]{Definition}

\numberwithin{equation}{section}

\DeclareMathOperator{\D}{D}

\begin{document}

	\title[Tame maximal weights, relative types and valuations]
{Tame maximal weights, relative types and valuations}

\author{Shijie Bao}
\address{Shijie Bao: Institute of Mathematics, Academy of Mathematics
	and Systems Science, Chinese Academy of Sciences, Beijing 100190, China.}
\email{bsjie@amss.ac.cn}

\author{Qi'an Guan}
\address{Qi'an Guan: School of Mathematical Sciences,
	Peking University, Beijing, 100871, China.}
\email{guanqian@math.pku.edu.cn}

\author{Zhitong Mi}
\address{Zhitong Mi: School of Mathematics and Statistics, Beijing Jiaotong University, Beijing,
	100044, China.
}
\email{zhitongmi@amss.ac.cn}

\author{Zheng Yuan}
\address{Zheng Yuan: Institute of Mathematics, Academy of Mathematics
	and Systems Science, Chinese Academy of Sciences, Beijing 100190, China.}
\email{yuanzheng@amss.ac.cn}

\thanks{}



\date{\today}

\dedicatory{}

\commby{}

\maketitle

\begin{abstract}
	In this article, we obtain a class of tame maximal weights (Zhou weights). Using Tian functions (the function of jumping numbers with respect to the exponents of a holomorphic function or the multiples of a plurisubharmonic function) as a main tool, we establish an expression of relative types (Zhou numbers) to these tame maximal weights in integral form, which shows that the relative types satisfy tropical multiplicativity and tropical additivity. Thus, the relative types to Zhou weights are valuations (Zhou valuations) on the ring of germs of holomorphic functions. 
	We use Tian functions and  Zhou numbers to measure the singularities of plurisubharmonic functions, involving jumping numbers and multiplier ideal sheaves.
	Especially, the relative types to Zhou weights characterize the division relations of the ring of germs of holomorphic functions. Finally, we consider a global version of Zhou weights on domains in $\mathbb{C}^n$, which is a generalization of the pluricomplex Green functions, and we obtain some properties of them, including continuity and some approximation results.
\end{abstract}

\maketitle

\tableofcontents

\section{Introduction}
The concept of \emph{Lelong number} is an analytic analogue of the algebraic notion of multiplicity (see \cite{demailly2010}), which is a fundamental invariant in several complex variables and was first introduced by Lelong in \cite{lelong57}.
Let $u$ be a plurisubharmonic function (see \cite{demailly-book,Si,siu74}) near the origin $o$ in $\mathbb{C}^n$. The Lelong number was defined by $$\nu(u,o):=\sup\big\{c\ge0:u(z)\le c\log|z|+O(1)\text{ near }o\big\},$$
which was used to measure the singularity of $u$ at $o$ (see \cite{demailly-book}). Some further developments can be seen in \cite{lelong69,siu74,lelong77,demailly82,demailly85,demailly87}.   

To give more accurate insights into the singularities of plurisubharmonic functions, some generalized Lelong numbers were studied.
Let $a=(a_1,\ldots,a_n)\in\mathbb{R}_{>0}^n$. Replacing $\log|z|$ by $\log\max_{1\le j\le n}|z_j|^{\frac{1}{a_j}}$ in the definition of Lelong numbers, the directional Lelong numbers $\nu_{a}(u,o)$ (see \cite{demailly-book}, also known as \emph{Kiselman numbers}) were introduced by Kiselman, which gave more detailed information on the singularity of $u$ at $o$. 
Note that, $\nu_{a}(\cdot,o)$ is  well suited for the tropical structure of the cone of plurisubharmonic functions (see \cite{Rash06,demailly-book}):
$$\nu_{a}(u+v,o)=\nu_a(u,o)+\nu_a(v,o)\text{   (\emph{tropical multiplicativity}) }$$
and
$$\nu_a\big(\max\{u,v\},o\big)=\min\big\{\nu_a(u,o),\nu_a(v,o)\big\}\text{  (\emph{tropical additivity}). }$$
These properties play an important role, for example, in the investigation of valuations on the ring of germs of holomorphic functions at $o$ (see \cite{FM05i, Rash06}).

With respect to a more general plurisubharmonic functions $\varphi$, an important generalized Lelong number 
\begin{equation}
	\label{eq:240608a}
	\nu_{\varphi}(u):=\lim_{t\rightarrow+\infty}\int_{\{\varphi<-t\}}dd^cu\wedge(dd^c\varphi)^{n-1}
\end{equation}
was introduced and studied by Demailly (see \cite{demailly82,demailly87,demailly-book}), and is called \emph{Demailly's generalized Lelong numbers} or \emph{Lelong-Demailly numbers}. Taking $\varphi=\log|z|$ or $\log\max_{1\le j\le n}|z_j|^{\frac{1}{a_j}}$, Demailly's generalized Lelong numbers degenerate to classical Lelong numbers and directional Lelong numbers respectively (see \cite{demailly-book}).  Demailly's generalized Lelong numbers are tropically multiplicative. However the tropical additivity may not hold for general plurisubharmonic functions $\varphi$. 
Motivated by the definitions of classical and directional Lelong numbers, Rashkovskii \cite{Rash06} introduced the \emph{relative types}
$$\sigma(u,\varphi):=\sup\big\{c\ge0:u\le c\varphi+O(1)\text{ near }o\big\}$$
of plurisubharmonic functions $u$ to maximal weights $\varphi$ with an isolated singularity at $o$, which are also generalizations of Lelong numbers. It is clear that the relative types are tropically additive. However,  the tropical multiplicativity of relative types is not satisfied generally. Some relations between Demailly's generalized Lelong numbers and relative types can be seen in \cite{Rash10}.

In \cite{BFJ08}, Boucksom-Favre-Jonsson established a fundamental result of the relations between  Lelong numbers,  multiplier ideals sheaves,  relative types, and Demailly's generalized Lelong numbers.

\begin{Theorem}[\cite{BFJ08}]
	\label{thm:BFJ}For any two plurisubharmonic functions $u$ and $v$ near $o$, consider the following four statements:
	
	$(1)$ for any proper modifications $\pi:X_{\pi}\rightarrow\mathbb{C}^n$ above $o$ and any points $p\in\pi^{-1}(o)$, the Lelong numbers $\nu(u\circ\pi,p)=\nu(v\circ\pi,p)$;
	
	$(2)$ for any $t>0$, we have $\mathcal{I}_+(tu)_o=\mathcal{I}_+(tv)_o$, where $\mathcal{I}_+(u):=\bigcup_{p>1}\mathcal{I}(pu)$;
	
	$(3)$ for any tame maximal weight $\varphi$, the relative types $\sigma(u,\varphi)=\sigma(v,\varphi)$;
	
	$(4)$ for any tame  weight $\varphi$,  Demailly's generalized Lelong numbers $\nu_{\varphi}(u)=\nu_{\varphi}(v)$.
	
	Then $(1)$, $(2)$ and $(3)$ are equivalent, and imply $(4)$.
\end{Theorem}

The definitions of \emph{maximal weights} and \emph{tame weights} can be seen in Section \ref{sec:1.2} (see also \cite{BFJ08}). For tame maximal weights, the tropical multiplicativity of relative types is not satisfied generally. 

The \emph{multiplier ideal sheaf} $\mathcal{I}(u)$ (see \cite{demailly2010}) related to a plurisubharmonic function $u$ is defined as follows, which plays an important role in complex geometry and algebraic geometry, which was widely discussed (see e.g. \cite{tian87,Nadel90,siu96,DEL00,D-K01,demailly-note2000,D-P03,Laz04,siu05,siu09,demailly2010}):

\emph{Let $u$ be a plurisubharmonic function on a complex manifold. The multiplier ideal sheaf $\mathcal{I}(u)$ is the sheaf of germs of holomorphic functions $f$ such that $|f|^{2}e^{-2u}$ is local integrable.}

After Demailly's \emph{strong openness conjecture} i.e. $\mathcal{I}(u)=\mathcal{I}_+(u)$ (see \cite{demailly-note2000,demailly2010}) was solved \cite{GZopen-c},
Professor Xiangyu Zhou suggested the second author to explore the analytic counterparts of the items in algebraic geometry, such as valuations. 

In this article, we obtain a class of tame maximal weights $\Phi_{o,\max}$ related to the strong openness property of multiplier ideal sheaves, and we call them \emph{Zhou weights} (see Section \ref{sec:1.1}). We establish an expression  of the relative types $\sigma(u,\Phi_{o,\max})$ of plurisubharmonic functions $u$ to Zhou weights $\Phi_{o,\max}$ in integral form by using \emph{Tian functions} as a main tool (see Section \ref{sec:tian function}), which shows that $\sigma(\cdot,\Phi_{o,\max})$ are tropically multiplicative. Especially, noting that $\log|f|$ is plurisubharmonic for any holomorphic function $f$,  we obtain a class of valuations on the ring  $\mathcal{O}_o$ of germs of holomorphic functions, denoted by $\nu(\cdot,\Phi_{o,\max})$, and we call them \emph{Zhou valuations}.

Then, we use $\sigma(\cdot,\Phi_{o,\max})$ and Tian functions to measure the singularities of plurisubharmonic functions near $o$, and we obtain some relations between  $\sigma(\cdot,\Phi_{o,\max})$ (or $\nu(\cdot,\Phi_{o,\max})$), Tian functions, jumping numbers and multiplier ideal sheaves: for any Zhou weights $\Phi_{o,\max}$, we prove that $c_o^{f}(\Phi_{o,\max})$ (jumping number) and $\nu(f,\Phi_{o,\max})$ for all $(f,o)\in\mathcal{O}_o$ are linearly controlled by each other; we give a characterization for inclusion relations of multiplier ideal sheaves by using Tian functions and $\sigma(\cdot,\Phi_{o,\max})$ to all Zhou weights $\Phi_{o,\max}$, which shows that Theorem \ref{thm:BFJ} also holds if replacing tame maximal weights by Zhou weights; as an application, the valuations $\nu(\cdot,\Phi_{o,\max})$ of $\mathcal{O}_o$  characterize the division relations in $\mathcal{O}_o$; we present an equality about jumping numbers and Zhou numbers, which shows that the Zhou numbers can characterize multiplier ideal sheaves.
Finally, we discuss \emph{global Zhou weights} on domains in $\mathbb{C}^n$, and present some properties of them, including continuity and some approximation results.

\subsection{Valuations related to local Zhou weights}
\label{sec:1.1}

Let $f_{0}=(f_{0,1},\cdots,f_{0,m})$ be a vector,
where $f_{0,1},\cdots,f_{0,m}$ are holomorphic functions near $o$.
Denote $|f_{0}|^{2}=|f_{0,1}|^{2}+\cdots+|f_{0,m}|^{2}$.
Let $\varphi_{0}$ be a plurisubharmonic function near $o$,
such that $|f_{0}|^{2}e^{-2\varphi_{0}}$ is integrable near $o$.

\begin{Definition}
	\label{def:max_relat}
	We call that $\Phi^{f_0,\varphi_0}_{o,\max}$ ($\Phi_{o,\max}$ for short) is a \textbf{local Zhou weight related to $|f_{0}|^{2}e^{-2\varphi_{0}}$ near $o$},
	if the following three statements hold
	
	$(1)$ $|f_{0}|^{2}e^{-2\varphi_{0}}|z|^{2N_{0}}e^{-2\Phi_{o,\max}}$ is integrable near $o$
	for large enough $N_{0}\gg0$;
	
	$(2)$ $|f_{0}|^{2}e^{-2\varphi_{0}}e^{-2\Phi_{o,\max}}$
	is not integrable near $o$;
	
	$(3)$ for any plurisubharmonic function $\varphi'\geq\Phi_{o,\max}+O(1)$ near $o$
	such that $|f_{0}|^{2}e^{-2\varphi_{0}}e^{-2\varphi'}$
	is not integrable near $o$,
	$\varphi'=\Phi_{o,\max}+O(1)$ holds.
\end{Definition}

Let $\varphi$ be a plurisubharmonic function near $o$. The existence of local Zhou weights follows from the strong openness property of multiplier ideal sheaves (see Theorem \ref{thm:SOC}, see also \cite{GZopen-c}).

\begin{Remark}
	\label{rem:max_existence}
	Assume that $|f_{0}|^{2}e^{-2\varphi_{0}}|z|^{2N_{0}}e^{-2\varphi}$ is integrable near $o$
	for large enough $N_{0}\gg0$,
	and $(f_{0},o)\not\in\mathcal{I}(\varphi+\varphi_{0})_o$ holds.
	
	Then there exists a local Zhou weight $\Phi_{o,\max}$ related to $|f_{0}|^{2}e^{-2\varphi_{0}}$ near $o$
	such that $\Phi_{o,\max}\geq\varphi$.
	
	Moreover, $\Phi_{o,\max}\geq N\log|z|+O(1)$ near $o$ for some $N\gg0$. We prove this remark in Section \ref{sec:proofs of remark1.2,1.3}.
\end{Remark}

A plurisubharmonic function $\varphi$ on a neighborhood of $o$ is said a \emph{maximal weight}, if $u(o)=-\infty$, $u$ is locally bounded on  $U\backslash\{o\}$ and $(dd^c \varphi)^n=0$ on $U\backslash\{o\}$ (see \cite{Rash06,BFJ08}), where $U$ is a neighborhood of $o$.
Let $\psi$ be any plurisubharmonic function near $o$, and $\varphi$ be a maximal weight. Denote 
$$\sigma(\psi,\varphi):=\sup\big\{b : \psi\leq b\varphi+O(1)\text{ near $o$}\big\},$$
which is called the \emph{relative type} of $\psi$ with respect to $\varphi$ (see \cite{Rash06,BFJ08}). Note that for any local Zhou weights $\Phi_{o,\max}$ near $o$, there exists a maximal weights $\varphi$ on a small enough neighborhood  of $o$ such that $\varphi=\Phi_{o,\max}+O(1)$ near $o$ (see Lemma \ref{l:local-glabal} and Proposition \ref{l:max2}). Then, we call the relative types $\sigma(\psi,\Phi_{o,\max})$ of $\psi$ to Zhou weights $\Phi_{o,\max}$  \textbf{Zhou numbers of $\psi$ to $\Phi_{o,\max}$}.

Note that for any $b<\sigma(\psi,\Phi_{o,\max})$,
$|f_{0}|^{2}e^{-2\varphi_{0}}e^{-2\max\{\Phi_{o,\max},\frac{1}{b}\psi\}}$ is not integrable near $o$.
Then it follows from the strong openness property of multiplier ideal sheaves (Theorem \ref{thm:SOC})
that $|f_{0}|^{2}e^{-2\varphi_{0}}e^{-2\max\big\{\Phi_{o,\max},\frac{1}{\sigma(\psi,\Phi_{o,\max})}\psi\big\}}$
is not integrable near $o$.
Note that
$$\max\left\{\Phi_{o,\max},\frac{1}{\sigma(\psi,\Phi_{o,\max})}\psi\right\}\geq\Phi_{o,\max}.$$
Then
$$\max\left\{\Phi_{o,\max},\frac{1}{\sigma(\psi,\Phi_{o,\max})}\psi\right\}=\Phi_{o,\max}+O(1),$$
which implies that
$$\Phi_{o,\max}\geq\frac{1}{\sigma(\psi,\Phi_{o,\max})}\psi+O(1),$$
i.e.,
\begin{equation*}
	\psi\leq \sigma(\psi,\Phi_{o,\max})\Phi_{o,\max}+O(1).
\end{equation*}
In fact, for any maximal weights $\varphi$, Rashkovskii proved $\psi\leq \sigma(\psi,\varphi)\varphi+O(1)$ near $o$ (see \cite{Rash06}).

\begin{Remark}
	\label{rem:sharp_bound}
	Let $\Phi_{o,\max}$ be a local Zhou weight related to $|f_{0}|^{2}e^{-2\varphi_{0}}$ near $o$.
	Then for any small enough neighborhood $U\ni o$,
	there exists a (unique) negative plurisubharmonic function $\Phi^U_{o,\max}=\Phi_{o,\max}+O(1)$ on $U$,
	such that for any negative plurisubharmonic function $\psi$ on $U$,
	the inequality
	\begin{equation}
		\label{equ:xiaomage}
		\psi\leq \sigma(\psi,\Phi_{o,\max})\Phi^U_{o,\max}
	\end{equation}
	holds on $U$.
	
	Especially, if plurisubharmonic functions $\psi_{j}\to\psi$ in $L^{1}_{\mathrm{loc}}$ when $j\rightarrow+\infty$,
	then
	\begin{equation}
		\label{equ:jianhao}
		\limsup_{j\rightarrow+\infty} \sigma(\psi_{j},\Phi_{o,\max})\leq \sigma(\psi,\Phi_{o,\max}).
	\end{equation}
	We prove this remark in Section \ref{sec:proofs of remark1.2,1.3}.
\end{Remark}

We give some examples of local Zhou weights, which show that  Zhou number is a generalization of the notion of directional Lelong number.
\begin{Example}\label{exam1}
	Let $\varphi=\log\max_{1\le j\le n}|z_j|^{a_j}$ on $\Delta^n\subset\mathbb{C}^n$, where $a_j>0$ for any $j$ satisfying that  $\sum_{1\le j\le n}\frac{1}{a_j}=1$. It is clear that $\mathcal{I}(\varphi)_o$ is the maximal ideal of $\mathcal{O}_o$. For any plurisubharmonic function $\tilde\varphi$ near $o$ satisfying that $\tilde\varphi\ge\varphi$ near $o$ and $e^{-2\tilde{\varphi}}$ is not integrable near $o$, $\tilde\varphi=\varphi+O(1)$ near $o$ (see \cite{guan-remapprox}). Thus, $\varphi$ is a local Zhou weight related to $1$ near $o$.
\end{Example}

For Lelong numbers and Kiselman numbers, there are expressions in integral form (see equality \eqref{eq:240608a}). Now, we give an expression  of Zhou numbers in integral form. 

\begin{Theorem}
	\label{thm:main_value}
	Let $\Phi_{o,\max}$ be a local Zhou weight related to $|f_{0}|^{2}e^{-2\varphi_{0}}$ near $o$.
	Then for any plurisubharmonic function $\psi$ near $o$,
	\begin{equation}\nonumber
		\begin{split}
			\sigma(\psi,\Phi_{o,\max})
			&=
			\lim_{t\to+\infty}
			\frac{\int_{\{\Phi_{o,\max}<-t\}}|f_{0}|^{2}e^{-2\varphi_{0}}(-\psi)}
			{t\int_{\{\Phi_{o,\max}<-t\}}|f_{0}|^{2}e^{-2\varphi_{0}}}.
		\end{split}
	\end{equation}
\end{Theorem}

It follows from Theorem \ref{thm:main_value} and the definition of Zhou numbers that Zhou numbers are tropically multiplicative and tropically additive.		

\begin{Corollary}
	\label{coro:main}Let $\psi_{1}$ and $\psi_{2}$ be plurisubharmonic functions near $o$.
	The following statements hold
	
	$(1)$ for any $c_{1}\geq 0$ and $c_{2}\geq 0$, $\sigma(c_{1}\psi_{1}+c_{2}\psi_{2},\Phi_{o,\max})=c_{1}\sigma(\psi_{1},\Phi_{o,\max})+c_{2}\sigma(\psi_{2},\Phi_{o,\max})$;
	
	$(2)$ $\sigma(\log|f_{1}+f_{2}|,\Phi_{o,\max})\geq\min\big\{\sigma(\log|f_{1}|,\Phi_{o,\max}),\sigma(\log|f_{2}|,\Phi_{o,\max})\big\}$,
	where $f_{1}$ and $f_{2}$ are holomorphic functions near $o$;
	
	$(3)$ $\sigma(\max\{\psi_{1},\psi_{2}\},\Phi_{o,\max})=\min\big\{\sigma(\psi_{1},\Phi_{o,\max}),\sigma(\psi_{2},\Phi_{o,\max})\big\}$.
\end{Corollary}

For general maximal weights, the relative types to them may not be tropically multiplicative. Thus, the expression in integral form does not hold for all maximal weights.
\begin{Example}
	Let $\varphi=\log(|z_1|^{10}+|z_1^2-z_2^2|)$ on $\mathbb{C}^2$. Note that any plurisubharmonic function $\psi$ near $o$ with analytic singularity is a maximal weight if $\psi^{-1}(-\infty)=\{o\}$ (see \cite{Rash13}). Choosing $u_1=\log|z_1-z_2|$ and $u_2=\log|z_1+z_2|$, it is clear that $\sigma(u_1+u_2,\varphi)\ge1$, $\sigma(u_1,\varphi)<\frac{1}{2}$ and $\sigma(u_2,\varphi)<\frac{1}{2}$, then $\sigma(u_1+u_2,\varphi)\not=\sigma(u_1,\varphi)+\sigma(u_2,\varphi)$. 
\end{Example}

Denote  
$$\nu(f,\Phi_{o,\max}):=\sigma(\log|f|,\Phi_{o,\max})$$ 
for any $(f,o)\in\mathcal{O}_o$. Corollary \ref{coro:main} shows that  Zhou numbers are  well suited for the tropical structure of the cone of plurisubharmonic functions. Thus, $\nu(\cdot,\Phi_{o,\max})$ is a \emph{valuation} of $\mathcal{O}_o$ for any local Zhou weight $\Phi_{o,\max}$, and we call it \textbf{Zhou valuation}.
\begin{Corollary}
	\label{coro:main2}For any local Zhou weight $\Phi_{o,\max}$ near $o$, $\nu(\cdot,\Phi_{o,\max}):\mathcal{O}_o\rightarrow\mathbb{R}_{\ge0}$ satisfies the following:
	
	$(1)$ $\nu(fg,\Phi_{o,\max})=\nu(f,\Phi_{o,\max})+\nu(g,\Phi_{o,\max})$;
	
	$(2)$ $\nu(f+g,\Phi_{o,\max})\ge\min\big\{\nu(f,\Phi_{o,\max}),\nu(g,\Phi_{o,\max})\big\}$;
	
	$(3)$ $\nu(f,\Phi_{o,\max})=0$ for any $f(o)\not=0$.
\end{Corollary}

\subsection{Singularities of plurisubharmonic functions}
\label{sec:1.2}
In this section, we use  Zhou numbers and  Tian functions to measure the singularities of plurisubharmonic functions near $o$.

Let $G$	be a holomorphic function near $o$.  Recall  the definition of \emph{jumping number} 
$$c^G_o(\Phi_{o,\max}):=\sup\big\{c:|G|^{2}e^{-2c\Phi_{o,\max}}\text{ is integrable near }o\big\}$$ (see \cite{JON-Mus2012,JON-Mus2014}). Denote  $c_o(\Phi_{o,\max}):=c^1_o(\Phi_{o,\max})$ (see \cite{tian87,demailly2010}). 

Given a local Zhou weight $\Phi_{o,\max}$ near $o$, the following theorem shows the jumping number $c_o^{G}(\Phi_{o,\max})$ and the Zhou valuation $\nu(G,\Phi_{o,\max})$ for any $(G,o)\in\mathcal{O}_o$ are linearly controlled by each other.

\begin{Theorem}
	\label{thm:valu-jump}
	For any holomorphic function $G$ near $o$,	
	we have the following relation between the jumping number $c^G_o(\Phi_{o,\max})$ and the Zhou valuation $\sigma(\cdot,\Phi_{o,\max})$,
	\begin{equation*}
		\begin{split}
			\nu(G,\Phi_{o,\max})+c_o(\Phi_{o,\max})
			\le& c^G_o(\Phi_{o,\max})\\
			\le& \nu(G,\Phi_{o,\max})-\sigma(\log|f_0|,\Phi_{o,\max})+1+\sigma(\varphi_0,\Phi_{o,\max}).
		\end{split}
	\end{equation*}
	Especially, if $|f_0|^2e^{-2\varphi_0}\equiv1$, we have 
	$$\nu(G,\Phi_{o,\max})+1=c^G_o(\Phi_{o,\max}).$$
\end{Theorem}

Let $f$ be a holomorphic function  near $o$, and let $v$ and $\varphi_{0}$ be plurisubharmonic functions near $o$. Let us consider a function defined by jumping numbers: 
$$\mathrm{Tn}(t;f,v,\varphi_{0}):=\sup\big\{c\ge0:|f|^{2t}e^{-2cv-2\varphi_{0}}\text{ is integrable near }o\big\}$$
for any $t\in\mathbb{R}$.
The complex singularity exponent was firstly studied by Tian (see \cite{tian87,tian90,demailly2010}), and jumping number is a generalization of the notion of complex singularity exponent. Thus, we call the function $\mathrm{Tn}(t;f,v,\varphi_{0})$ \textbf{Tian function}.

The following theorem gives a relation between  Zhou numbers, Tian functions and multiplier ideal sheaves.

\begin{Theorem}
	\label{thm:multi-valua}
	Let $u,$ $v$ be two plurisubharmonic functions near $o$. Then the following  statements are equivalent:
	
	$(1)$ there exists a plurisubharmonic function $\varphi_{0}$ near $o$ ($\not\equiv-\infty$) and two sequences of numbers  $\{t_{i,j}\}_{j\in\mathbb{Z}_{\ge0}}$ ($t_{i,j}\rightarrow+\infty$ when $j\rightarrow+\infty$, $i=1,2$) such that $\lim_{j\rightarrow+\infty}\frac{t_{1,j}}{t_{2,j}}=1$ and $$\mathcal{I}(\varphi_{0}+t_{1,j}v)_o\subset\mathcal{I}(\varphi_{0}+t_{2,j}u)_o$$ for any $j$;
	
	$(2)$ for any plurisubharmonic function $\varphi_0$ near $o$ and any $t>0$, we have $$\mathcal{I}(\varphi_0+tv)_o\subset\mathcal{I}(\varphi_0+tu)_o;$$

	$(3)$ for any local Zhou weight $\Phi_{o,\max}$ near $o$, we have $$\sigma(u,\Phi_{o,\max})\le\sigma(v,\Phi_{o,\max});$$
	
	$(4)$ there exists a plurisubharmonic function $\varphi_{0}$ near $o$ ($\not\equiv-\infty$) such that 
	$$\mathrm{Tn}(t;f,v,\varphi_{0})\le\mathrm{Tn}(t;f,u,\varphi_{0})$$
for any $t\ge0$ and any holomorphic function $f$ near $o$;

	$(5)$ for any plurisubharmonic function $\varphi_{0}$ near $o$, any $t\ge0$ and any holomorphic function $f$ near $o$, we have
	$$\mathrm{Tn}(t;f,v,\varphi_{0})\le\mathrm{Tn}(t;f,u,\varphi_{0}).$$
\end{Theorem}

The above theorem shows that one can find a ``maximal" plurisubharmonic function when fixing the multiplier ideal sheaves.
\begin{Remark}
	For any negative plurisubharmonic function $\psi$ on a domain $D\subset\mathbb{C}^n$ containing $o$, let 
	$$\psi_{\mathcal{I}}:=\sup\bigg\{\psi_1\in\text{PSH}(D):\psi\le\psi_1<0\,\&\,\mathcal{I}(t\psi)_o=\mathcal{I}(t\psi_{1})_o,\,\,\forall t>0\bigg\}.$$
	Then $\psi_{\mathcal{I}}$ is a negative plurisubharmonic function on $D$ and $\mathcal{I}(t\psi)_o=\mathcal{I}(t\psi_{1})_o$ for all $t>0$. In fact, it follows from Choquet's lemma (see \cite{demailly-book}, see also Lemma \ref{lem:Choquet}) and Theorem \ref{thm:multi-valua} that there is a sequence of increasing negative plurisubharmonic functions $\{u_j\}$ satisfying that $(\sup_j u_j)^*=(\psi_{\mathcal{I}})^*$ and $\mathcal{I}(t\psi)_o=\mathcal{I}(tu_j)_o$ for any $t>0$ and any $j$. Using Remark \ref{rem:sharp_bound}, we have $\mathcal{I}(t\psi)_o=\mathcal{I}(t(\psi_{\mathcal{I}})^*)_o$ for any $t$, which shows that $(\psi_{\mathcal{I}})^*=\psi_{\mathcal{I}}$.
\end{Remark}

We call a plurisubharmonic function $\varphi$ near $o$ a \emph{tame weight} (see \cite{BFJ08}), if $\varphi$ has an isolated singularity at $o$, $e^{\varphi}$ is continuous and there exists a constant $C>0$ such that for any $t>0$ and every $(f,o)\in\mathcal{I}(t\varphi)_o$, 
$$\log|f|\le (t-C)\varphi+O(1)$$ near $o$.

For any local Zhou weights $\Phi_{o,\max}$, there exists a plurisubharmonic function $\varphi$ on a neighborhood of $o$ such that $\varphi$ has an isolated singularity at $o$, $e^{\varphi}$ is continuous and 
$$\varphi=\Phi_{o,\max}+O(1)$$ near $o$ (using Lemma \ref{l:local-glabal}, Proposition \ref{l:max2} and Proposition \ref{thm:contin}), and Theorem \ref{thm:valu-jump} shows that  there exists a constant $C>0$ such that for any $t>0$ and every $(f,o)\in\mathcal{I}(t\Phi_{o,\max})_o$, 
$$\log|f|\le (t-C)\Phi_{o,\max}+O(1)$$ near $o$. Thus, local Zhou weights are all tame weights (difference by  bounded functions).

Theorem \ref{thm:multi-valua} shows that the maximal tame weights in Theorem \ref{thm:BFJ} can be replaced by local Zhou weights, as seen in the following remark. The proof of Theorem \ref{thm:multi-valua} is purely analytic and independent of Boucksom-Favre-Jonsson's result. 

\begin{Remark}
	Theorem \ref{thm:multi-valua} shows that the following  statements are equivalent:
	
	$(1)$ there exist two sequences of numbers  $\{t_{i,j}\}_{j\in\mathbb{Z}_{\ge0}}$ ($t_{i,j}\rightarrow+\infty$ when $j\rightarrow+\infty$, $i=1,2$) such that $\lim_{j\rightarrow+\infty}\frac{t_{1,j}}{t_{2,j}}=1$ and $\mathcal{I}(t_{1,j}v)_o=\mathcal{I}(t_{2,j}u)_o$ for any $j$;
	
	$(2)$ for any plurisubharmonic function $\varphi_0$ near $o$ and any $t>0$, $\mathcal{I}(\varphi_0+tv)_o=\mathcal{I}(\varphi_0+tu)_o;$

	$(3)$ for any local Zhou weight $\Phi_{o,\max}$ near $o$, $\sigma(u,\Phi_{o,\max})=\sigma(v,\Phi_{o,\max})$;
	
	$(4)$ there exists a plurisubharmonic function $\varphi_{0}$ near $o$ ($\not\equiv-\infty$) such that 
	$\mathrm{Tn}(t;f,v,\varphi_{0})=\mathrm{Tn}(t;f,u,\varphi_{0})$
	for any $t\ge0$ and any holomorphic function $f$ near $o$;
	
	$(5)$ for any plurisubharmonic function $\varphi_{0}$ near $o$, any $t\ge0$ and any holomorphic function $f$ near $o$, we have
	$\mathrm{Tn}(t;f,v,\varphi_{0})=\mathrm{Tn}(t;f,u,\varphi_{0}).$
\end{Remark}

Recall that a plurisubharmonic function $u$ is said to have \emph{analytic singularities} near $o$ if 
$$u=c\log\sum_{1\le j\le N}|f_j|^2+O(1)$$
near $o$, where $c\in\mathbb{R}^+$ and $\{f_j\}_{1\le j\le N}$ are holomorphic functions near $o$.

Theorem \ref{thm:multi-valua} implies the following corollary, which shows that the valuations $\nu(\cdot,\Phi_{o,\max})$  characterize the division relations in $\mathcal{O}_o$.

\begin{Corollary}
	\label{coro:multiplier-valuation}
	Let $u,$ $v$ be two plurisubharmonic functions near $o$. Assume that $u$ has  analytic singularities near $o$, then the following three statements are equivalent:
	
	$(1)$ $v\le u+O(1)$ near $o$;
	
	$(2)$ there exist two sequences of numbers  $\{t_{i,j}\}_{j\in\mathbb{Z}_{\ge0}}$ ($t_{i,j}\rightarrow+\infty$ when $j\rightarrow+\infty$, $i=1,2$) such that $\lim_{j\rightarrow+\infty}\frac{t_{1,j}}{t_{2,j}}=1$ and $\mathcal{I}(t_{1,j}v)_o\subset\mathcal{I}(t_{2,j}u)_o$ for any $j$;
	
	$(3)$ $\sigma(u,\Phi_{o,\max})\le\sigma(v,\Phi_{o,\max})$ holds for any local Zhou weight $\Phi_{o,\max}$ near $o$.
	
	Especially, for any two holomorphic functions $f$ and $g$ near $o$,  the following two statements are equivalent:
	
	$(1)$ $f=gh$ near $o$, where $h$ is a holomorphic function near $o$;
	
	$(2)$  $\nu(f,\Phi_{o,\max})\ge\nu(g,\Phi_{o,\max})$ holds for any local Zhou weight $\Phi_{o,\max}$ near $o$. 
\end{Corollary}

In \cite{kim}, Kim proved that the statement $(1)$ in Theorem \ref{thm:BFJ} implies $v\le u+O(1)$ near $o$ with the assumption that $u$ has analytic singularities near $o$.
Theorem \ref{thm:BFJ} and Theorem \ref{thm:multi-valua} shows that the statement $(1)$ in Theorem \ref{thm:BFJ} is equivalent to $\sigma(u,\Phi_{o,\max})=\sigma(v,\Phi_{o,\max})$ for any local Zhou weight $\Phi_{o,\max}$ near $o$.

 Given two ideal $I_1$ and $I_2$ of $\mathcal{O}_o$, $\bar{I}_1\subset \bar{I}_2$ if and only if $\log|I_1|\le\log|I_2|+O(1)$ near $o$ (see \cite{demailly2010}), where $\bar{I}_i$ is the integral closure of $I_i$ for $i=1,2.$ According to   Corollary \ref{coro:multiplier-valuation}, one can give a valuative characterization of the integral closures of ideals in $\mathcal{O}_o$ by Zhou valuations.
 
\begin{Remark}
	 $\nu(I_1,\Phi_{o,\max})\ge\nu(I_2,\Phi_{o,\max})$ holds for any local Zhou weight $\Phi_{o,\max}$ near $o$ if and only if  $\bar{I}_1\subset \bar{I}_2$. Especially, for any $(f,o)\in\mathcal{O}_o$ and ideal $I$ of $\mathcal{O}_o$,, $f\in\bar{I}$ if and only if $\nu(f,\Phi_{o,\max})\ge \nu(I,\Phi_{o,\max})$ for any local Zhou weight $\Phi_{o,\max}$.
\end{Remark}

Let $I\not=\{0\}$ be an ideal of $\mathcal{O}_o$, and denote the set containing  all local Zhou weights related to $|I|^2$ by  $V_I$. We present an equality about jumping numbers and Zhou numbers.

\begin{Theorem}\label{thm:A}
	Let $\varphi$ be any plurisubharmonic function near $o$. Then 
	$$c_o^{I}(\varphi)=\frac{1}{\sup_{\Phi_{o,\max}\in V_I}\sigma(\varphi,\Phi_{o,\max})},$$
	where $c_o^{I}(\varphi):=\sup\big\{c\ge0:|I|^{2}e^{-2c\varphi}\text{ is integrable near }o\big\}$.
\end{Theorem}

Combining Theorem \ref{thm:A} and the strong openness property of multiplier ideal sheaves \cite{GZopen-c}, we obtain that Zhou numbers can characterize multiplier ideal sheaves as follows:
\begin{Corollary}
	For any ideal $I$ of $\mathcal{O}_o$ and plurisubharmonic function $\varphi$ near $o$,  $I\subset\mathcal{I}(\varphi)_o$ if and only if there exists $p<1$ such that $\sigma(\varphi,\Phi_{o,\max})<p$ for any local Zhou weight $\Phi_{o,\max}$ related to $|I|^2$.
	
	Especially, $\mathcal{I}(\varphi)_o=\mathcal{O}_o$ if and only if there exists $p<1$ such that $\sigma(\varphi,\Phi_{o,\max})<p$ for any local Zhou weight $\Phi_{o,\max}$ related to $1$.
\end{Corollary}

\subsection{Global Zhou weights}
In this section, we consider  global Zhou weights, and discuss some properties about the global Zhou weights.

Let $D$ be a  domain in $\mathbb{C}^n$, such that the origin $o\in D$.	Let $f_{0}=(f_{0,1},\cdots,f_{0,m})$ be a vector,
where $f_{0,1},\cdots,f_{0,m}$ are holomorphic functions near $o$.
Denote $|f_{0}|^{2}=|f_{0,1}|^{2}+\cdots+|f_{0,m}|^{2}$.

Let $\varphi_{0}$ be a plurisubharmonic function near $o$,
such that $|f_{0}|^{2}e^{-2\varphi_{0}}$ is integrable near $o$. 
\begin{Definition}
	We call a negative plurisubharmonic function $\Phi^{f_0,\varphi_0,D}_{o,\max}$ ($\Phi^{D}_{o,\max}$ for short) on $D$ \textbf{a global Zhou weight related to $|f_0|^2e^{-2\varphi_0}$} if the following statements hold:
	
	$(1)$ $|f_0|^2e^{-2\varphi_0}|z|^{2N_0}e^{-2\Phi^{D}_{o,\max}}$ is integrable near $o$ for large enough $N_0$;
	
	$(2)$ $|f_0|^2e^{-2\varphi_0-2\Phi^{D}_{o,\max}}$ is not integrable near $o$;
	
	$(3)$ for any negative plurisubharmonic function $\tilde\varphi$ on $D$ satisfying that $\tilde\varphi\ge\Phi^{D}_{o,\max}$ on $D$ and $|f_0|^2e^{-2\varphi_0-2\tilde\varphi}$ is not integrable near $o$, $\tilde\varphi=\Phi^{D}_{o,\max}$ holds on $D$.
\end{Definition}

The existence of global Zhou weights can be referred to Remark \ref{r:exists}.

For any $w\in D$, denote  
$$L_w:=\big\{u\in \mathrm{PSH}(D):u<0 \ \& \ \limsup_{z\rightarrow w}(u(z)-\log|z-w|)<+\infty\big\},$$
where $\mathrm{PSH}(D)$ denotes the set of  plurisubharmonic functions on $D$.
If $L_w\not=\emptyset$, the \emph{pluricomplex Green function} of $D$ with a pole at $w$ was defined as follows (see \cite{Blo-note,Kli85}):
$$G_D(w,\cdot):=\sup\{u:u\in L_w\}.$$
When $D$ is a bounded domain,
$$\limsup_{z\rightarrow o}\big|G(o,z)-\log|z|\big|<+\infty.$$
Then $G(o,\cdot)$ is a global Zhou weight with respect to $e^{-2(n-1)\log|z|}$ on $D$ near $o$ (by Example \ref{exam1}, Lemma \ref{l:local-glabal} and the definition of $G_D(o,\cdot)$). For any global Zhou weight $\Phi_{o,\max}^D$, we have
$$\Phi_{o,\max}^D\ge NG_D(o,\cdot)$$
on $D$ for large enough $N\gg0$ (see Lemma \ref{l:max>green}). 

Denote  
$$\tilde L_o:=\big\{u\in L_o: u\in L^{\infty}_{\mathrm{loc}}(U\backslash\{o\})\text{ for some neighborhood $U$ of $o$}\big\}.$$
If $D$ is bounded or hyperconvex,  $\tilde L_o \not=\emptyset$ (see Remark \ref{r:8.4}). 
We recall the definitions of \emph{hyperconvex domain} and \emph{strictly hyperconvex domain} as follows.

\begin{Definition}[see \cite{Ni95}]\label{def-strhpconvex}
	A domain $D\subset\mathbb{C}^n$ is said to be hyperconvex if there exists a continuous plurisubharmonic exhausted function $\varrho : D\to (-\infty,0)$.
	
	A bounded domain $D\subset\mathbb{C}^n$ is said to be strictly hyperconvex if there exists a bounded domain $\Omega$ and a function $\varrho : \Omega\to (-\infty,1)$ such that $\varrho\in C(\Omega)\cap \mathrm{PSH}(\Omega)$, $D=\{z\in\Omega : \varrho(z)<0\}$, $\varrho$ is exhaustive for $\Omega$ and for any real number $c\in [0,1]$, the open set $\{z\in\Omega : \varrho(z)<c\}$ is connected. 
\end{Definition}

It follows from Lemma \ref{l:local-glabal}, Remark \ref{r:L tildeL} and Lemma \ref{l:max-loc} that if $\tilde L_o\not=\emptyset$, there is a one-to-one correspondence between global Zhou weights $\Phi_{o,\max}^D$ and local Zhou weights $\Phi_{o,\max}$, which is given as follows:
\begin{equation}
	\nonumber
	\begin{split}
		\Phi_{o,\max}^D(z)=\sup\big\{u(z):u\in &\mathrm{PSH}^-(D),\,(f_0,o)\not\in\mathcal{I}(\varphi_o+u)_o\\
		& \ \& \ u\ge\Phi_{o,\max}+O(1)\text{ near $o$}\big\},
	\end{split}
\end{equation} 
where $\mathrm{PSH}^-(D)$ denotes the set of negative plurisubharmonic functions on $D$.

We give some examples of global Zhou weights.
\begin{Example}\label{exam}
	Let $\varphi=\log\max_{1\le j\le n}|z_j|^{a_j}$ on $\Delta^n\subset\mathbb{C}^n$, where $a_j>0$ for any $j$ satisfying that  $\sum_{1\le j\le n}\frac{1}{a_j}=1$. We know that $\varphi$ is a global Zhou weight with respect to $1$ near $o$ on $\Delta^n$, thus $\varphi$ is a local Zhou weight with respect to $1$ near $o$. In fact, for any negative plurisubharmonic function $\tilde\varphi\ge\varphi$ on $\Delta^n$ satisfying that $e^{-2\tilde\varphi}$ is not integrable near $o$, we have 
	\begin{equation}
		\nonumber
		\tilde G(t):=\inf\left\{\int_{\{\tilde\varphi<-t\}}|\tilde f|^2:\tilde f\in\mathcal{O}(\{\tilde\varphi<-t\}) \ \& \ (\tilde f-f_0,o)\in\mathcal{I}(\tilde\varphi)_o\right\}\\
		\le \int_{\{\tilde\varphi<-t\}}1.
	\end{equation} 
	By a direct calculation, we have
	\begin{equation}
		\nonumber
		G(t):=\inf\left\{\int_{\{\varphi<-t\}}| \tilde f|^2:\tilde f\in\mathcal{O}(\{\varphi<-t\})\ \& \ (\tilde f-f_0,o)\in\mathcal{I}(\varphi)_o\right\}=\int_{\{\varphi<-t\}}1.
	\end{equation}
	Note that $\tilde G(0)=G(0)$, $\tilde G(-\log r)$ is concave on $(0,1)$ (see \cite{guan-effect}) and $G(-\log r)$ is linear on $(0,1)$ (see \cite{guan-remapprox,GY-concavity4}). Then $\tilde G(t)=G(t)$ and $\int_{\{\tilde\varphi<-t\}}1=\int_{\{\varphi<-t\}}1$ for any $t\ge0$, which implies that $\tilde\varphi\equiv\varphi$. Thus,  $\varphi$ is a global Zhou weight with respect to $1$ near $o$ on $\Delta^n$.
\end{Example}

Now, we give two properties of global Zhou weights.

\begin{Proposition}\label{l:max2}Assume that $\tilde L_o\not=\emptyset$.
	Let  $\Phi^{D}_{o,\max}$ be a global Zhou weight related to $|f_0|^2e^{-2\varphi_0}$ on $D$, then 
	$$\Phi^{D}_{o,\max}\in \mathrm{PSH}(D)\cap L_{\mathrm{loc}}^{\infty}(D\backslash\{o\})$$
	and 
	\[\big(dd^c\Phi^{D}_{o,\max}\big)^n=0 \ \text{on} \ D\backslash\{o\}.\]
\end{Proposition}

The following proposition gives the continuity of  $\Phi_{o,\max}^D$.

\begin{Proposition}\label{thm:contin}
	Assume that $D$ is a bounded hyperconvex domain.
	Let  $\Phi^{D}_{o,\max}$ be a global Zhou weight related to $|f_0|^2e^{-2\varphi_0}$ on $D$, then $e^{\Phi^{D}_{o,\max}}$ is continuous on $D$ and $\Phi^{D}_{o,\max}(z)\rightarrow0$ when $z\rightarrow\partial D$.
\end{Proposition}

Next, we discuss some approximations of global Zhou weights.

For any $m\in\mathbb{N}_+$, we define two compact subsets of $\mathcal{O}(D)$ as follows:
\[\mathscr{E}_m(D):=\big\{f\in\mathcal{O}(D) : \sup_{z\in D}|f(z)|\le 1, (f,o)\in\mathcal{I}(m\Phi^D_{o,\max})_o\big\},\]
\[\mathscr{A}^2_m(D):=\big\{f\in\mathcal{O}(D) : \|f\|_D\le 1, (f,o)\in\mathcal{I}(m\Phi^D_{o,\max})_o\big\},\] 
where $\|f\|_D^2:=\int_D|f|^2$. We also define two plurisubharmonic functions $\phi_m$ and $\varphi_m$ for any $m$ by:
\begin{equation}\nonumber
	\phi_m(z):=\sup_{f\in\mathscr{E}_m(D)} \frac{1}{m}\log |f(z)|, \ \forall z\in D,
\end{equation}

\begin{equation}\nonumber
	\varphi_m(z):=\sup_{f\in\mathscr{A}_m^2(D)}\frac{1}{m}\log|f(z)|, \ \forall z\in D.
\end{equation}

We obtain an approximation theorem for global Zhou weights $\Phi_{o,\max}^D$.

\begin{Theorem}\label{thm-approximation}
	If $D$ is a bounded strictly hyperconvex domain, then
	
	$(1)$
	\[\lim_{m\to\infty}\phi_m(z)=\lim_{m\to\infty}\varphi_m(z)=\Phi^D_{o,\max}(z),\ \forall z\in D.\]
	
	$(2)$ There exists a constant $\mathsf{C}$ independent of $m$, such that for any $m\in\mathbb{N}_+$,
	\[1-\frac{\mathsf{C}}{m}\le\sigma(\phi_m,\Phi^D_{o,\max})\le 1,\]
	and
	\[1-\frac{\mathsf{C}}{m}\le\sigma(\varphi_m,\Phi^D_{o,\max})\le 1.\]
\end{Theorem}

Theorem \ref{thm-approximation} shows that we can reconstruct the global Zhou weights $\Phi^D_{o,\max}$ on $D$ by the data of the multiplier ideals $\mathcal{I}(m\Phi^D_{o,\max})_o$. A similar method to approximate the pluricomplex Green functions by $\phi_{m}$ can be referred to \cite{Ni95}.

We give some corollaries of Theorem \ref{thm-approximation} below.
\begin{Corollary}\label{cor-approximation}
	If $D$ is a bounded strictly hyperconvex domain, and $\Phi^D_{o,\max}$ is a global Zhou weight related to some $|f_0|^2e^{-2\varphi_0}$   on $D$ near $o$, then for any $w\in D$, we have
	\[\Phi^D_{o,\max}(w)=\sup\left\{\frac{\log|f(w)|}{\sigma(\log|f|, \Phi^D_{o,\max})} : f\in \mathcal{O}(D), \ \sup_D|f|\le 1, \ f(o)=0, \  f\not\equiv 0\right\}.\]
\end{Corollary}

Note that for any non-constant holomorphic function $f$ with $f(o)=0$ near $o$, we have $\sigma(\log|f|,\Phi^D_{o,\max})>0$. In fact, since $\Phi_{o,\max}^D(z)\ge N\log|z|+O(1)$ near $o$ for some $N>0$, we only need to show that there exists $N'>0$ such that $\log|f|\le N'\log|z|+O(1)$ near $o$. Since $f(o)=0$ and $f\not\equiv 0$, the Lelong number $\nu(\log|f|,o)\in (0,+\infty)$, which implies the existence of $N'$.

\begin{Corollary}\label{cor-approximation2}
	Let $D$ be a bounded strictly hyperconvex domain, and $\Phi^D_{o,\max}$ be a global Zhou weight related to some $|f_0|^2e^{-2\varphi_0}$   on $D$ near $o$. Then for any $z\in D$, we have
	\[\Phi_{o,\max}^D(z)=\sup\big\{\phi(z) : \phi\in\mathrm{PSH}^-(D), \ \phi\sim_{\mathcal{I}}\Phi^D_{o,\max} \ \text{at} \ o\big\},\]
	where we write `\emph{$\phi\sim_{\mathcal{I}}\Phi^D_{o,\max}$ at $o$}' if there exists real numbers $\alpha\ge \beta$ independent of $m$, such that for $m\gg1$,
	\[\mathcal{I}\big((m+\alpha)\Phi_{o,\max}^D\big)_o\subset\mathcal{I}(m\phi)_o\subset\mathcal{I}\big((m+\beta)\Phi_{o,\max}^D\big)_o.\]
	In particular, we have
	\[\Phi_{o,\max}^D(z)=\sup\big\{\phi(z) : \phi\in\mathrm{PSH}^-(D), \ \mathcal{I}(m\phi)_o=\mathcal{I}\big(m\Phi_{o,\max}^D\big)_o, \ \forall m\in\mathbb{N}_+\big\}.\]
\end{Corollary}

One can also obtain Corollary \ref{cor-approximation2} by Theorem \ref{thm:multi-valua}, which will not be shown in this paper.

Note that the pluricomplex Green function $G_D(o,\cdot)$ is a global Zhou weight related to $e^{-2(n-1)\log|z|}$ (with $f_0\equiv 1$ and $\varphi_0=(n-1)\log|z|$)  on $D$ near $o$. Then Theorem \ref{thm-approximation} recovers the following result in \cite{Ni95}.
\begin{Corollary}[\cite{Ni95}]\label{coro:ni95}
	Let $D$ be a bounded strictly hyperconvex domain containing $o$, then
	\[\lim_{m\to\infty}g_m(z)=G_D(o,z), \ \forall z\in D,\]
	where
	\[g_m(z)=\sup_{f\in \mathcal{E}_m(D)}\frac{1}{m}\log|f(z)|, \ \forall z\in D,\]
	and
	\[\mathcal{E}_m(D):=\left\{f\in\mathcal{O}(D) : \sup_{w\in D}|f(w)|\le 1, \ D^{(\nu)}f(o)=0, \ \forall \nu\in\mathbb{N}^n \ \text{with} \ |\nu|\le m-1\right\}.\]
\end{Corollary}

\section{Preliminaries}

In this section, we give some preliminaries.

\subsection{Recall: $L^{2}$ method}

In this section, we recall the following lemma,
whose various forms already appeared in \cite{guan-zhou13p,guan-zhou13ap,guan-effect} etc.:

\begin{Lemma} [see \cite{guan-effect}]\label{lem:GZ_sharp}
	Let $B\in(0,+\infty)$ and $t_{0}\geq T$ be arbitrarily given.
	Let $D\subset\subset\mathbb{C}^{n}$ be a pseudoconvex domain.
	Let $\psi<-T$ be a plurisubharmonic function
	on $D$.
	Let $\varphi$ be a plurisubharmonic function on $D$.
	Let $F$ be a holomorphic function on $\{\psi<-t_{0}\}$,
	such that
	\begin{equation*}
		\int_{K\cap\{\psi<-t_{0}\}}|F|^{2}<+\infty
	\end{equation*}
	for any compact subset $K$ of $D$,
	and
	\begin{equation*}
		\int_{D}\frac{1}{B}\mathbb{I}_{\{-t_{0}-B<\psi<-t_{0}\}}|F|^{2}e^{-\varphi}<+\infty.
	\end{equation*}
	Then there exists a
	holomorphic function $\tilde{F}$ on $D$, such that,
	\begin{equation*}
		\begin{split}
			&\int_{D}|\tilde{F}_{t_{0}}-(1-b(\psi))F|^{2}e^{-\varphi+v(\psi)}\\
			\leq& \left(e^{-T}-e^{-t_{0}-B}\right)\int_{D}\frac{1}{B}\mathbb{I}_{\{-t_{0}-B<\psi<-t_{0}\}}|F|^{2}e^{-\varphi}
		\end{split}
	\end{equation*}
	where
	$b(t)=\int_{-\infty}^{t}\frac{1}{B}\mathbb{I}_{\{-t_{0}-B< s<-t_{0}\}}ds$,
	$v(t)=\int_{-t_0}^{t}b(s)ds-t_0$.
\end{Lemma}

It is clear that $\mathbb{I}_{(-t_{0},+\infty)}\leq b(t)\leq\mathbb{I}_{(-t_{0}-B,+\infty)}$ and $\max\{t,-t_{0}-B\}\leq v(t) \leq\max\{t,-t_{0}\}$.

Let $\varphi$ and $\varphi_{0}$ be plurisubharmonic functions on $\Delta^{n}\subset\mathbb{C}^{n}$,
and $f_{0}$ be a holomorphic function on $\Delta^{n}$. 

\begin{Lemma}[see \cite{GZopen-effect}]
	\label{lem:JM}
	Assume that $|f_{0}|^{2}e^{-2(\varphi+\varphi_{0})}$ is not integrable near $o$,
	and $|f_{0}|^{2}e^{-2\varphi_{0}}$ is integrable near $o$.
	Then for any small enough neighborhood $U\subset\subset\Delta^{n}$ of $o$,
	there exists $C>0$ such that
	$$e^{2t}\int_{\{\varphi<-t\}\cap U}|f_{0}|^{2}e^{-2\varphi_{0}}>C$$
	for any $t\ge0$.
\end{Lemma}

\begin{proof}
	Note that
	for small enough $U$,
	
	(1) $\sup_{U}\varphi\leq 0$;
	
	(2) $\sup_{U}\varphi_{0}\leq 0$;
	
	(3) $\int_{U}|f_{0}|^{2}e^{-2\varphi_{0}}<+\infty$
	\
	implies that
	\begin{equation*}
		\int_{U}|f_{0}|^{2}\leq e^{2\sup_{U}\varphi_{0}}\int_{U}|f_{0}|^{2}e^{-2\varphi_{0}}<+\infty,
	\end{equation*}
	and
	\begin{equation*}
		\int_{U}\mathbb{I}_{\{-2t_{0}-1<2\varphi<-2t_{0}\}}|f_{0}|^{2}e^{-2\varphi_{0}-2\varphi}
		\leq\int_{U}|f_{0}|^{2}e^{-2\varphi_{0}}e^{2t_{0}+1} <+\infty
	\end{equation*}
	hold for any $t_{0}\geq 0$.
	Then it follows from Lemma \ref{lem:GZ_sharp} $(T\sim 0,$ $t_{0}\sim 2t_{0},$
	$B\sim 1$, $\psi\sim2\varphi,$ $\varphi\sim\varphi+\varphi_{0}$, here `$\sim$' means that the former is replaced by the latter, and the notation will be used throughout the paper$)$ that
	there exists a
	holomorphic function $\tilde{F}_{t_{0}}$ on $U$ such that
	\begin{equation}
		\label{equ:IAS2018a}
		\begin{split}
			&\int_{U}|\tilde{F}_{t_{0}}-\big(1-b(2\varphi)\big)f_{0}|^{2}e^{-2(\varphi+\varphi_{0})+v(2\varphi)}\\
			\leq& \left(1-e^{-(2t_{0}+1)}\right)\int_{U}\mathbb{I}_{\{-2t_{0}-1<2\varphi<-2t_{0}\}}|f_{0}|^{2}e^{-2(\varphi+\varphi_{0})}\\
			\leq& e^{2t_{0}+1}\int_{U}\mathbb{I}_{\{-2t_{0}-1<2\varphi<-2t_{0}\}}|f_{0}|^{2}e^{-2\varphi_{0}}\\
			\leq& e^{2t_{0}+1}\int_{U}\mathbb{I}_{\{2\varphi<-2t_{0}\}}|f_{0}|^{2}e^{-2\varphi_{0}}\\
			=& e^{2t_{0}+1}\int_{U}\mathbb{I}_{\{\varphi<-t_{0}\}}|f_{0}|^{2}e^{-2\varphi_{0}},
		\end{split}
	\end{equation}
	where
	$b(t)=\int_{0}^{t}\mathbb{I}_{\{-t_{0}-1< s<-t_{0}\}}ds$,
	$v(t)=\int_{0}^{t}b(s)ds$.

	Inequality \eqref{equ:IAS2018a} shows that
	\[\int_{U}|\tilde{F}_{t_{0}}-\big(1-b(2\varphi)\big)f_{0}|^{2}e^{-2(\varphi+\varphi_{0})+v(2\varphi)}<+\infty.\]
	Note that $v(2\varphi)|_{U}\geq -(2t_{0}+1)$,
	then
	\[\int_{U}|\tilde{F}_{t_{0}}-\big(1-b(2\varphi)\big)f_{0}|^{2}e^{-2(\varphi+\varphi_{0})}<+\infty.\]
	Note that $(1-b(2\varphi))=1$ on $\{2\varphi<-2t_{0}-1\}$,
	then $(\tilde{F}_{t_{0}}-f_{0},o)\in\mathcal{I}(\varphi+\varphi_{0})_{o}$.
	As $(f_{0},o)\not\in\mathcal{I}(\varphi+\varphi_{0})_{o}$,
	$(\tilde{F}_{t_{0}},o)\not\in\mathcal{I}(\varphi+\varphi_{0})_{o}$,
	which implies that
	\begin{equation}
		\label{equ:IAS2018d}
		\liminf_{t_{0}\to+\infty}\int_{U}|\tilde{F}_{t_{0}}|^{2}>0.
	\end{equation}
	Inequality \eqref{equ:IAS2018d} could be proved by contradiction:
	if not, there exists a subsequence of $\{\tilde{F}_{t_{0}}\}_{t_{0}}$ compactly converging to $0$ $(t_{0}\to+\infty)$,
	which implies that there exists a subsequence of $\tilde{F}_{t_{0}}-f_{0}$ compactly converging to $-f_{0}$,
	which contradicts the closedness of coherent sheaves (see \cite{Grau-Rem84})
	by $(\tilde{F}_{t_{0}}-f_{0},o)\in\mathcal{I}(\varphi+\varphi_{0})_{o}$
	and $(-f_{0},o)\not\in\mathcal{I}(\varphi+\varphi_{0})_{o}$.
	
	Note that $\mathrm{Supp}\big(1-b(t)\big)\subset (-\infty,-2t_{0}]$, $v(t)\geq t$,
	\begin{equation}
		\label{equ:IAS2018b}
		\begin{split}
			&\left(\int_{U}\big|\tilde{F}_{t_{0}}-\big(1-b(2\varphi)\big)f_{0}\big|^{2}e^{-2(\varphi+\varphi_{0})+v(2\varphi)}\right)^{1/2}\\
			\geq&
			\left(\int_{U}\big|\tilde{F}_{t_{0}}-\big(1-b(2\varphi)\big)f_{0}\big|^{2}e^{-2\varphi_{0}}\right)^{1/2}
			\\\geq&
			\left(\int_{U}|\tilde{F}_{t_{0}}|^{2}e^{-2\varphi_{0}}\right)^{1/2}-\left(\int_{U}\big|\big(1-b(2\varphi)\big)f_{0}\big|^{2}e^{-2\varphi_{0}}\right)^{1/2}
			\\\geq&
			\left(\int_{U}|\tilde{F}_{t_{0}}|^{2}e^{-2\varphi_{0}}\right)^{1/2}-\left(\int_{U\cap\{2\varphi<-2t_{0}\}}|f_{0}|^{2}e^{-2\varphi_{0}}\right)^{1/2},
			\\\geq&
			\left(\int_{U}|\tilde{F}_{t_{0}}|^{2}\right)^{1/2}-\left(\int_{U\cap\{2\varphi<-2t_{0}\}}|f_{0}|^{2}e^{-2\varphi_{0}}\right)^{1/2},
		\end{split}
	\end{equation}
	and
	\begin{equation}
		\label{equ:IAS2018c}
		\begin{split}
			\lim_{t_{0}\to+\infty}\int_{U\cap\{2\varphi<-2t_{0}\}}|f_{0}|^{2}e^{-2\varphi_{0}}=0,
		\end{split}
	\end{equation}
	then
	the combination of inequality \eqref{equ:IAS2018a}, equality \eqref{equ:IAS2018b},
	inequality \eqref{equ:IAS2018c} and inequality \eqref{equ:IAS2018d}
	proves the present Lemma.
\end{proof}

\subsection{Some asymptotic properties related to integrability}

Let $\varphi$ and $\varphi_{0}$ be plurisubharmonic functions on $\Delta^{n}\subset\mathbb{C}^{n}$
such that $|f_{0}|^{2}e^{-2\varphi_{0}}$ is integrable near $o$.

\begin{Lemma}
	\label{lem:jump_asyp_A}
	Assume that $(f_{0},o)\not\in\mathcal{I}(\varphi+\varphi_{0})_{o}$.
	Then for any neighborhood $U\subset\subset\Delta^{n}$ of $o$,
	$$\limsup_{t\to+\infty}\frac{-\log \int_{\{\varphi<-t\}\cap U}|f_{0}|^{2}e^{-2\varphi_{0}}}{2t}\leq 1.$$
\end{Lemma}
\begin{proof}
	As $(f_{0},o)\not\in\mathcal{I}(\varphi+\varphi_{0})_{o}$, we have 
	$$\liminf_{t\rightarrow+\infty}e^{2t}\int_{\{\varphi<-t\}\cap U}|f_0|^2e^{-2\varphi_0}>0$$
	(see Lemma \ref{lem:JM}), which implies that $$\limsup_{t\to+\infty}\frac{-\log \int_{\{\varphi<-t\}\cap U}|f_{0}|^{2}e^{-2\varphi_{0}}}{2t}\leq 1.$$
\end{proof}

Assume that $\varphi\geq N\log|z|+O(1)$ for large enough $N\gg0$.

\begin{Lemma}
	Assume that $(f_{0},o)\in\mathcal{I}(\varphi+\varphi_{0})_{o}$.
	Then for any neighborhood $U\subset\subset\Delta^{n}$ of $o$,
	$$\lim_{t\to+\infty}e^{2t}\int_{\{\varphi<-t\}\cap U}|f_{0}|^{2}e^{-2\varphi_{0}}=0.$$
\end{Lemma}
\begin{proof}As $\varphi\geq N\log|z|+O(1)$ for large enough $N\gg0$ and $(f_0,o)\in\mathcal{I}(\varphi+\varphi_0)_o$, we have  
	$$\limsup_{t\to+\infty}\int_{\{\varphi<-t\}\cap U}|f_{0}|^{2}e^{-2\varphi_{0}-2\varphi}=0,$$
	which implies that
	\begin{equation}
		\nonumber\begin{split}
			&\limsup_{t\to+\infty}e^{2t}\int_{\{\varphi<-t\}\cap U}|f_{0}|^{2}e^{-2\varphi_{0}}\\
			\le& \limsup_{t\to+\infty}\int_{\{\varphi<-t\}\cap U}|f_{0}|^{2}e^{-2\varphi_{0}-2\varphi}\\
			=&0.
		\end{split}
	\end{equation}
\end{proof}

The strong openness strong property (\cite{GZopen-c}, see also Theorem \ref{thm:SOC}) shows that $(f_{0},o)\in\mathcal{I}(\varphi+\varphi_{0})_{o}$ implies that there exists small enough $\varepsilon>0$ such that $(f_{0},o)\in\mathcal{I}\big((1+\varepsilon)\varphi+\varphi_{0}\big)_{o}$.

\begin{Lemma}
	\label{lem:jump_asyp_B}
	Assume that $(f_{0},o)\in\mathcal{I}(\varphi+\varphi_{0})_{o}$.
	Then for some neighborhood $U\subset\subset\Delta^{n}$ of $o$
	there exists small enough $\varepsilon>0$ such that
	$$\lim_{t\to+\infty}e^{2t(1+\varepsilon)}\int_{\{\varphi<-t\}\cap U}|f_{0}|^{2}e^{-2\varphi_{0}}=0.$$
\end{Lemma}
Lemma \ref{lem:jump_asyp_B} implies the following result.
\begin{Lemma}\label{l:918}
	Assume that $(f_{0},o)\in\mathcal{I}(\varphi+\varphi_{0})_{o}$.
	Then for some neighborhood $U$ of $o$,
	$$\liminf_{t\to+\infty}\frac{-\log \int_{\{\varphi<-t\}\cap U}|f_{0}|^{2}e^{-2\varphi_{0}}}{2t}>1.$$
\end{Lemma}

It follows from Lemma \ref{lem:jump_asyp_A} and Lemma \ref{l:918} that the following lemma holds.
\begin{Lemma}
	\label{lem:jump_asyp_C}
	Assume that $(f_{0},o)\not\in\mathcal{I}(\varphi+\varphi_{0})_{o}$,
	and $(f_{0},o)\in\mathcal{I}\big((1-\varepsilon)\varphi+\varphi_{0}\big)_{o}$ for any $\varepsilon\in(0,1)$.
	Then
	for some neighborhood $U\subset\subset\Delta^{n}$ of $o$,
	$$\lim_{t\to+\infty}\frac{-\log\int_{\{\varphi<-t\}\cap U}|f_{0}|^{2}e^{-2\varphi_{0}}}{2t}=1.$$
\end{Lemma}

\subsection{Strong openness property and related results}

Recall that the \emph{strong openness property} of the multiplier ideal sheaves is that
$$\mathcal{I}(u)=\mathcal{I}_{+}(u):=\bigcup_{p>1}\mathcal{I}(pu)$$
for any plurisubharmonic function $u$, which was
conjectured by Demailly  in \cite{demailly-note2000}  and \cite{demailly2010} (called \emph{strong openness conjecture}).
When $\mathcal{I}(u)=\mathcal{O}$, Demailly's strong openness conjecture degenerates to the \emph{openness conjecture}, which was posed by Demailly-Koll\'{a}r in \cite{D-K01}.

Favre-Jonsson \cite{FM05j} proved the 2-dimensional case of the openness conjecture by using valuation theory. By studying asymptotic jumping numbers for graded sequences of ideals in valuation theory, Jonsson-Musta\c{t}\u{a} \cite{JON-Mus2012} proved the 2-dimensional case of the strong openness conjecture. For the higher dimension, the connection between valuation theory and the openness conjecture has been highlighted by the higher-dimensional framework in \cite{BFJ08}.

In \cite{berndtsson13}, Berndtsson proved the openness conjecture by using a complex variant of the Brunn-Minkowski inequality. After that, Guan-Zhou \cite{GZopen-c} proved the strong openness conjecture by movably using the famous  Ohsawa-Takegoshi $L^2$ extension theorem \cite{OT87}.

\begin{Theorem}[\cite{GZopen-c}, see also \cite{Lempert14} and \cite{Hiep2014}]
	\label{thm:SOC}
	Let $(\varphi_{j})$ be a sequence of plurisubharmonic functions, which is increasingly convergent to a plurisubharmonic function $\varphi$. Then $\mathcal{I}(\varphi)=\bigcup_{j}\mathcal{I}(\varphi_{j})$.
\end{Theorem}

Recently, Xu \cite{xu2019} completed the algebraic approach of solving the openness conjecture, which was conjectured by Jonsson-Musta\c{t}\u{a} \cite{JON-Mus2012,JON-Mus2014}.

Let $\{\phi_{m}\}_{m\in\mathbb{N}^{+}}$ be a sequence of negative plurisubharmonic functions on $\Delta^{n}$,
which is convergent to a negative Lebesgue measurable function $\phi$ on $\Delta^{n}$ in Lebesgue measure.

Let $f$ be a holomorphic function near $o$, and let $I$ be an ideal of $\mathcal{O}_{o}$.
We denote 
\[C_{f,I}(U):=\inf\left\{\int_{U} |\tilde{f}|^{2} :(\tilde{f}-f,o)\in I \ \& \ \tilde{f}\in\mathcal{O}(U)\right\},\]
where $U\subseteq \Delta^{n}$ is a domain with $o\in U$.
Especially, if $I=\mathcal{I}(\phi)_{o}$, we denote $C_{f,\phi}(U):=C_{f,I}(U)$.

In \cite{GZopen-effect}, Guan-Zhou presented the following lower semicontinuity property of
plurisubharmonic functions with a multiplier.

\begin{Proposition}[\cite{GZopen-effect}]
	\label{p:effect_GZ}
	Let $f$ be a holomorphic function near $o$.
	Assume that for any small enough neighborhood $U$ of $o$,
	the pairs $(f,\phi_{m})$ $(m\in\mathbb{N}^+)$ satisfies
	\begin{equation}
		\label{equ:A}
		\inf_{m}C_{f,\varphi_0+\phi_{m}}(U)>0.
	\end{equation}
	Then $|f_{0}|^{2}e^{-2\varphi_{0}}e^{-2\phi}$ is not integrable near $o$.
\end{Proposition}

The Noetherian property of multiplier ideal sheaves (see \cite{demailly-book}) shows that
\begin{Remark}
	\label{rem:effect_GZ}
	Assume that
	
	$(1)$ $\phi_{m+1}\geq\phi_{m}$ holds for any $m$;
	
	$(2)$ $|f_{0}|^{2}e^{-2\varphi_{0}}e^{-2\phi_{m}}$ is not integrable near $o$ for any $m$.
	
	Then inequality \eqref{equ:A} holds.
\end{Remark}

\subsection{Proofs of Remark \ref{rem:max_existence} and Remark \ref{rem:sharp_bound}}\label{sec:proofs of remark1.2,1.3}

In this section, we give the proofs of Remark \ref{rem:max_existence} and Remark \ref{rem:sharp_bound}.

Firstly, we recall two useful results. 

\begin{Lemma}[\emph{Choquet's lemma}, see \cite{demailly-book}]
	\label{lem:Choquet} Every family $(u_{\alpha})$ of uppersemicontinuous functions has a countable subfamily $(v_{j})=(u_{\alpha(j)})$,
	such that its upper envelope $v=\sup_{j}v_{j}$ satisfies $v\leq u\leq u^{*} = v^{*}$,
	where $u=\sup_{\alpha}u_{\alpha}$, $u^{*}(z):=\lim_{\varepsilon\to0}\sup_{\mathbb{B}^{n}(z,\varepsilon)}u$
	and $v^{*}(z):=\lim_{\varepsilon\to0}\sup_{\mathbb{B}^{n}(z,\varepsilon)}v$ are the regularizations of $u$ and $v$.
\end{Lemma}

\begin{Proposition}[see Proposition (4.24) in \cite{demailly-book}]
	\label{pro:Demailly}
	If all $(u_{\alpha})$ are subharmonic, the upper regularization $u^{*}$ is subharmonic
	and equals almost everywhere to $u$.
\end{Proposition}

Then, we prove Remark \ref{rem:max_existence}.
\begin{proof}
	[Proof of Remark \ref{rem:max_existence}]
	Theorem \ref{thm:SOC} shows that there exists $\varepsilon>0$,
	such that
	$$|f_{0}|^{2}e^{-2\varphi_{0}}|z|^{2N_{0}}e^{-2(1+\varepsilon)\varphi}$$
	is integrable near $o$.
	We have
	\begin{equation}
		\nonumber\begin{split}
			&\int_U|f_{0}|^{2}e^{-2\varphi_{0}}e^{-2\varphi}-|f_{0}|^{2}e^{-2\varphi_{0}}e^{-2\max\big\{\varphi,\frac{N_{0}}{\varepsilon}\log|z|\big\}}\\
			\le&\int_{U\cap\big\{\frac{N_{0}}{\varepsilon}\log|z|\ge\varphi\big\}}|f_{0}|^{2}e^{-2\varphi_{0}}e^{-2\varphi}\\
			\le&\int_{U\cap\big\{\frac{N_{0}}{\varepsilon}\log|z|\ge\varphi\big\}}|f_{0}|^{2}e^{-2\varphi_{0}}|z|^{2N_0}e^{-2(1+\varepsilon)\varphi}\\
			<&+\infty,
		\end{split}
	\end{equation}
	where $U$ is a neighborhood of 
	$o$.
	Then 	it suffices to consider that there exists $N$ large enough such that 
	$$N\log|z|\leq \varphi$$
	near $o$.
	
	Let $V$ be a small neighborhood of $o$ such that $\mathrm{Supp}\big((\mathcal{O}/\mathcal{I}(\varphi))|_{V}\big)=\{o\}$
	and let $(u_{\alpha})_{\alpha}$ be the negative plurisubharmonic functions on $V$
	such that $u_{\alpha}\geq \varphi+O(1)$ near $o$ and $|f_0|^2e^{-2\varphi_0}e^{-2u_{\alpha}}$ is not integrable near $o$.
	
	\emph{Zorn's Lemma} shows that
	there exists $\Gamma$ which is the maximal set such that for any $\alpha,\alpha'\in\Gamma$,
	$u_{\alpha}\leq u_{\alpha'}+O(1)$ or $u_{\alpha'}\leq u_{\alpha}+O(1)$ holds near $o$,
	where $(u_{\alpha})$ are negative plurisubharmonic functions on $V$.
	
	Let $u(z):=\sup_{\alpha\in\Gamma}u_{\alpha}(z)$ on $V$,
	and let $u^{*}(z)=\lim_{\varepsilon\to0}\sup_{\mathbb{B}^{n}(z,\varepsilon)}u$.
	Lemma \ref{lem:Choquet} shows that there exists subsequence $(v_{j})$ of $(u_{\alpha})$ such that
	$(\max_{j}v_{j})^{*}=u^{*}$.
	Moreover one can choose $v_{j}(:=\sup_{j'\leq j}v_{j})$ increasing with respect to $j$
	such that $|f_0|^2e^{-2\varphi_0}e^{-2v_j}$ is not integrable near $o$.
	
	Proposition \ref{pro:Demailly} shows that
	$(v_{j})$ is convergent to $v^{*}$ with respect to $j$ almost everywhere with respect to Lebesgue measure,
	and $v^{*}$ is a plurisubharmonic function on $V$.
	Proposition \ref{p:effect_GZ} (and Remark \ref{rem:effect_GZ})
	shows that $|f_0|^2e^{-2\varphi_0}e^{-2v^*}$ is not integrable near $o$.
	In fact, the definition of $u$ implies that $u=v^{*}=u^{*}$.
	
	In the following part,
	we prove that $v^{*}$ is a local Zhou weight  related to $|f_0|^2e^{-2\varphi_0}$ near $o$ by contradiction.
	If not, then there exists a plurisubharmonic function $\tilde{v}$ near $o$ such that $\tilde{v}\geq v^{*}$,
	$|f_0|^2e^{-2\varphi_0}e^{-2\tilde v}$ is not integrable near $o$, and
	\[\limsup_{z\to o}\big(\tilde{v}(z)-v^{*}(z)\big)=+\infty.\]
	
	Note that $v^{*}\geq\varphi\geq N\log|z|$.
	Then for small ball $\mathbb{B}(o,\varepsilon)$,
	there exists $M\ll 0$ such that $\tilde{v}+M< N\log|z|<v^{*}$ near the boundary of $\mathbb{B}(o,\varepsilon)$,
	which implies that $\max\{\tilde{v}+M,v^{*}\}=v^{*}$ near the boundary of $\mathbb{B}(o,\varepsilon)$.
	Let 
	\begin{equation*}
		\tilde{\varphi}:=\left\{
		\begin{array}{ll}
			\max\{\tilde{v}+M,v^{*}\} & \ \text{on} \ \mathbb{B}(o,\varepsilon), \\
			v^{*} & \ \text{on} \ V\setminus\mathbb{B}(o,\varepsilon).
		\end{array}
		\right.
	\end{equation*}
	$\tilde v\ge v^*$ implies that $\tilde\varphi=\tilde v+O(1)$ near $o$. Then $\tilde{\varphi}$ is a plurisubharmonic function on $V$ such that
	\[\limsup_{z\to o}(\tilde{\varphi}(z)-u^{*}(z))=\limsup_{z\to o}(\tilde{\varphi}(z)-v^{*}(z))\geq \limsup_{z\to o}(\tilde{v}(z)+M-v^{*}(z))=+\infty,\]
	and $|f_0|^2e^{-2\varphi_0}e^{-2\tilde\varphi}$ is not integrable near $o$, which contradicts the definition of $u^{*}$.
	
	This proves Remark \ref{rem:max_existence}.
\end{proof}

Finally, we prove Remark \ref{rem:sharp_bound}.
\begin{proof}
	[Proof of Remark \ref{rem:sharp_bound}]
	For $U\ni o$ small enough, there exists a negative $\Phi_{o,\max}$ on $U$. The uniqueness of $\Phi_{o,\max}^U$ is just follows from inequality \eqref{equ:xiaomage}, then it suffices to prove the existence.
	
	For any negative plurisubharmonic functions $\varphi_{1}\geq\Phi_{o,\max}$ and $\varphi_{2}\geq\Phi_{o,\max}$ on $U$,
	such that $|f_{0}|^{2}e^{-2\varphi_{0}}e^{-2\varphi_{1}}$
	and $|f_{0}|^{2}e^{-2\varphi_{0}}e^{-2\varphi_{2}}$ are not integrable near $o$,
	it follows from the definition of $\Phi_{o,\max}$ that
	$\varphi_{1}=\Phi_{o,\max}+O(1)$ and $\varphi_{2}=\Phi_{o,\max}+O(1)$ near $o$,
	which implies that
	$\max\{\varphi_{1},\varphi_{2}\}=\Phi_{o,\max}+O(1)$ near $o$.
	
	We consider the upper-envelop of 
	\[\sup\big\{\varphi(z) : \varphi\in\mathrm{PSH}^-(U) \ \& \ (f_0,o)\notin\mathcal{I}(\varphi_0+\varphi)_o \ \& \ \varphi\geq\Phi_{o,\max}+O(1)\big\},\]
	denoted by $\Phi^U_{o,\max}$.
	
	Lemma \ref{lem:Choquet} shows that there exist negative plurisubharmonic functions $\{\varphi_{j}\}_{j}$ on $U$
	such that
	
	(1) $|f_{0}|^{2}e^{-2\varphi_{0}}e^{-2\varphi_{j}}$ is not integrable near $o$;
	
	(2) $\varphi_{j}\geq\Phi_{o,\max}+O(1)$;
	
	(3) the upper-envelop $(\sup_{j}\{\varphi_{j}\})^{*}=\Phi^U_{o,\max}$.
	
	Note that for any $t_{0}>0$, $\sup_{j<t_{0}}\{\varphi_{j}\}=\Phi_{o,\max}+O(1)$,
	and $\sup_{j<t_{0}}\{\varphi_{j}\}$ is increasingly convergent to $\Phi^U_{o,\max}$ almost everywhere on $U$
	when $t_{0}\to+\infty$,
	then it follows from Proposition \ref{p:effect_GZ} (Remark \ref{rem:effect_GZ})
	that
	
	(4) $|f_{0}|^{2}e^{-2\varphi_{0}}e^{-2\Phi^U_{o,\max}}$ is not integrable near $o$;
	
	(5) $\Phi^U_{o,\max}\geq\Phi_{o,\max}+O(1)$, which implies that $\Phi^U_{o,\max}=\Phi_{o,\max}+O(1)$.
	
	(5) shows that for any negative plurisubharmonic function $\psi$ on $U$,
	\[\psi\leq \sigma(\psi,\Phi_{o,\max})\Phi^U_{o,\max}+O(1),\]
	which implies that
	$|f_{0}|^{2}e^{-\varphi_{0}}e^{-2\max\big\{\Phi^U_{o,\max},\frac{1}{\sigma(\psi,\Phi_{o,\max})}\psi\big\}}$ is not integrable near $o$.
	Note that $\max\big\{\Phi^U_{o,\max},\frac{1}{\sigma(\psi,\Phi_{o,\max})}\psi\big\}<0$ on $U$,
	then the definition of $\Phi^U_{o,\max}$ shows that
	$$\Phi^U_{o,\max}\geq \max\bigg\{\Phi^U_{o,\max},\frac{1}{\sigma(\psi,\Phi_{o,\max})}\psi\bigg\},$$
	which implies that
	$$\frac{1}{\sigma(\psi,\Phi_{o,\max})}\psi\leq\Phi^U_{o,\max}.$$
	This proves inequality \eqref{equ:xiaomage}.
	
	In the following part, we prove inequality \eqref{equ:jianhao}.
	
	Note that for small enough neighborhood $U\ni o$, the following statements hold
	
	(1) $\psi<0$ on $U$;
	
	(2) there exists subsequence $\psi_{j_{k}}$ of $\psi_{j}$ convergent to $\psi$  in the sense of $L^{1}_{\mathrm{loc}}$ and almost everywhere with respect to the Lebesgue measure at the same time when $k\to+\infty$;
	
	(3)  $\psi_{j_{k}}<0$ on $U$ for any $k$.
	$\\$
	Then inequality (\ref{equ:xiaomage}) shows that
	$\psi_{j_{k}}\leq \sigma(\psi_{j_{k}},\Phi_{o,\max})\Phi^U_{o,\max}$ for any $k$.
	Without loss of generality,
	we can assume that
	$$\lim_{k\to+\infty}\sigma(\psi_{j_{k}},\Phi_{o,\max})=\limsup_{j\to+\infty}\sigma(\psi_{j},\Phi_{o,\max})$$
	on $U$.
	
	Note that
	$$\sup_{k\geq k_{0}}\{\psi_{j_{k}}\}\leq\sup_{k\geq k_{0}}\left(\sigma(\psi_{j_{k}},\Phi_{o,\max})\Phi^U_{o,\max}\right)$$
	on $U$,
	then
	\begin{equation}
		\label{equ:wushuang}
		\begin{split}
			&\lim_{k_{0}\to+\infty}\sup_{k\geq k_{0}}\{\psi_{j_{k}}\}
			\\\leq&
			\lim_{k_{0}\to+\infty}\sup_{k\geq k_{0}}\left(\sigma(\psi_{j_{k}},\Phi_{o,\max})\Phi^U_{o,\max}\right)
			\\=&\lim_{k\to+\infty}\sigma(\psi_{j_{k}},\Phi_{o,\max})\Phi^U_{o,\max}
			\\=&\left(\limsup_{j\to+\infty}\sigma(\psi_{j},\Phi_{o,\max})\right)\Phi^U_{o,\max}
		\end{split}
	\end{equation}
	on $U$.
	Lemma \ref{lem:Choquet} shows that
	$$\lim_{k_{0}\to+\infty}\big(\sup_{k\geq k_{0}}\{\psi_{j_{k}}\}\big)^{*}$$
	is a plurisubharmonic function on $U$
	and
	$$\lim_{k_{0}\to+\infty}\big(\sup_{k\geq k_{0}}\{\psi_{j_{k}}\}\big)^{*}=\lim_{k_{0}\to+\infty}\big(\sup_{k\geq k_{0}}\{\psi_{j_{k}}\}\big)$$
	on $U$ almost everywhere with respect to the Lebesgue measure. Note that
	$$\psi=\lim_{k_{0}\to+\infty}\sup_{k\geq k_{0}}\{\psi_{j_{k}}\}$$
	almost everywhere with respect to the Lebesgue measure on $U$,
	then
	$$\psi=\lim_{k_{0}\to+\infty}(\sup_{k\geq k_{0}}\{\psi_{j_{k}}\})^{*}$$
	on $U$.
	Combining with inequality \eqref{equ:wushuang},
	we obtain
	$$\psi\leq\limsup_{j\to+\infty}\sigma(\psi_{j},\Phi_{o,\max})\Phi^U_{o,\max}$$
	on $U$, which is inequality \eqref{equ:jianhao}.
	This proves Remark \ref{rem:sharp_bound}.
\end{proof}

\section{Concavity: real analysis}

Let $D$ be a bounded domain in $\mathbb{C}^{n}$, and the origin $o\in D$. Let $u$ and $v$ be Lebesgue measurable functions on $D$ with upper-bound near $o$. Let $g$ be a nonnegative Lebesgue measurable function on $D$.

\begin{Lemma}[see \cite{demailly2010}]
	\label{lem:G_key}
	Assume that $g^{2}e^{2(l_{1}v-(1+l_{2})u)}$ is integrable near $o$, where $l_{1},l_{2}>0$.
	Then
	$g^{2}e^{-2u}-g^{2}e^{-2\max\big\{u,\frac{l_{1}}{l_{2}}v\big\}}$ is integrable
	on a small enough neighborhood $V_{o}$ of $o$.
\end{Lemma}

\begin{proof}
	Recall (see \cite{demailly2010}) that
	\begin{equation*}
		\begin{split}
			&\int_{V_{o}}g^{2}\left(e^{-2u}-e^{-2\max\big\{u,\frac{l_{1}}{l_{2}}v\big\}}\right)
			\\=&
			\int_{V_{o}\cap\{l_{2}u<l_{1}v\}}g^{2}\left(e^{-2u}-e^{-2\max\big\{u,\frac{l_{1}}{l_{2}}v\big\}}\right)
			\\ \leq&
			\int_{V_{o}\cap\{l_{2}u<l_{1}v\}}g^{2}e^{-2u}
			\\\leq&
			\int_{V_{o}\cap\{l_{2}u<l_{1}v\}}g^{2}\left(e^{-2(l_{2}u-l_{1}v)}e^{-2u}\right)
			\\\leq&
			\int_{V_{o}}g^{2}e^{-2(l_{2}u-l_{1}v)-2u}
			\\\leq&
			\int_{V_{o}}g^{2}e^{2l_{1}v-2(1+l_{2})u}.
		\end{split}
	\end{equation*}
	This proves Lemma \ref{lem:G_key} for $V_{o}$ small enough.
\end{proof}

Denote that
\[A_{u,v}(t):=\sup\big\{c:g^{2}e^{2(tv-cu)} \ \text{is integrable near} \ o\big\}.\]
Note that $u$ and $v$ are local upper-bounded near $o$, then $A_{u,v}(t)$ is increasing with respect to $t$ (maybe $+\infty$ or $-\infty$).
Assume that $A_{u,v}(t)\in(0,+\infty]$ on $(t_{0}-\delta,t_{0}+\delta)$.
It follows from H\"{o}lder inequality that $A_{u,v}(t)$ is concave on $(t_{0}-\delta,t_{0}+\delta)$.

\begin{Lemma}
	\label{lem:hanjiangxue1}
	Assume that $A_{u,v}(t)$ is strictly increasing on $(t_{0}-\delta,t_{0}+\delta)$.
	Then
	\begin{equation*}
		A_{\max\{u,\frac{1}{b}v\},v}(t_{0})=A_{u,v}(t_{0})
	\end{equation*}
	holds for any $b\in\left(0,\lim_{\Delta t\to0+0}\frac{A_{u,v}(t_{0}+\Delta t)-A_{u,v}(t_{0})}{\Delta t}\right)$.
\end{Lemma}

\begin{proof}
	As $A_{u,v}(t)$ is strictly increasing,
	then $g^{2}e^{2(t_{0}+\Delta t+\varepsilon)v}e^{-2A_{u,v}(t_{0}+\Delta t)u}$ is integrable for any pair $(\Delta t,\varepsilon)$ satisfying $\Delta t,$ $\varepsilon>0$ and $\Delta t+\varepsilon<+\delta$.
	
	Note that
	\begin{equation*}
		\begin{split}
			&g^{2}e^{2(t_{0}+\Delta t+\varepsilon)v}e^{-2A_{u,v}(t_{0}+\Delta t)u}
			\\=
			&g^{2}e^{2(t_{0}-\varepsilon)v}e^{2(\Delta t+2\varepsilon)v}e^{-2\left(\frac{A_{u,v}(t_{0}+\Delta t)-A_{u,v}(t_{0})}{A_{u,v}(t_{0})}+1\right)A_{u,v}(t_{0})u}.
		\end{split}
	\end{equation*}
	Then it follows from Lemma \ref{lem:G_key} 
	\[\left(g^{2}\sim g^{2}e^{2(t_{0}-\varepsilon)v}, \ l_{1}\sim \Delta t+2\varepsilon, \ l_{2}\sim\frac{A_{u,v}(t_{0}+\Delta t)-A_{u,v}(t_{0})}{A_{u,v}(t_{0})}\right)\]
	that
	\begin{equation*}
		\begin{split}
			&
			g^{2}e^{2(t_{0}-\varepsilon)v}e^{-2A_{u,v}(t_{0})u}-g^{2}e^{2(t_{0}-\varepsilon)v}e^{-2\max\big\{A_{u,v}(t_{0})u,\frac{\Delta t+2\varepsilon}{\frac{A_{u,v}(t_{0}+\Delta t)-A_{u,v}(t_{0})}{A_{u,v}(t_{0})}}v\big\}}
			\\=&
			g^{2}e^{2(t_{0}-\varepsilon)v}e^{-2A_{u,v}(t_{0})u}-g^{2}e^{2(t_{0}-\varepsilon)v}e^{-2A_{u,v}(t_{0})\max\big\{u,\frac{1}{\frac{A_{u,v}(t_{0}+\Delta t)-A_{u,v}(t_{0})}{\Delta t+2\varepsilon}}v\big\}}
		\end{split}
	\end{equation*}
	is integrable near $o$.
	
	Note that $A_{u,v}$ is strictly increasing on $(t_{0}-\delta,t_{0}+\delta)$,
	then
	$$g^{2}e^{2(t_{0}-\varepsilon)v}e^{-2A_{u,v}(t_{0})u}$$
	is not integrable near $o$,
	which implies that
	$$g^{2}e^{2(t_{0}-\varepsilon)v}e^{-2A_{u,v}(t_{0})\max\big\{u,\frac{1}{\frac{A(t_{0}+\Delta t)-A(t_{0})}{\Delta t+2\varepsilon}}v\big\}}$$
	is not integrable near $o$.
	
	Note that for any $b\in\left(0,\lim_{\Delta t\to0+0}\frac{A_{u,v}(t_{0}+\Delta t)-A_{u,v}(t_{0})}{\Delta t}\right)$,
	there exists small enough $\delta_{1},\delta_{2}>0$ such that
	for any $\Delta t\in(0,\delta_{1})$ and $\varepsilon\in(0,\delta_{2}\Delta t)$,
	$$\frac{A(t_{0}+\Delta t)-A(t_{0})}{\Delta t+2\varepsilon}>b,$$
	which implies that
	$$g^{2}e^{2(t_{0}-\varepsilon)v}e^{-2A_{u,v}(t_{0})\max\{u,\frac{1}{b}v\}}\left(\geq g^{2}e^{2(t_{0}-\varepsilon)v}e^{-2A_{u,v}(t_{0})\max\big\{u,\frac{1}{\frac{A(t_{0}+\Delta t)-A(t_{0})}{\Delta t+2\varepsilon}}v\big\}}\right)$$
	is not integrable near $o$.
	It is clear that
	$A_{\max\{u,\frac{1}{b}v\},v}(t_{0}-\varepsilon)\leq A_{u,v}(t_{0})$.
	For the arbitrariness of $\varepsilon\big(\in(0,\delta_{2}\Delta t)\big)$,
	and the continuity of $A_{\max\{u,\frac{1}{b}v\},v}(t)$ near $t_{0}$,
	it is clear that
	$A_{\max\{u,\frac{1}{b}v\},v}(t_{0})\leq A_{u,v}(t_{0})$.
	Note that $\max\{u,\frac{1}{b}v\}\geq u$,
	then it follows that $A_{\max\{u,\frac{1}{b}v\},v}(t_{0})\geq A_{u,v}(t_{0})$,
	which implies that
	$A_{\max\{u,\frac{1}{b}v\},v}(t_{0})=A_{u,v}(t_{0})$.
\end{proof}

\begin{Lemma}
	\label{lem:hanjiangxue2}
	Assume that $A_{u,v}(t)$ is strictly increasing on $(t_{0}-\delta,t_{0}+\delta)$.
	Then
	\begin{equation*}
		A_{\max\{u,\frac{1}{b}v\},v}(t_{0})=A_{u,v}(t_{0})
	\end{equation*}
	holds for any $b\in\left(0,\lim_{\Delta t\to0-0}\frac{A_{u,v}(t_{0}+\Delta t)-A_{u,v}(t_{0})}{\Delta t}\right)$.
\end{Lemma}

\begin{proof}
	The concavity of $A_{u,v}$ shows that for any $t\in(t_{0}-\delta,t_{0})$,
	$$\lim_{\Delta t\to0+0}\frac{A_{u,v}(t+\Delta t)-A_{u,v}(t)}{\Delta t})\geq\lim_{\Delta t\to0-0}\frac{A_{u,v}(t_{0}+\Delta t)-A_{u,v}(t_{0})}{\Delta t}.$$
	Lemma \ref{lem:hanjiangxue1} shows that
	for any
	$$b\in\left(0,\lim_{\Delta t\to0-0}\frac{A_{u,v}(t_{0}+\Delta t)-A_{u,v}(t_{0})}{\Delta t}\right)\subseteq\left(0,\lim_{\Delta t\to0+0}\frac{A_{u,v}(t+\Delta t)-A_{u,v}(t)}{\Delta t}\right),$$
	we have that
	$$A_{\max\{u,\frac{1}{b}v\},v}(t)=A_{u,v}(t)$$
	holds for any $t\in(t_{0}-\delta,t_{0})$.
	For the continuity of $A_{\max\{u,\frac{1}{b}v\},v}$ and $A_{u,v}$ at $t_{0}$,
	the proof is done.
\end{proof}

\begin{Lemma}
	\label{lem:jialidun0131}
	Assume that $A_{u,v}(t)$ is strictly increasing on $(t_{0}-\delta,t_{0}+\delta)$.
	Then
	\begin{equation}
		\label{equ:jialidun1}
		A_{\max\{u,\frac{1}{b}v\},v}(t_{0})=A_{u,v}(t_{0})
	\end{equation}
	holds for any $b\in\left(0,\lim_{\Delta t\to0-0}\frac{A_{u,v}(t_{0}+\Delta t)-A_{u,v}(t_{0})}{\Delta t}\right]$.
\end{Lemma}

\begin{proof}
	Let
	\[b_{0}=\lim_{\Delta t\to0-0}\frac{A_{u,v}(t_{0}+\Delta t)-A_{u,v}(t_{0})}{\Delta t}.\]
	Lemma \ref{lem:hanjiangxue2} shows that equality \eqref{equ:jialidun1} holds for any
	$b\in(0,b_{0})$.
	
	As $u$ and $v$ are local bounded above near $o$,
	we assume that $u<0$ and $v<0$. Note that for any $b\in(0,b_{0})$, 
	\[u\leq\max\left\{u,\frac{1}{b_{0}}v\right\}\leq\max\left\{\frac{b}{b_{0}}u,\frac{1}{b_{0}}v\right\}\leq\frac{b}{b_{0}}\max\left\{u,\frac{1}{b}v\right\},\]
	then it is clear that
	\begin{equation*}
		\begin{split}
			A_{u,v}(t_{0})\leq
			&A_{\max\{u,\frac{1}{b_{0}}v\},v}(t_{0})
			\\\leq
			&A_{\max\{\frac{b}{b_{0}}u,\frac{1}{b_{0}}v\},v}(t_{0})
			\leq A_{\frac{b}{b_{0}}\max\{u,\frac{1}{b}v\},v}(t_{0})
			\\=&\frac{b}{b_{0}}A_{\max\{u,\frac{1}{b}v\},v}(t_{0})
			=\frac{b}{b_{0}}A_{u,v}(t_{0}).
		\end{split}
	\end{equation*}
	Letting $b\to b_{0}$, we obtain equality \eqref{equ:jialidun1}.
\end{proof}

\begin{Lemma}
	\label{lem:hanjiangxue3}
	Assume that $A_{u,v}(t)$ is strictly increasing on $(t_{0}-\delta,t_{0}+\delta)$.
	Then
	\begin{equation*}
		A_{\max\{u,\frac{1}{b}v\},v}(t_{0})=A_{u,v}(t_{0})
	\end{equation*}
	does not hold for any $b>\lim_{\Delta t\to0-0}\frac{A_{u,v}(t_{0}+\Delta t)-A_{u,v}(t_{0})}{\Delta t}$.
\end{Lemma}

\begin{proof}
	We prove by contradiction:
	if not,
	then $A_{\max\{u,\frac{1}{b}v\},v}(t_{0})=A_{u,v}(t_{0})$ holds for some $b>\lim_{\Delta t\to0-0}\frac{A_{u,v}(t_{0}+\Delta t)-A_{u,v}(t_{0})}{\Delta t}$.
	Note that
	$$A_{\max\{u,\frac{1}{b}v\},v}(t)\geq A_{u,v}(t)$$
	holds for any $t\in(t_{0}-\delta,t_{0}+\delta)$.
	Then
	\begin{equation}
		\label{equ:hanjiangxue_a}
		\begin{split}
			&\lim_{\Delta t\to0-0}\frac{A_{\max\{u,\frac{1}{b}v\},v}(t_{0}+\Delta t)-A_{\max\{u,\frac{1}{b}v\},v}(t_{0})}{\Delta t}
			\\\leq&\lim_{\Delta t\to0-0}\frac{A_{u,v}(t_{0}+\Delta t)-A_{u,v}(t_{0})}{\Delta t}<b.
		\end{split}
	\end{equation}
	Note that
	$$\max\left\{u,\frac{1}{b}v\right\}\geq \frac{1}{b}v,$$
	which implies that
	$$A_{\max\{u,\frac{1}{b}v\},v}(t_{0}+\Delta t)\geq A_{\max\{u,\frac{1}{b}v\},v}(t_{0})+b \Delta t.$$
	Then we obtain
	\begin{equation*}
		\begin{split}
			\lim_{\Delta t\to0+0}\frac{A_{\max\{u,\frac{1}{b}v\},v}(t_{0}+\Delta t)-A_{\max\{u,\frac{1}{b}v\},v}(t_{0})}{\Delta t}\geq b,
		\end{split}
	\end{equation*}
	which contradicts inequality \eqref{equ:hanjiangxue_a}.
\end{proof}

\section{Tian functions and Zhou numbers}
\label{sec:tian function}

Let $\varphi$, $\psi$, $\varphi_{0}$ be plurisubharmonic functions near $o$, and let $f_{0}=(f_{0,1},\ldots,f_{0,m})$ be a vector, where $f_{0,1},\ldots,f_{0,m}$ are holomorphic functions near $o$.

Denote that
\[c_{o}(\varphi,t\psi):=\sup\big\{c:|f_{0}|^{2}e^{-2\varphi_{0}}e^{2t\psi}e^{-2c\varphi} \ \text{is integrable near} \ o\big\},\]
which is a generalization of the jumping number (see \cite{JON-Mus2012,JON-Mus2014}).
Define the Tian function 
$$\mathrm{Tn}(t):=c_{o}(\varphi,t\psi)$$
for any $t\in\mathbb{R}$.

Assume that the following three statements hold

(1) $|f_{0}|^{2}e^{-2\varphi_{0}}$ is integrable near $o$;

(2) There exists integer $N_{0}\gg0$ such that $|f_{0}|^{2}e^{-2\varphi_{0}}|z|^{2N_{0}}e^{-2\mathrm{Tn}(0)\varphi}$ is integrable near $o$;

(3) The Lelong number $\nu(\psi,o)>0$, i.e. there exists $\nu_{0}>0$, such that $\psi\leq\nu_{0}\log|z|+O(1)$ holds (see \cite{demailly-book}).

Theorem \ref{thm:SOC} shows that (2) implies that there exists $\varepsilon>0$,
such that
$$|f_{0}|^{2}e^{-2\varphi_{0}}|z|^{2N_{1}}e^{-2(1+\varepsilon)\mathrm{Tn}(0)\varphi}$$
is integrable near $o$.
Lemma \ref{lem:G_key} shows that
$$|f_{0}|^{2}e^{-2\varphi_{0}}e^{-2\mathrm{Tn}(0)\varphi}-|f_{0}|^{2}e^{-2\varphi_{0}}e^{-2\max\big\{\mathrm{Tn}(0)\varphi,\frac{N_{1}}{\varepsilon}\log|z|\big\}}$$
is integrable near $o$.
Then it follows from the definition of $\Phi_{o,\max}$ that
\begin{equation*}
	\Phi_{o,\max}\geq \max\left\{\mathrm{Tn}(0)\varphi,\frac{N_{1}}{\varepsilon}\log|z|\right\}+O(1)\geq N\log|z|+O(1)
\end{equation*}
near $o$.

The H\"{o}lder inequality shows that $\mathrm{Tn}(t)$ is concave with respect to $t\in(-\infty,+\infty)$.

\subsection{$\mathrm{Tn}(t)$ related to $\varphi$}

In this section, we discuss the derivatives of Tian functions $\mathrm{Tn}(t)$.

We give the	strictly increasing property of Tian functions in the following lemma.
\begin{Lemma}
	\label{lem:strict_decreasing}
	$\mathrm{Tn}(t)$ is strictly increasing near $0$.
\end{Lemma}

\begin{proof}
	Theorem \ref{thm:SOC} implies that
	there exists $\varepsilon_{0}>0$ such that
	$$|f_{0}|^{2}e^{-2\varphi_{0}}|z|^{2N_{0}}e^{-2(1+\varepsilon_{0})\mathrm{Tn}(0)\varphi}$$
	is integrable near $o$.
	Lemma \ref{lem:G_key}  shows that
	$$|f_{0}|^{2}e^{-2\varphi_{0}}e^{-2\big(1+\frac{\varepsilon_{0}}{2}\big)\mathrm{Tn}(0)\varphi}-
	|f_{0}|^{2}e^{-2\varphi_{0}}e^{-2\max\big\{\big(1+\frac{\varepsilon_{0}}{2}\big)\mathrm{Tn}(0)\varphi,\frac{2+\varepsilon_{0}}{\varepsilon_{0}}\log|z|^{N_{0}}\big\}}$$
	is integrable near $o$.
	
	Note that for any $t>0$,
	$$
	|f_{0}|^{2}e^{-2\varphi_{0}}\left(e^{2t\psi}e^{-2\big(1+\frac{\varepsilon_{0}}{2}\big)\mathrm{Tn}(0)\varphi}-
	e^{2t\psi}e^{-2\max\big\{\big(1+\frac{\varepsilon_{0}}{2}\big)\mathrm{Tn}(0)\varphi,\frac{2+\varepsilon_{0}}{\varepsilon_{0}}\log|z|^{N_{0}}\big\}}\right)
	$$
	is integrable near $o$
	and
	\begin{equation*}
		\begin{split}
			&|f_{0}|^{2}e^{-2\varphi_{0}}e^{2t\psi}e^{-2\max\big\{\big(1+\frac{\varepsilon_{0}}{2}\big)\mathrm{Tn}(0)\varphi,\frac{2+\varepsilon_{0}}{\varepsilon_{0}}\log|z|^{N_{0}}\big\}}
			\\\leq&
			C|f_{0}|^{2}e^{-2\varphi_{0}}|z|^{2t\nu(\psi,o)}e^{-2\max\big\{\big(1+\frac{\varepsilon_{0}}{2}\big)\mathrm{Tn}(0)\varphi,\frac{2+\varepsilon_{0}}{\varepsilon_{0}}\log|z|^{N_{0}}\big\}}
			\\\leq&
			C|f_{0}|^{2}e^{-2\varphi_{0}}|z|^{2t\nu(\psi,o)}e^{-2\frac{(2+\epsilon_0)N_{0}}{\varepsilon_{0}}\log|z|}
		\end{split}
	\end{equation*}
	near $o$, then it is clear that for any
	$t>\frac{2+\epsilon_0}{\varepsilon_{0}}\frac{N_0}{\nu(\psi,o)}$,
	\[|f_{0}|^{2}e^{-2\varphi_{0}}e^{2t\psi}e^{-2(1+\frac{\varepsilon_{0}}{2})\mathrm{Tn}(0)\varphi}\]
	is integrable near $o$, which shows
	$\mathrm{Tn}(t)>\big(1+\frac{\varepsilon_{0}}{2}\big)\mathrm{Tn}(0)$ for any $t>\frac{2+\epsilon_0}{\varepsilon_{0}}\frac{N_0}{\nu(\psi,o)}$.
	Then the concavity of $\mathrm{Tn}(t)$ implies that $\mathrm{Tn}(t)$ is strictly increasing near $0$.
\end{proof}

Denote that $b_0:=\lim_{t\rightarrow0+0}\frac{\mathrm{Tn}(0)-\mathrm{Tn}(-t)}{t}$. Lemma \ref{lem:jialidun0131} shows that

\begin{Lemma}
	\label{lem:upper}
	$|f_{0}|^{2}e^{-2\varphi_{0}}e^{-2\max\{\mathrm{Tn}(0)\varphi,\frac{\mathrm{Tn}(0)}{b_{0}}\psi\}}$
	is not integrable near $o$.
\end{Lemma}

Lemma \ref{lem:hanjiangxue3} shows that

\begin{Lemma}
	\label{lem:lower}
	For any $b>b_{0}$,
	$|f_{0}|^{2}e^{-2\varphi_{0}}e^{-2\max\{\mathrm{Tn}(0)\varphi,\frac{1}{b}\mathrm{Tn}(0)\psi\}}$
	is integrable near $o$.
\end{Lemma}

The combination of Lemma \ref{lem:upper} and Lemma \ref{lem:lower} shows

\begin{Lemma}
	\label{lem:combi}
	$|f_{0}|^{2}e^{-2\varphi_{0}}e^{-2\max\{\mathrm{Tn}(0)\varphi,\frac{1}{b}\mathrm{Tn}(0)\psi\}}$
	is not integrable near $o$ if and only if $b\leq b_{0}$.
\end{Lemma}

Using Lemma \ref{lem:combi},
considering $\max\big\{\mathrm{Tn}(0)\varphi,\frac{1}{b_{0}}\mathrm{Tn}(0)\psi\big\}$ instead of $\varphi$ in Remark \ref{rem:max_existence},
we obtain

\begin{Lemma}
	\label{lem:low}
	There exists a local Zhou weight $\Phi_{o,\max}$ related to $|f_{0}|^{2}e^{-2\varphi_{0}}$
	such that
	
	(1) $\sigma(\psi,\Phi_{o,\max})=\frac{b_{0}}{\mathrm{Tn}(0)}$;
	
	(2) $\Phi_{o,\max}\geq \max\big\{\mathrm{Tn}(0)\varphi,\frac{1}{b_{0}}\mathrm{Tn}(0)\psi\big\}+O(1)$ near $o$.
	
\end{Lemma}

The following property of Tian functions $\mathrm{Tn}(t)$ will be used in the proof of Theorem \ref{thm:main_value}.

\begin{Proposition}
	\label{p:sidediv}Assume that there exists $N\gg0$ such that $\varphi\geq N\log|z|$ near $o$. The following inequality holds
	\begin{equation*}
		\begin{split}
			&\frac{1}{\mathrm{Tn}(0)}\lim_{t\to0+0}\frac{\mathrm{Tn}(0)-\mathrm{Tn}(t)}{-t}
			\\\leq&
			\liminf_{t_{1}\to+\infty}\frac{1}{2t_{1}}\frac{\int_{\{\mathrm{Tn}(0)\varphi<-t_{1}\}\cap U}|f_{0}|^{2}e^{-2\varphi_{0}}(-2\psi)}{\int_{\{\mathrm{Tn}(0)\varphi<-t_{1}\}\cap U}|f_{0}|^{2}e^{-2\varphi_{0}}}
			\\\leq&
			\limsup_{t_{1}\to+\infty}\frac{1}{2t_{1}}\frac{\int_{\{\mathrm{Tn}(0)\varphi<-t_{1}\}\cap U}|f_{0}|^{2}e^{-2\varphi_{0}}(-2\psi)}{\int_{\{\mathrm{Tn}(0)\varphi<-t_{1}\}\cap U}|f_{0}|^{2}e^{-2\varphi_{0}}}
			\\\leq &\frac{1}{\mathrm{Tn}(0)}\lim_{t\to0+0}\frac{\mathrm{Tn}(0)-\mathrm{Tn}(-t)}{t}.
		\end{split}
	\end{equation*}
\end{Proposition}

\begin{proof}
	We prove Proposition \ref{p:sidediv} in two steps.
	
	\textbf{Step 1.}
	Theorem \ref{thm:SOC} shows that $|f_{0}|^{2}e^{-2\varphi_{0}}e^{-2t\psi}$ is local integrable near $o$ for small enough $t>0$.
	Note that there exists $N\gg0$ such that $\varphi\geq N\log|z|$ near $o$,
	then there exists a neighborhood $U$ of $o$ such that
	for any small enough $t>0$ and $\varepsilon>0$,
	$$\limsup_{t_{1}\to+\infty}\int_{\{\mathrm{Tn}(0)\varphi<-t_{1}\}\cap U}|f_{0}|^{2}e^{-2\varphi_{0}}e^{-2\big(t\psi+(1-\varepsilon)\mathrm{Tn}(-t)\varphi\big)}=0,$$
	which implies that
	$$\limsup_{t_{1}\to+\infty}e^{2t_{1}}\int_{\{\mathrm{Tn}(0)\varphi<-t_{1}\}\cap U}|f_{0}|^{2}e^{-2\varphi_{0}}e^{-2\big(t\psi+(1-\varepsilon)\mathrm{Tn}(-t)\varphi+t_{1}\big)}=0.$$
	Then for large enough $t_{1}>0$,
	$$\int_{\{\mathrm{Tn}(0)\varphi<-t_{1}\}\cap U}|f_{0}|^{2}e^{-2\varphi_{0}}e^{-2\big(t\psi+(1-\varepsilon)\mathrm{Tn}(-t)\varphi+t_{1}\big)}<e^{-2t_{1}},$$
	i.e.
	$$\log\left(\int_{\{\mathrm{Tn}(0)\varphi<-t_{1}\}\cap U}|f_{0}|^{2}e^{-2\varphi_{0}}e^{-2\big(t\psi+(1-\varepsilon)\mathrm{Tn}(-t)\varphi+t_{1}\big)}\right)<-2t_{1}.$$
	Combining with Lemma \ref{lem:jump_asyp_C},
	we obtain
	\begin{equation}
		\label{equ:tiandao}
		\begin{split}
			&\limsup_{t_{1}\to+\infty}\frac{1}{2t_{1}}\log\frac{\int_{\{\mathrm{Tn}(0)\varphi<-t_{1}\}\cap U}|f_{0}|^{2}e^{-2\varphi_{0}}e^{-2\big(t\psi+(1-\varepsilon)\mathrm{Tn}(-t)\varphi+t_{1}\big)}}{\int_{\{\mathrm{Tn}(0)\varphi<-t_{1}\}\cap U}|f_{0}|^{2}e^{-2\varphi_{0}}}
			\\=&
			\limsup_{t_{1}\to+\infty}\frac{1}{2t_{1}}\log\int_{\{\mathrm{Tn}(0)\varphi<-t_{1}\}\cap U}|f_{0}|^{2}e^{-2\varphi_{0}}e^{-2\big(t\psi+(1-\varepsilon)\mathrm{Tn}(-t)\varphi+t_{1}\big)}
			\\&-\lim_{t_{1}\to+\infty}\frac{1}{2t_{1}}\log\int_{\{\mathrm{Tn}(0)\varphi<-t_{1}\}\cap U}|f_{0}|^{2}e^{-2\varphi_{0}}
			\\\leq&1-1=0.
		\end{split}
	\end{equation}

	Jensen's inequality and the concavity of logarithm shows that
	\begin{equation}
		\nonumber
		\begin{split}
			&\log\frac{\int_{\{\mathrm{Tn}(0)\varphi<-t_{1}\}\cap U}|f_{0}|^{2}e^{-2\varphi_{0}}e^{-2\big(t\psi+(1-\varepsilon)\mathrm{Tn}(-t)\varphi+t_{1}\big)}}{\int_{\{\mathrm{Tn}(0)\varphi<-t_{1}\}\cap U}|f_{0}|^{2}e^{-2\varphi_{0}}}
			\\\geq&
			\frac{\int_{\{\mathrm{Tn}(0)\varphi<-t_{1}\}\cap U}|f_{0}|^{2}e^{-2\varphi_{0}}\log\left(e^{-2(t\psi+(1-\varepsilon)\mathrm{Tn}(-t)\varphi)+t_{1}}\right)}{\int_{\{\mathrm{Tn}(0)\varphi<-t_{1}\}\cap U}|f_{0}|^{2}e^{-2\varphi_{0}}}
			\\=&
			\frac{\int_{\{\mathrm{Tn}(0)\varphi<-t_{1}\}\cap U}|f_{0}|^{2}e^{-2\varphi_{0}}\big(-2(t\psi+(1-\varepsilon)\mathrm{Tn}(-t)\varphi+t_{1})\big)}{\int_{\{\mathrm{Tn}(0)\varphi<-t_{1}\}\cap U}|f_{0}|^{2}e^{-2\varphi_{0}}}
			\\\geq&
			\frac{\int_{\{\mathrm{Tn}(0)\varphi<-t_{1}\}\cap U}|f_{0}|^{2}e^{-2\varphi_{0}}\big(-2(t\psi+(1-\varepsilon)\mathrm{Tn}(-t)(-t_{1})\frac{1}{\mathrm{Tn}(0)}+t_{1})\big)}{\int_{\{\mathrm{Tn}(0)\varphi<-t_{1}\}\cap U}|f_{0}|^{2}e^{-2\varphi_{0}}}.
		\end{split}
	\end{equation}
	Combining with inequality \eqref{equ:tiandao},
	we obtain that
	$$\limsup_{t_{1}\to+\infty}\frac{1}{2t_{1}}\frac{\int_{\{\mathrm{Tn}(0)\varphi<-t_{1}\}\cap U}|f_{0}|^{2}e^{-2\varphi_{0}}\big(-2(t\psi+(1-\epsilon)\mathrm{Tn}(-t)(-t_{1})\frac{1}{\mathrm{Tn}(0)}+t_{1})\big)}{\int_{\{\mathrm{Tn}(0)\varphi<-t_{1}\}\cap U}|f_{0}|^{2}e^{-2\varphi_{0}}}\leq 0.$$
	Letting $\varepsilon\to0+0$,
	we obtain
	\begin{equation}
		\label{equ:degang20190406c}
		\limsup_{t_{1}\to+\infty}\frac{1}{2t_{1}}\frac{\int_{\{\mathrm{Tn}(0)\varphi<-t_{1}\}\cap U}|f_{0}|^{2}e^{-2\varphi_{0}}(-2\psi)}{\int_{\{\mathrm{Tn}(0)\varphi<-t_{1}\}\cap U}|f_{0}|^{2}e^{-2\varphi_{0}}}\leq \frac{1}{\mathrm{Tn}(0)}\lim_{t\to0+0}\frac{\mathrm{Tn}(0)-\mathrm{Tn}(-t)}{t}.
	\end{equation}
	
	\
	
	\textbf{Step 2.} Theorem \ref{thm:SOC} shows that $|f_{0}|^{2}e^{-2\varphi_{0}}e^{-2t\psi}$ is local integrable near $o$ for small enough $t>0$.
	Note that there exists $N\gg0$ such that $\varphi\geq N\log|z|$ near $o$,
	then there exists a neighborhood $U$ of $o$ such that
	for any small enough $t>0$ and $\varepsilon>0$,
	$$\lim_{t_{1}\to+\infty}\int_{\{\mathrm{Tn}(0)\varphi<-t_{1}\}\cap U}|f_{0}|^{2}e^{-2\varphi_{0}}e^{-2\big(-t\psi+(1-\varepsilon)\mathrm{Tn}(t)\varphi\big)}=0,$$
	which implies that
	$$\lim_{t_{1}\to+\infty}e^{2t_{1}}\int_{\{\mathrm{Tn}(0)\varphi<-t_{1}\}\cap U}|f_{0}|^{2}e^{-2\varphi_{0}}e^{-2\big(-t\psi+(1-\varepsilon)\mathrm{Tn}(t)\varphi+t_{1}\big)}=0.$$
	Then for large enough $t_{1}>0$,
	$$\int_{\{\mathrm{Tn}(0)\varphi<-t_{1}\}\cap U}|f_{0}|^{2}e^{-2\varphi_{0}}e^{-2\big(-t\psi+(1-\varepsilon)\mathrm{Tn}(t)\varphi+t_{1}\big)}<e^{-2t_{1}},$$
	i.e.
	$$\log\left(\int_{\{\mathrm{Tn}(0)\varphi<-t_{1}\}\cap U}|f_{0}|^{2}e^{-2\varphi_{0}}e^{-2\big(-t\psi+(1-\varepsilon)\mathrm{Tn}(t)\varphi+t_{1}\big)}\right)<-2t_{1}.$$
	Combining with Lemma \ref{lem:jump_asyp_C},
	we obtain
	\begin{equation}
		\label{equ:degang20230406}
		\begin{split}
			&\limsup_{t_{1}\to+\infty}\frac{1}{2t_{1}}\log\frac{\int_{\{\mathrm{Tn}(0)\varphi<-t_{1}\}\cap U}|f_{0}|^{2}e^{-2\varphi_{0}}e^{-2\big(-t\psi+(1-\varepsilon)\mathrm{Tn}(t)\varphi+t_{1}\big)}}{\int_{\{\mathrm{Tn}(0)\varphi<-t_{1}\}\cap U}|f_{0}|^{2}e^{-2\varphi_{0}}}
			\\=&
			\limsup_{t_{1}\to+\infty}\frac{1}{2t_{1}}\log\int_{\{\mathrm{Tn}(0)\varphi<-t_{1}\}\cap U}|f_{0}|^{2}e^{-2\varphi_{0}}e^{-2\big(-t\psi+(1-\varepsilon)\mathrm{Tn}(t)\varphi+t_{1}\big)}
			\\&-\lim_{t_{1}\to+\infty}\frac{1}{2t_{1}}\log\int_{\{\mathrm{Tn}(0)\varphi<-t_{1}\}\cap U}|f_{0}|^{2}e^{-2\varphi_{0}}
			\\\leq&1-1=0.
		\end{split}
	\end{equation}

	Jensen's inequality and the concavity of logarithm shows that
	\begin{equation}
		\nonumber
		\begin{split}
			&\log\frac{\int_{\{\mathrm{Tn}(0)\varphi<-t_{1}\}\cap U}|f_{0}|^{2}e^{-2\varphi_{0}}e^{-2\big(-t\psi+(1-\varepsilon)\mathrm{Tn}(t)\varphi+t_{1}\big)}}{\int_{\{\mathrm{Tn}(0)\varphi<-t_{1}\}\cap U}|f_{0}|^{2}e^{-2\varphi_{0}}}
			\\\geq&
			\frac{\int_{\{\mathrm{Tn}(0)\varphi<-t_{1}\}\cap U}|f_{0}|^{2}e^{-2\varphi_{0}}\log\left(e^{-2(-t\psi+(1-\varepsilon)\mathrm{Tn}(t)\varphi+t_{1})}\right)}{\int_{\{\mathrm{Tn}(0)\varphi<-t_{1}\}\cap U}|f_{0}|^{2}e^{-2\varphi_{0}}}
			\\=&
			\frac{\int_{\{\mathrm{Tn}(0)\varphi<-t_{1}\}\cap U}|f_{0}|^{2}e^{-2\varphi_{0}}\big(-2(-t\psi+(1-\varepsilon)\mathrm{Tn}(t)\varphi+t_{1})\big)}{\int_{\{\mathrm{Tn}(0)\varphi<-t_{1}\}\cap U}|f_{0}|^{2}e^{-2\varphi_{0}}}
			\\\geq&
			\frac{\int_{\{\mathrm{Tn}(0)\varphi<-t_{1}\}\cap U}|f_{0}|^{2}e^{-2\varphi_{0}}\big(-2(-t\psi+(1-\varepsilon)\mathrm{Tn}(t)(-t_{1})\frac{1}{\mathrm{Tn}(0)}+t_{1})\big)}{\int_{\{\mathrm{Tn}(0)\varphi<-t_{1}\}\cap U}|f_{0}|^{2}e^{-2\varphi_{0}}}.
		\end{split}
	\end{equation}
	Combining with inequality \eqref{equ:degang20230406},
	we obtain that
	$$\limsup_{t_{1}\to+\infty}\frac{1}{2t_{1}}\frac{\int_{\{\mathrm{Tn}(0)\varphi<-t_{1}\}\cap U}|f_{0}|^{2}e^{-2\varphi_{0}}\big(-2(-t\psi+(1-\epsilon)\mathrm{Tn}(t)(-t_{1})\frac{1}{\mathrm{Tn}(0)}+t_{1})\big)}{\int_{\{\mathrm{Tn}(0)\varphi<-t_{1}\}\cap U}|f_{0}|^{2}e^{-2\varphi_{0}}}\leq 0$$
	Letting $\varepsilon\to0+0$,
	we obtain
	$$\liminf_{t_{1}\to+\infty}\frac{1}{2t_{1}}\frac{\int_{\{\mathrm{Tn}(0)\varphi<-t_{1}\}\cap U}|f_{0}|^{2}e^{-2\varphi_{0}}(-2\psi)}{\int_{\{\mathrm{Tn}(0)\varphi<-t_{1}\}\cap U}|f_{0}|^{2}e^{-2\varphi_{0}}}\geq\frac{1}{\mathrm{Tn}(0)}\frac{\mathrm{Tn}(0)-\mathrm{Tn}(t)}{-t}.$$
	Letting $t\to0+0$,
	we obtain
	\begin{equation}
		\label{equ:20190406b}
		\begin{split}
			\frac{1}{\mathrm{Tn}(0)}\lim_{t\to0+0}\frac{\mathrm{Tn}(0)-\mathrm{Tn}(t)}{-t}
			\leq
			\liminf_{t_{1}\to+\infty}\frac{1}{2t_{1}}\frac{\int_{\{\mathrm{Tn}(0)\varphi<-t_{1}\}\cap U}|f_{0}|^{2}e^{-2\varphi_{0}}(-2\psi)}{\int_{\{\mathrm{Tn}(0)\varphi<-t_{1}\}\cap U}|f_{0}|^{2}e^{-2\varphi_{0}}}.
		\end{split}
	\end{equation}
	
	Combining inequality (\ref{equ:20190406b}) and inequality (\ref{equ:degang20190406c}),
	we obtain Proposition \ref{p:sidediv}.
\end{proof}

\subsection{Zhou numbers related to $\Phi_{o,\max}$}

Let $\Phi_{o,\max}$ be a local Zhou weight related to $|f_0|^2e^{-2\varphi_0}$ near $o$, and let Tian function
$\mathrm{Tn}(t):=c_{o}(\Phi_{o,\max},t\psi)$. The definition of $\Phi_{o,\max}$ shows that $\mathrm{Tn}(0)=1$.

Note that
$|f_{0}|^{2}e^{-2\varphi_{0}}e^{-2\max\{\Phi_{o,\max},\frac{1}{b}\psi\}}$ is not integrable near $o$
if and only if $b\leq \sigma(\psi,\Phi_{o,\max})$,
then Lemma \ref{lem:combi} shows

\begin{Proposition}
	\label{prop:concave_A}
	For any plurisubharmonic function $\psi$ near $o$,
	\begin{equation*}
		\sigma(\psi,\Phi_{o,\max})=\lim_{t\to0+0}\frac{\mathrm{Tn}(0)-\mathrm{Tn}(-t)}{t}.
	\end{equation*}
\end{Proposition}

As $\psi\leq \sigma(\psi,\Phi_{o,\max})\Phi_{o,\max}+O(1)$ near $o$,
the definition of $\mathrm{Tn}(t)$ shows that
$\mathrm{Tn}(t)\geq \mathrm{Tn}(0)+\sigma(\psi,\Phi_{o,\max})t$ for any $t>0$.
Note that $\mathrm{Tn}(t)$ is concave,
then $\mathrm{Tn}(t)\leq \mathrm{Tn}(0)+\sigma(\psi,\Phi_{o,\max})t$ for any $t>0$,
which implies that
\begin{Proposition}
	\label{prop:concave_B}The Tian function $\mathrm{Tn}(t)$ is differentiable at $t=0$, and
	\begin{equation*}
		\mathrm{Tn}(t)=\mathrm{Tn}(0)+\sigma(\psi,\Phi_{o,\max})t
	\end{equation*}
	holds for any $t\ge0$.
\end{Proposition}

Using Proposition \ref{prop:concave_B},
 we give the following property of local Zhou weights.

\begin{Proposition}
	\label{thm:main_value1}
	Let $f=(f_{1},\cdots,f_{m'})$ be a vector, where $f_{1},\cdots,f_{m'}$ are holomorphic functions near $o$.
	Denote that $|f|:=(|f_{1}|^2+\cdots+|f_{m'}|^{2})^{1/2}$.
	Let $\Phi_{o,\max}$ be a local Zhou weight related to $|f_{0}|^{2}e^{-2\varphi_{0}}$ near $o$. Then the following two statements holds:
	
	$(1)$ for any $\alpha>0$,
	$\big(1+\alpha\sigma(\log|f|,\Phi_{o,\max})\big)\Phi_{o,\max}$ is a local Zhou weight related to $|f|^{2\alpha}|f_{0}|^{2}e^{-2\varphi_{0}}$ near $o$;
	
	$(2)$ $\big(1+\sigma(\varphi_{0},\Phi_{o,\max})\big)\Phi_{o,\max}$ is a local Zhou weight related to $|f_{0}|^{2}$.
\end{Proposition}	
\begin{proof}
	Firstly, we prove the statement $(1)$ in Proposition \ref{thm:main_value1}.
	
	Denote
	\[\mathrm{Tn}(t):=\sup\big\{c:|f_0|^2|f|^{2t}e^{-2\varphi_0-2c\Phi_{o,\max}} \ \text{is integrable near} \ o\big\}.\]
	By Proposition \ref{prop:concave_B}, we have $\mathrm{Tn}(\alpha)=1+\sigma(\log|f|,\Phi_{o,\max})\alpha$, which implies that $$|f_0|^2|f|^{2\alpha}e^{-2\varphi_0-2(1+\alpha\sigma(\log|f|,\Phi_{o,\max}))\Phi_{o,\max}}$$ is not integrable near $o$ by Theorem \ref{thm:SOC}. As $\log|f|\le\sigma(\log|f|,\Phi_{o,\max})\Phi_{o,\max}+O(1)$ and there exists $N\gg0$ such that $|f_0|^2|z|^{2N}e^{-2\varphi_0-2\Phi_{o,\max}}$ is integrable near $o$, we know that $$|f_0|^2|f|^{2\alpha}|z|^{2N}e^{-2\varphi_0-2\big(1+\alpha\sigma(\log|f|,\Phi_{o,\max})\big)\Phi_{o,\max}}$$ is integrable near $o$.
	
	Let $\tilde{\varphi}$ be a subharmonic function near $o$ satisfying that $$\tilde\varphi\ge\big(1+\alpha\sigma(\log|f|,\Phi_{o,\max})\big)\Phi_{o,\max}+O(1)$$ and $|f|^{2\alpha}|f_0|^2e^{-2\varphi_0-2\tilde\varphi}$ is not integrable near $o$.
	Note that 
	$$\log|f|\le\sigma(\log|f|,\Phi_{o,\max})\Phi_{o,\max}+O(1),$$
	then 
	\[\tilde{\varphi}\ge\frac{1+\alpha\sigma(\log|f|,\Phi_{o,\max})}{\sigma(\log|f|,\Phi_{o,\max})}\log|f|+O(1)\]
	and
	\begin{equation*}
		\begin{split}
			|f|^{2\alpha}|f_0|^2e^{-2\varphi_0-2\tilde\varphi} & \le C|f_0|^2
			e^{-2\varphi_0}e^{-\frac{2}{1+\alpha\sigma(\log|f|,\Phi_{o,\max})}\tilde{\varphi}}.
		\end{split}
	\end{equation*}
	As $|f|^{2\alpha}|f_0|^2e^{-2\varphi_0-2\tilde\varphi}$ is not integrable near $o$, we know that $$|f_0|^2e^{-2\varphi_0-\frac{2}{1+\alpha\sigma(\log|f|,\Phi_{o,\max})}\tilde\varphi}$$
	is not integrable near $o$. Note that  $\Phi_{o,\max}$ is a local Zhou weight related to $|f_{0}|^{2}e^{-2\varphi_{0}}$. Then we obtain
	$$\tilde\varphi=(1+\alpha\sigma(\log|f|,\Phi_{o,\max}))\Phi_{o,\max}+O(1),$$
	which shows that $\big(1+\alpha\sigma(\log|f|,\Phi_{o,\max})\big)\Phi_{o,\max}$ is a local Zhou weight related to $|f|^{2\alpha}|f_{0}|^{2}e^{-2\varphi_{0}}$.
	
	Next, we give the proof of statement $(2)$, which is similar to the proof of statement $(1)$.
	
	Denote that
	\[\mathrm{Tn}(t):=\sup\big\{c:|f_0|^2e^{-2\varphi_0}e^{2t\varphi_0}e^{-2c\Phi_{o,\max}} \ \text{is integarble near} \ o\big\}.\]
	By Proposition \ref{prop:concave_B}, we have $\mathrm{Tn}(1)=1+\sigma(\varphi_0,\Phi_{o,\max})$, which implies that $$|f_0|^2e^{-2\big(1+\sigma(\varphi_0,\Phi_{o,\max})\big)\Phi_{o,\max}}$$ is not integrable near $o$. As $\varphi_0\le\sigma(\varphi_0,\Phi_{o,\max})\Phi_{o,\max}+O(1)$ and there exists $N\gg0$ such that $|f_0|^2|z|^{2N}e^{-2\varphi_0-2\Phi_{o,\max}}$ is integrable near $o$, we know that $|f_0|^2|z|^{2N}e^{-2\big(1+\sigma(\varphi_0,\Phi_{o,\max})\big)\Phi_{o,\max}}$ is integrable near $o$.
	
	Let $\tilde{\varphi}$ be a subharmonic function near $o$ satisfying that $$\tilde\varphi\ge\big(1+\sigma(\varphi_0,\Phi_{o,\max})\big)\Phi_{o,\max}+O(1)$$ and $|f_0|^2e^{-2\tilde\varphi}$ is not integrable near $o$.
	It follows from $\varphi_0\le\sigma(\varphi_0,\Phi_{o,\max})\Phi_{o,\max}+O(1)$ that
	\begin{equation*}
		\begin{split}
			|f_0|^2e^{-2\tilde\varphi} & =e^{2\varphi_0}|f_0|^2
			e^{-2\varphi_0}e^{-2\tilde{\varphi}} \\
			& \le Ce^{2\sigma(\varphi_0,\Phi_{o,\max})\Phi_{o,\max}}|f_0|^2
			e^{-2\varphi_0}e^{-2\tilde{\varphi}} \\
			& \le C_1|f_0|^2
			e^{-2\varphi_0}e^{2\frac{\sigma(\varphi_0,\Phi_{o,\max})}{1+\sigma(\varphi_0,\Phi_{o,\max})}\tilde{\varphi}}
			e^{-2\tilde{\varphi}}\\
			&= C_1|f_0|^2
			e^{-2\varphi_0}e^{-\frac{2}{1+\sigma(\varphi_0,\varphi)}\tilde{\varphi}}.
		\end{split}
	\end{equation*}
	As $|f_0|^2e^{-2\tilde\varphi}$ is not integrable near $o$, we know that $|f_0|^2e^{-2\varphi_0-\frac{2}{1+\sigma(\varphi_0,\Phi_{o,\max})}\tilde\varphi}$
	is not integrable near $o$. Note that  $\Phi_{o,\max}$ is a local Zhou weight related to $|f_{0}|^{2}e^{-2\varphi_{0}}$. Then we obtain
	$$\tilde\varphi=\big(1+ \sigma(\varphi_0,\Phi_{o,\max})\big)\Phi_{o,\max}+O(1),$$
	which shows that $\big(1+\sigma(\varphi_0,\Phi_{o,\max})\big)\Phi_{o,\max}$ is a local Zhou weight related to $|f_{0}|^{2}$.
\end{proof}

The following proposition gives a comparison result between Zhou weights.
\begin{Proposition}
	\label{prop:gowest}
	Let $\varphi_{i}$ be a local Zhou weight related to $|f_{0,i}|^{2}$, where $i\in\{1,2\}$.
	Assume that
	\begin{equation}
		\label{equ:keluodiya}
		\sigma(\log|f|,\varphi_{1})\leq \sigma(\log|f|,\varphi_{2})
	\end{equation}
	holds for any holomorphic function $f$ near $o$.
	
	Then $\varphi_{1}\leq \varphi_{2}+O(1)$ near $o$.
	Especially,
	``=" in inequality \eqref{equ:keluodiya} holds if and only
	$\varphi_{1}=\varphi_{2}+O(1)$ near $o$.
\end{Proposition}	

Before proving Proposition \ref{prop:gowest}, we give a lemma which will be used in the proof of Proposition \ref{prop:gowest}.

\begin{Lemma}
	\label{lem:gowest}
	Let $\varphi$ be a local Zhou weight related to $|f_{0}|^{2}$ at $o$ on a bounded pseudoconvex domain $D$. Let $\{f_{j}\}_{j=1,\cdots,N}$ be holomorphic functions on $D$ such that
	$\{(f_{j},o)\}_{j}$ generates $\mathcal{I}(m\varphi)_o$.
	Then $$\frac{1}{\sigma(\log\sum_{j}|f_{j}|,\varphi)+1}\log\sum_j|f_{j}|+O(1)\geq \varphi\geq \frac{1}{\sigma(\log\sum_j|f_{j}|,\varphi)}\log\sum_j|f_{j}|+O(1).$$
\end{Lemma}

\begin{proof}
	As  $\varphi$ is a local Zhou weight related to $|f_{0}|^{2}$, we have $$\log\sum_j|f_{j}|\le\sigma\big(\log\sum_j|f_{j}|,\varphi\big)\varphi+O(1).$$
	Thus, it suffices to prove
	$$\frac{1}{\sigma(\log\sum_{j}|f_{j}|,\varphi)+1}\log\sum_j|f_{j}|+O(1)\geq \varphi.$$
	
	Denote $\mathrm{Tn}(t):=c_o(\varphi,t\log\sum_{j}|f_{j}|)$. As $\{(f_{j},o)\}_{j}$ generates $\mathcal{I}(m\varphi)_o$, we have $\mathrm{Tn}(1)\ge m$, which implies that
	\begin{equation}
		\label{eq:0730a}\sigma(\log\sum_j|f_{j}|,\varphi)\ge\lim_{t\rightarrow0+0}\frac{\mathrm{Tn}(t)-1}{t}\ge m-1
	\end{equation}
	by Proposition \ref{prop:concave_A} and the concavity property of $\mathrm{Tn}(t)$.
	By Demailly's approximation theorem (see \cite{demailly2010}, see also Lemma \ref{l:appro-Berg}), we have
	\begin{equation}
		\label{eq:0730b}m\varphi\le\log\sum_j|f_j|+O(1).
	\end{equation}
	Thus, we have $\frac{1}{\sigma(\log\sum_{j}|f_{j}|,\varphi)+1}\log\sum_j|f_{j}|+O(1)\geq \varphi$.
\end{proof}

Now, we prove Proposition \ref{prop:gowest}.

\begin{proof}[Proof of Proposition \ref{prop:gowest}]
	Let $\{f_{j}\}_{j=1,\cdots,N}$ be holomorphic functions on bounded pseudoconvex domain $D$ such that $\{(f_{j},o)\}_{j}$ generates $\mathcal{I}(m\varphi_{1})_o$.
	By definition, we have $\sigma(\log\sum_j|f_{j}|,\varphi_i)=\sigma(\max_j\{\log|f_{j}|\},\varphi_i)=\min_j\{\sigma(\log|f_j|,\varphi_i)\}$ for $i\in\{1,2\}$ (see the proof of statement $(3)$ in Corollary \ref{coro:main}). Then we have
	$$\sigma(\log\sum_j|f_{j}|,\varphi_1)\le\sigma(\log\sum_j|f_{j}|,\varphi_2).$$

	It follows from inequality \eqref{eq:0730a} and \eqref{eq:0730b} that $\sigma(\log\sum_j|f_{j}|,\varphi_1)\in[m-1,m]$. Then we have
	\begin{equation*}
		\begin{split}
			\varphi_{2}&\geq\frac{1}{\sigma(\log\sum_j|f_{j}|,\varphi_{2})}\log\sum_j|f_{j}|+O(1)
			\\&\geq \frac{1}{\sigma(\log\sum_j|f_{j}|,\varphi_{1})}\log\sum_j|f_{j}|+O(1)
			\\&\geq \frac{m+1}{m-1}\frac{1}{\sigma(\log\sum_j|f_{j}|,\varphi_{1})+1}\log\sum_j|f_{j}|+O(1)
			\\&\geq \frac{m+1}{m-1}\varphi_{1}+O(1)
		\end{split}
	\end{equation*}
	near $o$, where the last inequality follows from Lemma \ref{lem:gowest}.
	Hence, we have 
	$$\sigma(\varphi_2,\varphi_1)=1$$
	by letting $m\rightarrow+\infty$. Note that $\varphi_2$ is a local Zhou weight related to $|f_{0,2}|^{2}$, then we have 
	$$\varphi_{1}\leq \sigma(\varphi_2,\varphi_1)\varphi_{2}+O(1)=\varphi_{2}+O(1)$$ near $o$.
\end{proof}

\section{Proofs of Theorem \ref{thm:main_value} and Corollary \ref{coro:main}}

In this section, we prove  Theorem \ref{thm:main_value} and Corollary \ref{coro:main}.

\begin{proof}
	[Proof of Theorem \ref{thm:main_value}]
	It follows from Propositions \ref{p:sidediv}, \ref{prop:concave_A} and \ref{prop:concave_B} that Theorem \ref{thm:main_value} holds. 
\end{proof}

\begin{proof}
	[Proof of Corollary \ref{coro:main}]
	The	statement $(1)$ holds by Theorem \ref{thm:main_value}. 
	
	Denote $a_i:=\sigma(\log|f_i|,\Phi_{o,\max})$ for $i\in\{1,2\}$, then we have $\log|f_i|\le a_i\Phi_{o,\max}+O(1)$ for  $i\in\{1,2\}$, which implies that		
	\begin{equation}
		\nonumber
		\begin{split}
			\log|f_1+f_2|&\le \log(|f_1|+|f_2|)\\
			&\le \log\left(e^{a_1\Phi_{o,\max}}+e^{a_2\Phi_{o,\max}}\right) +O(1)\\
			&\le \min\{a_1,a_2\}\Phi_{o,\max}+O(1).
		\end{split}
	\end{equation}
	Thus, $\sigma(\log|f_1+f_2|,\Phi_{o,\max})\ge \min\big\{\sigma(\log|f_1|,\Phi_{o,\max}),\sigma(\log|f_2|,\Phi_{o,\max})\big\}$.
	
	Denote that $b_i:=\sigma(\psi_i,\Phi_{o,\max})$ for $i\in\{1,2\}$, then we have $\psi_i\le b_i\Phi_{o,\max}+O(1)$ for  $i\in\{1,2\}$, which implies that		
	\begin{equation}
		\nonumber
		\begin{split}
			\max\{\psi_1,\psi_2\}&\le \max\{b_1\Phi_{o,\max},b_2\Phi_{o,\max}\}+O(1)\\
			&=\min\{b_1,b_2\}\Phi_{o,\max}+O(1).
		\end{split}
	\end{equation}
	Thus, $\sigma(\max\{\psi_1,\psi_2\},\Phi_{o,\max})\ge \min\big\{\sigma(\psi_1,\Phi_{o,\max}),\sigma(\psi_2,\Phi_{o,\max})\big\}$. $\max\{\psi_1,\psi_2\}\ge\psi_i$ for $i\in\{1,2\}$ implies $\sigma(\max\{\psi_1,\psi_2\},\Phi_{o,\max})\le \min\big\{\sigma(\psi_1,\Phi_{o,\max}),\sigma(\psi_2,\Phi_{o,\max})\big\}$. Then the statement $(3)$ holds.
\end{proof}

\section{Jumping numbers and Zhou valuations: proof of Theorem \ref{thm:valu-jump}}
Let $f_{0}=(f_{0,1},\cdots,f_{0,m})$ be a vector,
where $f_{0,1},\cdots,f_{0,m}$ are holomorphic functions near $o$.
Denote $|f_{0}|^{2}=|f_{0,1}|^{2}+\cdots+|f_{0,m}|^{2}$.  Let $\Phi_{o,\max}$ be a local Zhou weight related to $|f_{0}|^{2}e^{-2\varphi_0}$ near $o$.
Let $G$	be a holomorphic function on $U$. Denote $$k_{\varphi_0}:=\sigma(\varphi_0,\Phi_{o,\max}).$$ Recall  the definition of jumping number 
$$c^G_o(\Phi_{o,\max}):=\sup\big\{c:|G|^{2}e^{-2c\Phi_{o,\max}}\text{  is integrable near $o$}\big\}$$ (see \cite{JON-Mus2012,JON-Mus2014}). 

In this section, we discuss the relation between the jumping numbers  and the Zhou valuations. The following remark shows that it suffices to consider the case $\varphi\equiv0$.

\begin{Remark}
	\label{rem:reduce to tri weight}
	It follows from Proposition \ref{thm:main_value1} that $(1+k_{\varphi_0})\Phi_{o,\max}$ is a local Zhou weight related to $|f_0|^2$. By definitions, we have 
	$$\sigma \big(\log|G|,(1+k_{\varphi_0})\Phi_{o,\max}\big)=\frac{\nu(G,\Phi_{o,\max})}{1+k_{\varphi_0}}$$
	and
	\begin{equation*}
		\begin{split}
			c^G_o\big((1+k_{\varphi_0})\Phi_{o,\max}\big) & := \sup\big\{c:|G|^{2}e^{-2c(1+k_{\varphi_0})\Phi_{o,\max}}\text{  is integrable near $o$}\big\}\\
			& =\sup\left\{\frac{c_1}{1+k_{\varphi_0}}:|G|^{2}e^{-2c_1\Phi_{o,\max}}\text{  is integrable near $o$}\right\}\\
			&=\frac{1}{1+k_{\varphi_0}}\sup\big\{c:|G|^{2}e^{-2c\Phi_{o,\max}}\text{  is integrable near $o$}\big\}\\
			&=\frac{1}{1+k_{\varphi_0}}c^G_o(\Phi_{o,\max}).
		\end{split}
	\end{equation*}
	Hence when discussing the relation between the jumping number $\nu(G,\Phi_{o,\max})$  and the Zhou valuation $\sigma(\cdot,\Phi_{o,\max})$, it suffices to consider $(1+k_{\varphi_0})\Phi_{o,\max}$, where $(1+k_{\varphi_0})\Phi_{o,\max}$ is a local Zhou weight related to $|f_{0}|^{2}$.
\end{Remark}

Let $\Phi_{o,\max}$ be any local Zhou weight related to $|f_{0}|^{2}$. Denote 
$$k_1:=\sigma(\log|f_0|,\Phi_{o,\max})$$
and 
$$k_2:=\nu(G,\Phi_{o,\max}).$$
We introduce following notations. Let us consider the following Tian functions:
$$\mathrm{Tn}_1(t):=\sup\big\{c:|f_0|^{2t}e^{-2c\Phi_{o,\max}}\text{  is integrable near $o$}\big\},$$
$$\mathrm{Tn}_2(t):=\sup\big\{c:|G|^{2t}e^{-2c\Phi_{o,\max}}\text{  is integrable near $o$}\big\},$$
and
$$\mathrm{Tn}_3(s,t):=\sup\big\{c:|G|^{2s}|f_0|^{2t}e^{-2c\Phi_{o,\max}}\text{  is integrable near $o$}\big\}.$$

Let $c_o(\Phi_{o,\max}):=\sup\big\{c:e^{-2c\Phi_{o,\max}}\text{  is integrable near $o$}\big\}$ be the complex singularity exponent of $\Phi_{o,\max}$ at $o$. It is easy to see that
$\mathrm{Tn}_3(0,t)=\mathrm{Tn}_1(t)$,
$\mathrm{Tn}_3(s,0)=\mathrm{Tn}_2(s)$,
and
$$\mathrm{Tn}_3(0,0)=\mathrm{Tn}_1(0)=\mathrm{Tn}_2(0)=c_o(\Phi_{o,\max}).$$

We note that $\mathrm{Tn}_1(t)$ is concave with respect to $t\in\mathbb{R}$. By the definitions of $\mathrm{Tn}_1(t)$ and $\Phi_{o,\max}$, we know $\mathrm{Tn}_{1}(1)=1$. It follows from Proposition \ref{prop:concave_B} that, when $t\ge 1$, $$\mathrm{Tn}_1(t)=k_1t-k_1+1.$$

We present an estimate for Tian function $\mathrm{Tn}_2(t)$ in the view of Zhou valuation.
\begin{Proposition}
	\label{pro:Jum and val}
	When $t\ge0$, we have $$c_o(\Phi_{o,\max})+k_2t=\mathrm{Tn}_2(0)+k_2t \le \mathrm{Tn}_2(t)\le k_2t-k_1+1,$$ where $k_1=\sigma(\log|f_0|,\Phi_{o,\max})$ and $k_2=\nu(G,\Phi_{o,\max})$.
\end{Proposition}

We need the following Lemma \ref{lem:estimate for A3} and \ref{lem:est of cur} in the proof of Proposition \ref{pro:Jum and val}.
\begin{Lemma}
	\label{lem:estimate for A3}Let $m\ge 0$ be fixed,  then we have
	\begin{equation}
		\label{estimate:A3sm2}
		\begin{split}
			\mathrm{Tn}_3(s,m+1)
			=1+mk_1+k_2s,
		\end{split}
	\end{equation}
	holds for any $s\ge0$.
\end{Lemma}
\begin{proof}
	Recall that $k_1:=\sigma(\log|f_0|,\Phi_{o,\max})$. It follows from Proposition \ref{thm:main_value1} that $\Phi^{f_0^{m+1}}_{o,\max}=(1+mk_1)\Phi_{o,\max}$ is a local Zhou weight related to $|f_0|^{2m+2}$ near $o$. Then we have
	\begin{equation}
		\label{estimate:A3sm1}
		\begin{split}
			&	\mathrm{Tn}_3(s,m+1) \\=& \sup\big\{c:|G|^{2s}|f_0|^{2(m+1)}e^{-2c\Phi_{o,\max}}\text{  is integrable near $o$}\big\}\\
			=&\sup\big\{c:|G|^{2s}|f_0|^{2m+2}e^{-2c\frac{1}{1+mk_1}\Phi^{f_0^{m+1}}_{o,\max}}
			\text{  is integrable near $o$}\big\}\\
			=&(1+mk_1)\sup\big\{c:|G|^{2s}|f_0|^{2m+2}e^{-2c\Phi^{f_0^{m+1}}_{o,\max}}
			\text{  is integrable near $o$}\big\}.
		\end{split}
	\end{equation}
	Let $$c_o\big(\Phi^{f_0^{m+1}}_{o,\max},s\log|G|\big):=\sup\big\{c:|G|^{2s}|f_0|^{2m+2}e^{-2c\Phi^{f_0^{m+1}}_{o,\max}}
	\text{  is integrable near $o$}\big\}.$$
	As $\Phi^{f_0^{m+1}}_{o,\max}=(1+mk_1)\Phi_{o,\max}$ and $k_2=\nu(G,\Phi_{o,\max})$, we know that $\nu(G,\Phi^{f_0^{m+1}}_{o,\max})=\frac{k_2}{1+mk_1}$.
	It follows from Proposition \ref{prop:concave_B} that, when $s\ge 0$,
	\begin{equation}
		\label{estimate:2m+2}
		\begin{split}
			c_o(\Phi^{f_0^{m+1}}_{o,\max},s\log|G|)
			=&c_o(\Phi^{f_0^{m+1}}_{o,\max},0)+\nu(G,\Phi^{f_0^{m+1}}_{o,\max})s\\
			=&1+ \frac{k_2}{1+mk_1}s.
		\end{split}
	\end{equation}
	Combining \eqref{estimate:A3sm1} and \eqref{estimate:2m+2}, when $m\ge 0$ is fixed,  we have that
	\begin{equation*}
		\begin{split}
			\mathrm{Tn}_3(s,m+1) &= (1+mk_1)\big(1+ \frac{k_2}{1+mk_1}s\big)\\
			&=1+mk_1+k_2s
		\end{split}
	\end{equation*}
	holds for any $s\ge0$. Lemma \ref{lem:estimate for A3} is proved.
\end{proof}

\begin{Lemma}
	\label{lem:est of cur}Let $t_1\ge t_2$ be two nonnegative integers. Let $H$ be a holomorphic function near $o$. Let $F$ be any holomorphic function near $o$ and $k:=\sigma(\log|F|,\Phi_{o,\max})$. Denote $$\mathrm{Tn}_F(t):=\sup\{c:|F|^{2t}|H|^2e^{-2c\Phi_{o,\max}}\text{  is integrable near $o$ }\}.$$
	Then we have $\mathrm{Tn}_F(t_1)\ge \mathrm{Tn}_F(t_2)+ k(t_1-t_2)$.
\end{Lemma}
\begin{proof}It follows from $k=\sigma(\log|F|,\Phi_{o,\max})$ and the definition of $\sigma(\log|F|,\Phi_{o,\max})$ that
	\begin{equation*}
		\log|F|\le k\Phi_{o,\max}+O(1)
	\end{equation*}
	holds near $o$, which implies that $|F|^{2(t_1-t_2)}e^{-2k(t_1-t_2)\Phi_{o,\max}}\le C_1$ for some positive constant $C_1$ on a neighborhood $U$ of $o$.
	
	Let $c$ be any real number such that $|F|^{2t_2}|H|^2e^{-2c\Phi_{o,\max}}$ is integrable near $o$. Then we have
	\begin{equation*}
		\begin{split}
			&|F|^{2t_1}|H|^2e^{-2k(t_1-t_2)\Phi_{o,\max}}e^{-2c\Phi_{o,\max}} \\
			=& |F|^{2t_2}|F|^{2(t_1-t_2)}|H|^2e^{-2k(t_1-t_2)\Phi_{o,\max}}e^{-2c\Phi_{o,\max}} \\
			\le& C_1 |F|^{2t_2}|H|^2e^{-2c\Phi_{o,\max}}
		\end{split}
	\end{equation*}
	on $U$.
	Hence $|F|^{2t_1}|H|^2e^{-2k(t_1-t_2)\Phi_{o,\max}}e^{-2c\Phi_{o,\max}}$ is integrable near $o$, and $\mathrm{Tn}_F(t_1)\ge c+k(t_1-t_2)$. By the choice of $c$, $\mathrm{Tn}_F(t_1)\ge \mathrm{Tn}_F(t_2)+k(t_1-t_2)$. Lemma \ref{lem:est of cur} is proved.		
\end{proof}

Now we prove Proposition \ref{pro:Jum and val}.
\begin{proof}[Proof of Proposition \ref{pro:Jum and val}]
	Note that $\mathrm{Tn}_2(0)=c_o(\Phi_{o,\max})$. It follows from Lemma \ref{lem:est of cur} that $\mathrm{Tn}_2(t)\ge \mathrm{Tn}_2(0)+k_2t=c_o(\Phi_{o,\max})+k_2t$. Hence the first inequality in Proposition \ref{pro:Jum and val} has been proved.
	
	We now prove the inequality $\mathrm{Tn}_2(t)\le k_2t-k_1+1$ by contradiction.
	If not, there exists a positive real number $t_0$ such that $\mathrm{Tn}_2(t_0)> k_2t_0-k_1+1$.
	
	Let $n=\lfloor t_0\rfloor$ be the largest integer less than or equal to $t_0$. Then it follows from Lemma \ref{lem:est of cur}  that $\mathrm{Tn}_3(n+1,0)=\mathrm{Tn}_2(n+1)\ge \mathrm{Tn}_2(t_0)+k_2(n+1-t_0)> k_2t_0-k_1+1+k_2(n+1-t_0)>k_2(n+1)-k_1+1$. Note that $k_1=\sigma(\log|f_0|,\Phi_{o,\max})$. Again, by Lemma \ref{lem:est of cur}, we have
	\begin{equation}
		\label{estimate:A3n+1,1}
		\begin{split}
			\mathrm{Tn}_3(n+1,1)\ge& \mathrm{Tn}_3(n+1,0)+k_1\\
			>& k_2(n+1)-k_1+1+k_1\\
			=& k_2(n+1)+1.
		\end{split}
	\end{equation}

	However, it follows from equality \eqref{estimate:A3sm2} that $\mathrm{Tn}_3(n+1,1)=k_2(n+1)+1$, which contradicts to inequality \eqref{estimate:A3n+1,1}. Hence $\mathrm{Tn}_2(t)\le k_2t-k_1+1$ holds when $t\in[0,+\infty)$.
	
	Proposition \ref{pro:Jum and val} has been proved.
\end{proof}

Now, we prove Theorem \ref{thm:valu-jump}.
\begin{proof}
	[Proof of Theorem \ref{thm:valu-jump}]
	It follows from Remark \ref{rem:reduce to tri weight} and Proposition \ref{pro:Jum and val} (taking $t=1$) that Theorem \ref{thm:valu-jump} holds.
\end{proof}

\section{Proofs of Theorem \ref{thm:multi-valua}, Corollary \ref{coro:multiplier-valuation} and Theorem \ref{thm:A}}
In this section, we prove Theorem \ref{thm:multi-valua}, Corollary \ref{coro:multiplier-valuation} and Theorem \ref{thm:A}.

Firstly, we consider the case of $u,v\ge N\log|z|$ for $N\gg0$.

\begin{Lemma}
	\label{l:0921-4}Let $u,$ $v$ be two plurisubharmonic functions near $o$ satisfying that there exists $N\gg0$ such that $u\ge N\log|z|$ and $v\ge N\log|z|$ near $o$. Assume that $\sigma(u,\Phi_{o,\max})\le\sigma(v,\Phi_{o,\max})$ for any local Zhou weight $\Phi_{o,\max}$ near $o$, then $\mathcal{I}(tv)_o\subset\mathcal{I}(tu)_o$ for any $t>0$.
\end{Lemma}
\begin{proof}
	Theorem \ref{thm:SOC} shows that for any holomorphic function $f$ and plurisubharmonic function $\varphi$ near $o$, $(f,o)\in\mathcal{I}(\varphi)_o$ if and only if $c_o^f(\varphi)>1$.
	We prove Lemma \ref{l:0921-4} by contradiction: if not, there exists a holomorphic function $f$ near $o$ such that 
	$$t_0:=c_o^f(u)<c_o^{f}(v).$$
	Then we have $|f|^2e^{-2t_0u}$ is not integrable near $o$ and $|f|^2|z|^{2t_0N}e^{-2t_0u}$ is  integrable near $o$ for $u\ge N\log|z|$.
	By Remark \ref{rem:max_existence}, there exists a local Zhou weight $\Phi_{o,\max}$ related to $|f|^2$ near $o$ satisfying that
	\begin{equation}
		\label{eq:0921b}\Phi_{o,\max}\ge t_0u
	\end{equation}
	and $|f|^2e^{-2\Phi_{o,\max}}$ is not integrable near $o$. Noting that $t_0u\le \sigma(t_0u,\Phi_{o,\max})\Phi_{o,\max}+O(1)$ near $o$, we have 
	$|f|^2e^{-2\frac{t_0}{ \sigma(t_0u,\Phi_{o,\max})}u}$ is not integrable near $o$. As $c_o^{f}(u)=t_0$, we have $\sigma(t_0u,\Phi_{o,\max})\le 1$. Combining inequality \eqref{eq:0921b}, we have $$\sigma(t_0u,\Phi_{o,\max})= 1.$$
	By the assumption of Lemma \ref{l:0921-4}, we have 
	$$\sigma(t_0v,\Phi_{o,\max})\ge 1,$$
	then $|f|^2e^{-2t_0v}$ is not integrable near $o$
	which contradicts to $c_o^f(v)>t_0$.
	
	Thus, Lemma \ref{l:0921-4} holds.
\end{proof}

Let us recall \emph{Demailly's approximation theorem}.

\begin{Lemma}[see \cite{demailly2010}]\label{l:appro-Berg}
	Let $D\subset\mathbb{C}^n$ be a bounded pseudoconvex domain, and let $\varphi\in \mathrm{PSH}(D)$. For any positive integer $m$, let $\{\sigma_{m,k}\}_{k=1}^{\infty}$ be an orthonormal basis of $A^2(D,2m\varphi):=\big\{f\in\mathcal{O}(D):\int_{D}|f|^2e^{-2m\varphi}d\lambda<+\infty\big\}$. Denote that 
	$$\varphi_m:=\frac{1}{2m}\log\sum_{k=1}^{\infty}|\sigma_{m,k}|^2$$
	on $D$. Then there exist two positive constants $c_1$ (depending only on $n$ and diameter of $D$) and $c_2$ such that 
	\begin{equation}
		\label{eq:0911b}\varphi(z)-\frac{c_1}{m}\le\varphi_m(z)\le\sup_{|\tilde z-z|<r}\varphi(\tilde z)+\frac{1}{m}\log\frac{c_2}{r^n}
	\end{equation}
	for any $z\in D$ satisfying $\{\tilde z\in\mathbb{C}^n:|\tilde z-z|<r\}\subset\subset D$. Especially, $\varphi_m$ converges to $\varphi$ pointwisely and in $L^1_{\mathrm{loc}}$ on $D$.
\end{Lemma}

\begin{Remark}\label{r:0930a}
	Let $(\tau_l)$ be an orthonormal basis of a closed subspace $H$ of the space $A^2(\Omega, 2m\varphi)$, then we can see that
	\[\sum_l|\tau_l(z)|^2=\sup\left\{|f(z)|^2 : f\in H \ \& \ \int_{\Omega}|f|^2e^{-2m\varphi}\le 1\right\}\]
	for any $z\in \Omega$.
\end{Remark}

The following two lemmas will be used in the proofs of Theorem \ref{thm:multi-valua} and Corollary \ref{coro:multiplier-valuation}, and we prove them in Appendix.

\begin{Lemma}
	\label{l:0921-3}
	For any plurisubharmonic function $\varphi$ near $o$ and any holomorphic function $f$ near $o$, we have $\lim_{N\rightarrow+\infty}c_o^f(\max\{\varphi,N\log|z|\})=c_o^f(\varphi)$. 
\end{Lemma}

\begin{Lemma}
	\label{l:sigma-analytic}
	Let $u,$ $v$ be two  plurisubharmonic functions near $o$. Assume that $u$ has  analytic singularities near $o$. If $v\le(1-\epsilon)u+O(1)$ near $o$ for any $\epsilon>0$, then $v\le u+O(1)$ near $o$.
\end{Lemma}

Now, we prove Theorem \ref{thm:multi-valua}.

\begin{proof}[Proof of Theorem \ref{thm:multi-valua}]
	Note that  $(5)\Rightarrow(4)$ is trivial. Then we only need to prove that $(1) \Rightarrow (3)$, $(3) \Rightarrow (2)$, $(2) \Rightarrow (5)$ and $(4) \Rightarrow (1)$.
	We prove Theorem \ref{thm:multi-valua} in four steps.
	
	\textbf{Step 1}. $(1) \Rightarrow (3)$.
	
	Assume that the statement $(1)$ holds. Then there exists a plurisubharmonic function $\varphi_{0}$ ($\not\equiv-\infty$) near $o$ and two sequences of numbers $\{t_{i,j}\}_{j\in\mathbb{Z}_{\ge0}}$ ($t_{i,j}\rightarrow+\infty$ when $j\rightarrow+\infty$, $i=1,2$) such that $\lim_{j\rightarrow+\infty}\frac{t_{1,j}}{t_{2,j}}=1$ and $$\mathcal{I}(\varphi_{0}+t_{1,j}v)_o\subset\mathcal{I}(\varphi_{0}+t_{2,j}u)_o$$ for any $j$. Let $\{f_{j,1},\ldots,f_{j,m_j}\}$ be the generators set of $\mathcal{I}(\varphi_{0}+t_{1,j}v)_o$. Then it follows from Demailly's approximation theorem (Lemma \ref{l:appro-Berg}) that 
	\begin{equation}
		\label{eq:0921a}
		v+\frac{\varphi_{0}}{t_{1,j}}\le \frac{1}{2t_{1,j}}\log\sum_{1\le l\le m_j}|f_{j,l}|^2+O(1)
	\end{equation}
	near $o$.
	Let $\Phi_{o,\max}$ be any local Zhou weight near $o$. Note that $\mathcal{I}(\varphi_{0}+t_{1,j}v)_o\subset\mathcal{I}(\varphi_{0}+t_{2,j}u)_o$. We have 
	$$\int_U\left(\sum_{1\le l\le m_j}|f_{j,l}|^2\right)e^{-2{t_{2,j}}u-2\varphi_{0}}<+\infty$$
	for a neighborhood $U$ of $o$. Then
	\begin{equation}
		\label{eq:1023a}
		\int_U\left(\sum_{1\le l\le m_j}|f_{j,l}|^2\right)e^{-2t_{2,j}\sigma(u,\Phi_{o,\max})\Phi_{o,\max}-2\varphi_{0}}<+\infty.
	\end{equation}
As $\sup_{U_1}\varphi_{0}<+\infty$ for some neighborhood $U_1$ of $o$, it follows from Theorem \ref{thm:valu-jump} and inequality \eqref{eq:1023a} that
	\begin{equation}
		\label{eq:1023b}
		\sigma\left(\log\sum_{1\le l\le m_j}|f_{j,l}|,\Phi_{o,\max}\right)\ge t_{2,j}\sigma(u,\Phi_{o,\max})-C_1,
	\end{equation}
 where $C_1$ is a constant independent of $j$. Combining inequality \eqref{eq:0921a}, we have 
	$$v+\frac{\varphi_{0}}{t_{1,j}}\le \frac{1}{2t_{1,j}}\log\sum_{1\le l\le m_j}|f_{j,l}|^2+O(1)\le\left(\frac{t_{2,j}}{t_{1,j}}\sigma(u,\Phi_{o,\max})-\frac{C_1}{t_{1,j}}\right)\Phi_{o,\max}+O(1) $$
	near $o$, which shows that 
	\[\sigma(v,\Phi_{o,\max})+\frac{1}{t_{1,j}}\sigma(\varphi_{0},\Phi_{o,\max})=\sigma(v+\frac{\varphi_{0}}{t_{1,j}},\Phi_{o,\max})\ge\frac{t_{2,j}}{t_{1,j}}\sigma(u,\Phi_{o,\max})-\frac{C_1}{t_{1,j}}\]
	by Corollary \ref{coro:main}. 
	As $\lim_{j\rightarrow+\infty}\frac{t_{1,j}}{t_{2,j}}=1$, we get $$\sigma(v,\Phi_{o,\max})\ge \sigma(u,\Phi_{o,\max})$$ by letting $j\rightarrow+\infty$.
	
	Thus, $(1) \Rightarrow (3)$ has been proved.
	
	\textbf{Step 2.} $(3) \Rightarrow (2)$.
	
	Using the statement $(3)$ and Corollary \ref{coro:main}, we get that
	\begin{equation}
		\nonumber
		\begin{split}
			&\sigma(\varphi_0+tu,\Phi_{o,\max})\\
			=&\sigma(\varphi_0,\Phi_{o,\max})+t\sigma(u,\Phi_{o,\max})\\
			\le&\sigma(\varphi_0,\Phi_{o,\max})+t\sigma(v,\Phi_{o,\max})\\
			=&\sigma(\varphi_0+tv,\Phi_{o,\max})
		\end{split}
	\end{equation} 
	holds for any local Zhou weight $\Phi_{o,\max}$ near $o$ and any $t>0$. Then it suffices to prove that: if $\sigma(u,\Phi_{o,\max})\le\sigma(v,\Phi_{o,\max})$ holds for any local Zhou weight $\Phi_{o,\max}$ near $o$, then $\mathcal{I}(v)_o\subset\mathcal{I}(u)_o.$
	
	Note that Theorem \ref{thm:SOC} shows that for any holomorphic function $f$ and plurisubharmonic function $\varphi$ near $o$, $(f,o)\in\mathcal{I}(\varphi)_o$ if and only if $c_o^f(\varphi)>1$.
	For any $N>0$, it follows from the statement $(3)$ and Corollary \ref{coro:main} that 
	\begin{equation}
		\nonumber
		\begin{split}
			&\sigma\big(\max\{u,N\log|z|\},\Phi_{o,\max}\big)\\
			=&\min\big\{\sigma(u,\Phi_{o,\max}),\sigma(N\log|z|,\Phi_{o,\max})\big\}\\
			\le&\min\big\{\sigma(v,\Phi_{o,\max}),\sigma(N\log|z|,\Phi_{o,\max})\big\}\\
			=&\sigma\big(\max\{v,N\log|z|\},\Phi_{o,\max}\big)
		\end{split}
	\end{equation}
	for any  local Zhou weight $\Phi_{o,\max}$ near $o$.
	Using Lemma \ref{l:0921-4}, we have
	\[\mathcal{I}\big(t\max\{v,N\log|z|\}\big)_o\subset\mathcal{I}\big(t\max\{u,N\log|z|\}\big)_o\]
	for any $t>0$ and any $N>0$, which implies that 
	$$c_o^f\big(\max\{u,N\log|z|\}\big)\ge c_o^f\big(\max\{v,N\log|z|\}\big)$$
	for any $N>0$ and any holomorphic function $f$ near $o$.
	By Lemma \ref{l:0921-3}, we have
	$$c_o^f(u)\ge c_o^f(v)$$ for  any holomorphic function $f$ near $o$,  which implies $\mathcal{I}(tv)_o\subset \mathcal{I}(tu)_o$ for any $t>0$.

	\textbf{Step 3}. $(2) \Rightarrow (5)$.	
	
	For any $c<\mathrm{Tn}(t;f,v,\varphi_{0})$, $|f|^{2t}e^{-2cv-2\varphi_{0}}=|f|^{2\lceil t\rceil}e^{-2cv-2\varphi_{0}-2(\lceil t\rceil-t)\log|f|}$ is integrable near $o$, where $\lceil m\rceil:=\min\{n\in \mathbb{Z}:n\geq m\}$. Following from 
	$$\mathcal{I}(cv+\varphi_{0}+(\lceil t\rceil-t)\log|f|)_o\subset \mathcal{I}(cu+\varphi_{0}+(\lceil t\rceil-t)\log|f|)_o,$$ 
	we obtain that
	$|f|^{2t}e^{-2cv-2\varphi_{0}}$ is integrable near $o$, i.e., 
	$$c<\mathrm{Tn}(t;f,u,\varphi_{0}).$$
	Thus, we have
	$$\mathrm{Tn}(t;f,v,\varphi_{0})\le \mathrm{Tn}(t;f,u,\varphi_{0}).$$
	
	\textbf{Step 4}. $(4) \Rightarrow (1)$.	
	
	For any holomorphic function $f$ and plurisubharmonic function $\psi$ near $o,$ $(f,o)\in\mathcal{I}(\psi)_o$ if and only if the jumping number $c_o^{f}(\psi)>1$. Thus, statement $(4)$ can deduce statement $(1)$ by the definition of  Tian functions and by taking $t_{i,j}=j$ for any $i=1,2$ and $j\in\mathbb{Z}_{\ge0}$.

	Thus, Theorem \ref{thm:multi-valua} holds.
\end{proof}

\begin{Remark}
	\label{r:1-3}
	Note that ``inequality \eqref{eq:1023a} $\Rightarrow$ inequality \eqref{eq:1023b}" also holds when replacing $\Phi_{o,\max}$ with any tame maximal weight. Thus, the Step 1 in the above proof shows that statement $(1)$ implies $\sigma(u,\varphi)\le\sigma(v,\varphi)$ for any tame maximal weight $\varphi$.
\end{Remark}

Next, we prove Corollary \ref{coro:multiplier-valuation}.

\begin{proof}[Proof of Corollary \ref{coro:multiplier-valuation}]
	It follows from Theorem \ref{thm:multi-valua} that it suffices to prove $(3)\Rightarrow(1)$. Assume that $\sigma(u,\Phi_{o,\max})\le\sigma(v,\Phi_{o,\max})$ holds for any local Zhou weight $\Phi_{o,\max}$ near $o$. Theorem \ref{thm:multi-valua} shows that $$\mathcal{I}(tv)_o\subset\mathcal{I}(tu)_o$$ for any $t>0$. Let $D$ be a small enough neighborhood of $o$, and let $\{\sigma_{m,k}\}_{k=1}^{\infty}$ be an orthonormal basis of $A^2(D,2mu):=\big\{f\in\mathcal{O}(D):\int_{D}|f|^2e^{-2mu}d\lambda<+\infty\big\}$. Without loss of generality, assume that $e^u$ is smooth on $D$. Denote that 
	$$u_m:=\frac{1}{2m}\log\sum_{k=1}^{\infty}|\sigma_{m,k}|^2$$
	on $D$. Then it follows from Lemma \ref{l:appro-Berg} and  $\mathcal{I}(mv)_o\subset\mathcal{I}(mu)_o$ that
	$$v\le u_m+O(1)$$
	near $o$ and
	\begin{equation}
		\label{eq:0923a}
		u_m(z)\le\sup_{|\tilde z-z|<r}u(\tilde z)+\frac{1}{m}\log\frac{c_2}{r^n}
	\end{equation}
	for any $z\in D$ satisfying $\{\tilde z\in\mathbb{C}^n:|\tilde z-z|<r\}\subset\subset D$.  As $e^{u}$ is smooth on $D$, we have $\sup_{|\tilde z-z|<r}e^{u(\tilde z)}\le e^{u(z)}+Cr$. Taking $r=e^{u(z)}$, inequality \eqref{eq:0923a} implies that 
	$$u_m\le \left(1-\frac{n}{m}\right)u+O(1)$$
	near $o$, which implies $v \le \big(1-\frac{n}{m}\big)u+O(1)$ near $o$ for any $m$. It follows from Lemma \ref{l:sigma-analytic} that 
	$$v\le u+O(1)$$
	near $o$. Thus, Corollary \ref{coro:multiplier-valuation} holds.
\end{proof}

Finally, we prove Theorem \ref{thm:A}.

\begin{proof}[Proof of Theorem \ref{thm:A}]
	Note that $|I|^2e^{-2\Phi_{o,\max}}$ is not integrable near $o$ for any $\Phi_{o,\max}\in V_I$. By definition of $c_o^{I}(\varphi)$, we have 
	$\sigma(c_o^{I}(\varphi)\varphi,\Phi_{o,\max})\le 1$, which implies that
	$$c_o^{I}(\varphi)\le\frac{1}{\sup_{\Phi_{o,\max}\in V_I}\sigma(\varphi,\Phi_{o,\max})}.$$
	
	In the following, we prove $c_o^{I}(\varphi)=\frac{1}{\sup_{\Phi_{o,\max}\in V_I}\sigma(\varphi,\Phi_{o,\max})}$ by contradiction. If not, we have 
	$$c_o^{I}(\varphi)<\frac{1}{\sup_{\Phi_{o,\max}\in V_I}\sigma(\varphi,\Phi_{o,\max})}.$$
	Then
	there exists $p<1$ such that 
	\begin{equation}
		\label{eq:0220a}
		\sigma(c_o^I(\varphi)\varphi,\Phi_{o,\max})<p
	\end{equation}
	for any local Zhou weight $\Phi_{o,\max}$ related to $|I|^2$.
	Using Lemma \ref{l:0921-3}, there exists $N\gg0$ such that
	\begin{equation}
		\label{eq:0220b}
		c_o^I(\max\{\varphi,N\log|z|\})\le\frac{c_o^{I}(\varphi)}{p}.
	\end{equation}
	Note that $|I|^2e^{-2c_o^I(\max\{\varphi,N\log|z|\})\max\{\varphi,N\log|z|\}}$ is not integrable near $o$. Remark \ref{rem:max_existence} shows that there exists a local Zhou weight $\Phi_{o,\max}$ related to $|I|^2$ satisfying that 
	$$\Phi_{o,\max}\ge c_o^I(\max\{\varphi,N\log|z|\})\max\{\varphi,N\log|z|\}+O(1)$$
	near $o$.
	Following from inequality \eqref{eq:0220b}, we have
	$$\varphi\le\max\{\varphi,N\log|z|\}\le \frac{\Phi_{o,\max}}{c_o^I(\max\{\varphi,N\log|z|\})}+O(1)\le\frac{p}{c_o^I(\varphi)}\Phi_{o,\max}+O(1)$$
	near $o$, which contradicts to inequality \eqref{eq:0220a}.
	Thus, we get $c_o^{I}(\varphi)=\frac{1}{\sup_{\Phi_{o,\max}\in V_I}\sigma(\varphi,\Phi_{o,\max})}.$
\end{proof}

\section{Global Zhou weights}

In this section, we discuss the global Zhou weights, and prove Proposition \ref{l:max2} and Proposition \ref{thm:contin}.

\subsection{Some properties of the global Zhou weights}

Let $D$ be a domain in $\mathbb{C}^n$, such that the origin $o\in D$.	Let $f_{0}=(f_{0,1},\cdots,f_{0,m})$ be a vector,
where $f_{0,1},\cdots,f_{0,m}$ are holomorphic functions near $o$.
Denote $|f_{0}|^{2}=|f_{0,1}|^{2}+\cdots+|f_{0,m}|^{2}$.

Let $\varphi_{0}$ be a plurisubharmonic function near $o$,
such that $|f_{0}|^{2}e^{-2\varphi_{0}}$ is integrable near $o$. Let us recall the definition of global Zhou weights.

\begin{Definition}
	We call a negative plurisubharmonic function $\Phi^{f_0,\varphi_0,D}_{o,\max}$ ($\Phi^{D}_{o,\max}$ for short) on $D$ a global Zhou weight related to $|f_0|^2e^{-2\varphi_0}$ if the following statements hold:
	
	$(1)$ $|f_0|^2e^{-2\varphi_0}|z|^{2N_0}e^{-2\Phi^{D}_{o,\max}}$ is integrable near $o$ for large enough $N_0$;
	
	$(2)$ $|f_0|^2e^{-2\varphi_0-2\Phi^{D}_{o,\max}}$ is not integrable near $o$;
	
	$(3)$ for any negative plurisubharmonic function $\tilde\varphi$ on $D$ satisfying that $\tilde\varphi\ge\Phi^{D}_{o,\max}$ on $D$ and $|f_0|^2e^{-2\varphi_0-2\tilde\varphi}$ is not integrable near $o$, $\tilde\varphi=\Phi^{D}_{o,\max}$ holds on $D$.
\end{Definition}

The following remark shows the existence of the global Zhou weights.

\begin{Remark}\label{r:exists}
	Assume that there exists a negative plurisubharmonic function $\varphi$ on $D$ such that $|f_0|^2e^{-2\varphi_0-2\varphi}|z|^{2N_0}$ is integrable near $o$ for large enough $N_0$ and $(f_0,o)\not\in\mathcal{I}(\varphi+\varphi_0)_o.$
	
	Then there exists a global Zhou weight $\Phi^{D}_{o,\max}$ on $D$ related to $|f_0|^2e^{-2\varphi_0}$ such that $\Phi^{D}_{o,\max}\ge\varphi$ on $D$.
	
\end{Remark}

\begin{proof}
	Let $(\varphi_{\alpha})_{\alpha}$ be the negative plurisubharmonic functions on $V$
	such that $\varphi_{\alpha}\geq\varphi$ and $|f_0|^{-2\varphi_0-2\varphi_{\alpha}}$ is not integrable near $o$.
	
	Zorn's Lemma shows that
	there exists $\Gamma$ which is the maximal set such that for any $\alpha,\alpha'\in\Gamma$,
	$\varphi_{\alpha}\leq \varphi_{\alpha'}$ or $\varphi_{\alpha'}\leq \varphi_{\alpha}$ holds on $D$,
	where $(\varphi_{\alpha})$ are negative plurisubharmonic functions on $D$.
	
	Let $u(z)=\sup_{\alpha\in\Gamma}\varphi_{\alpha}(z)$,
	and let $u^{*}(z)=\lim_{\varepsilon\to0}\sup_{\mathbb{B}^{n}(z,\varepsilon)}u$.
	Lemma \ref{lem:Choquet} shows that there exists increasing subsequence $(\varphi_{j})$ of $(\varphi_{\alpha})$ such that
	$(\lim_{j\rightarrow+\infty}\varphi_{j})^{*}=u^{*}$.
	Proposition \ref{pro:Demailly} shows that
	$(\varphi_{j})$ is convergent to $u^{*}$ with respect to $j$ almost everywhere with respect to Lebesgue measure,
	and $u^{*}$ is a plurisubharmonic function on $D$.
	Proposition \ref{p:effect_GZ} (and Remark \ref{rem:effect_GZ})
	shows that $|f_0|^2e^{-2\varphi_0-2u^*}$ is not integrable near $o$.
	In fact, the definition of $u^*$ shows that $\Phi^{D}_{o,\max}:=u^*$ is a global Zhou weight on $D$ related to $|f_0|^2e^{-2\varphi_0}$.
\end{proof}

Denote that 
$$L_w:=\big\{u\in \mathrm{PSH}(D):u<0 \ \& \ \limsup_{z\rightarrow w}(u(z)-\log|z-w|)<+\infty\big\}.$$
If $L_w\not=\emptyset$, the pluricomplex Green function of $D$ with a pole at $w$ is defined as follows:
$$G_D(w,\cdot):=\sup\{u:u\in L_w\}.$$
We give a relation between global Zhou weights $\Phi^{D}_{o,\max}$ and the pluricomplex Green functions $G_D(o,\cdot)$.

\begin{Lemma}
	\label{l:max>green}
	Assume that $L_o\not=\emptyset$. Then
	$\Phi^{D}_{o,\max}\ge NG_D(o,\cdot)$ for some $N\gg0$.
\end{Lemma}

\begin{proof}	
	As  $|f_0|^2e^{-2\varphi_0-2\Phi^{D}_{o,\max}}|z|^{2N_0}$ is integrable near $o$, it follows from Theorem \ref{thm:SOC} that  
	$$|f_0|^2e^{-2\varphi_0}|z|^{2N_0}e^{-2(1+\epsilon)\Phi^{D}_{o,\max}}$$
	is integrable near $o$. Note that
	$\sup_{z\in U}\big(G_D(o,z)-\log|z|\big)<+\infty$ for a neighborhood $U$ of $o$, then 
	$$|f_0|^2e^{-2\varphi_0}e^{2N_0G_D(o,\cdot)}e^{-2(1+\epsilon)\Phi^{D}_{o,\max}}$$
	is integrable near $o$. Lemma \ref{lem:G_key} shows that
	$$|f_0|^2e^{-2\varphi_0}e^{-2\Phi^{D}_{o,\max}}-|f_0|^2e^{-2\varphi_0}e^{-2\max\big\{\Phi^{D}_{o,\max},\frac{N_0}{\epsilon}G_D(o,\cdot)\big\}}$$
	is integrable near $o$. Then it follows from the definition of $\Phi^{D}_{o,\max}$ that 
	$$\Phi^{D}_{o,\max}\ge\frac{N_0}{\epsilon}G_D(o,\cdot).$$
	Lemma \ref{l:max>green} has been proved.
\end{proof}

Denote that 
$$\tilde L_o:=\big\{u\in L_o: u\in L^{\infty}_{\mathrm{loc}}(U\backslash\{o\})\text{ for some neighborhood $U$ of $o$}\big\}.$$
\begin{Remark}\label{r:8.4}
	If $D$ is bounded or hyperconvex, then $\tilde L_o\not=\emptyset$.
\end{Remark}
\begin{proof}
	If $D$ is bounded, $\log|z|-C\in L_o$ for large enough $C>0$. Assume that $D$ is hyperconvex, then there exists a continuous exhaustion plurisubharmonic function $\varphi_1<0$ on $D$. For $0<c_1<c_2$ satisfying $\{|z|<c_2\}\subset\subset D$, there exist two constants $a>0$ and $b\in\mathbb{R}$ such that $a\varphi_1+b<\log|z|$ on $\{|z|=c_1\}$ and $a\varphi_1+b>\log|z|$ on $\{|z|=c_2\}$. Let
	\begin{equation*}
		\tilde{\varphi}:=\left\{
		\begin{array}{ll}
			\log|z|-b & \ \text{on} \ \{|z|\le c_1\}, \\
			\max\{a\varphi_1,\log|z|-b\} & \ \text{on} \ \{c_1<|z|<c_2\},\\
			a\varphi_1 & \ \text{on} \ D\backslash\{|z|<c_2\}.
		\end{array}
		\right.
	\end{equation*}
	Then $\tilde\varphi$ is a plurisubharmonic function and $\tilde\varphi\in \tilde L_o\cap L_o$.
\end{proof}

Given any local Zhou weight, we can obtain a global Zhou weight on $D$ by using the following lemma.

\begin{Lemma}\label{l:local-glabal}
	Let $\Phi_{o,\max}$ be a local Zhou weight related to $|f_0|^2e^{-2\varphi_0}$ near $o$, and denote that
	\begin{equation}
		\nonumber
		\begin{split}
			L(\Phi_{o,\max}):=&\big\{\tilde\varphi(z)\in \mathrm{PSH}(D):\tilde\varphi<0\\
			&\ \& \ |f_0|^2e^{-2\varphi_0-2\tilde\varphi}\text{ is not integrable near $o$} \ \& \ \tilde\varphi\ge\Phi_{o,\max}+O(1)\text{ near $o$}\big\}.
		\end{split}
	\end{equation}
	If $L(\Phi_{o,\max})\not=\emptyset$, then 
	\begin{equation}
		\label{eq:0918a}\Phi^D_{o,\max}(z):=\sup\big\{\tilde\varphi(z):\tilde\varphi\in L(\Phi_{o,\max})\big\}, \ \forall z\in D
	\end{equation}
	is a global Zhou weight related to $|f_0|^2e^{-2\varphi_0}$ on $D$ satisfying that 
	$$\Phi^D_{o,\max}=\Phi_{o,\max}+O(1)$$ near $o$.
\end{Lemma}
\begin{proof} 
	Choosing any $\tilde\varphi_1\in L(\Phi_{o,\max})$, it follows from Remark \ref{r:exists} that there exists a global Zhou weight $\tilde\varphi_1^*$ related to $|f_0|^2e^{-2\varphi_0}$ on $D$ satisfying that $\tilde\varphi_1^*\ge\tilde\varphi_1$ on $D$ and 
	$$\tilde\varphi_1^*\in L(\Phi_{o,\max}).$$
	As $\Phi_{o,\max}$ is a local Zhou weight related to $|f_0|^2e^{-2\varphi_0}$ near $o$, we have that 
	$$\tilde\varphi=\Phi_{o,\max}+O(1)$$
	near $o$ for any $\tilde\varphi\in L(\Phi_{o,\max})$, which shows that  $|f_0|^2e^{-2\varphi_0}e^{-2\max\{\tilde\varphi,\tilde\varphi_1^*\}}$ is not integrable near $o$. By the definition of global Zhou weight, we have $\tilde\varphi_1^*\ge\tilde\varphi$. Thus, $\tilde\varphi_1^*=\Phi^D_{o,\max}:=\sup\big\{\tilde\varphi:\tilde\varphi\in L(\Phi_{o,\max})\big\}$ is a global Zhou weight related to $|f_0|^2e^{-2\varphi_0}$ on $D$ satisfying that $\Phi^D_{o,\max}=\Phi_{o,\max}+O(1)$ near $o$.
\end{proof}

We give a sufficient and necessary condition for $L(\Phi_{o,\max})\not=\emptyset.$
\begin{Remark}\label{r:L tildeL}
	$L(\Phi_{o,\max})\not=\emptyset$ if and only if $\tilde L_o\not=\emptyset$.
\end{Remark}
\begin{proof}
	For any $\tilde\varphi\in L(\Phi_{o,\max})$, there exists two positive constants $N_1$ and $N_2$ such that $N_2\log|z|+O(1)\ge\tilde\varphi\ge N_1\log|z|+O(1)$ near $o$, which implies that $\tilde L_o\not=\emptyset.$
	
	For any $u\in \tilde L_o$, note that $\Phi_{o,\max}\ge N\log|z|+O(1)\ge Nu+O(1)$
	near $o$ for large enough $N\gg0$. Letting $r>0$ small enough, there exists a constant $C_1$ such that $\Phi_{o,\max}+C_1<Nu$ on a neighborhood of $\{|z|=r\}$. Define
	\begin{equation*}
		\tilde{\varphi}:=\left\{
		\begin{array}{ll}
			\max\big\{\Phi_{o,\max}+C_1,Nu\big\} & \ \text{on} \ \{|z|<r\}, \\
			Nu & \ \text{on} \ D\backslash\{|z|<r\}.
		\end{array}
		\right.
	\end{equation*}
	Then $\tilde\varphi$ is a negative plurisubharmonic function on $D$ satisfying that 
	\begin{equation*}
		\tilde\varphi=\Phi_{o,\max}+O(1),
	\end{equation*}
	which shows that $\tilde\varphi\in  L(\Phi_{o,\max})$. Thus, Remark \ref{r:L tildeL} holds.
\end{proof}

Remark \ref{r:8.4}, Lemma \ref{l:local-glabal} and Remark \ref{r:L tildeL} show that if $D$ is bounded or hyperconvex, given any local Zhou weight, we can obtain a global Zhou weight by using equality \eqref{eq:0918a}. On the other hand, the following lemma shows that any global Zhou weight is also a local Zhou weight.

\begin{Lemma}\label{l:max-loc}
	Let  $\Phi^{D}_{o,\max}$ be a global Zhou weight related to $|f_0|^2e^{-2\varphi_0}$ on $D$. Then $\Phi^{D}_{o,\max}$ is a local Zhou weight related to $|f_0|^2e^{-2\varphi_0}$ near $o$ if  and only if  $\tilde L_o\not=\emptyset$. 
\end{Lemma}
\begin{proof} If $\Phi^{D}_{o,\max}$ is a local Zhou weight related to $|f_0|^2e^{-2\varphi_0}$ near $o$, then $\Phi^{D}_{o,\max}\ge N\log|z|+O(1)$ for large enough $N\gg0$, i.e., $\Phi^{D}_{o,\max}\in\tilde L_o\not=\emptyset$.
	
	In the following, assume that $\tilde L_o\not=\emptyset$.
	By Remark \ref{rem:max_existence}, there exists a local Zhou weight $\Phi_{o,\max}$ related to $|f_0|^2e^{-2\varphi_0}$ near $o$ satisfying that $$\Phi_{o,\max}\ge \Phi_{o,\max}^D$$ near $o$. Remark \ref{r:L tildeL} tells us that there exists a negative plurisubharmonic function $\tilde\varphi$ on $D$ satisfying that $\tilde\varphi\ge \Phi_{o,\max}+O(1)$ near $o$ and $|f_0|^2e^{-2\varphi_0-2\tilde\varphi}$ is not integrable near $o$. Hence,
	\begin{equation}\nonumber
		\tilde\varphi\ge\Phi_{o,\max}+O(1)\ge\Phi_{o,\max}^D+O(1),
	\end{equation}
	and $|f_0|^2e^{-2\varphi_0-2\max\{\tilde\varphi,\Phi_{o,\max}^D\}}$ is not integrable near $o$. 
	By the definition of $\Phi_{o,\max}^D$, we have 
	$$\Phi_{o,\max}^D=\max\big\{\tilde\varphi,\Phi_{o,\max}^D\big\}=\tilde\varphi,$$
	which implies that $\Phi_{o,\max}^D=\Phi_{o,\max}+O(1)$ is a  local Zhou weight related to $|f_0|^2e^{-2\varphi_0}$ near $o$.
\end{proof}

The following lemma gives a global inequality with the Zhou number $\sigma(\cdot,\Phi^{D}_{o,\max})$.

\begin{Lemma}\label{r:value-inequa}
	Let  $\Phi^{D}_{o,\max}$ be a global Zhou weight related to $|f_0|^2e^{-2\varphi_0}$ on $D$, and let $\psi$ be any negative plurisubharmonic function $D$. Then the inequality 
	$$\psi\le\sigma(\psi,\Phi^{D}_{o,\max})\Phi^{D}_{o,\max}$$
	holds on $D$.
\end{Lemma}
\begin{proof}
	For any $p\in\big(0,\sigma(\psi,\Phi^{D}_{o,\max})\big)$, $\psi\le p\Phi^{D}_{o,\max}+O(1)$ near $o$, which shows that $|f_0|^2e^{-2\varphi_0-2\max\{\Phi^{D}_{o,\max},\frac{1}{p}\psi\}}$ is not integrable near $o$. Following from Theorem \ref{thm:SOC}, $|f_0|^2e^{-2\varphi_0-2\max\big\{\Phi^{D}_{o,\max},\frac{1}{\sigma(\psi,\Phi^{D}_{o,\max})}\psi\big\}}$ is not integrable near $o$. Note that $\max\big\{\Phi^{D}_{o,\max},\frac{1}{\sigma(\psi,\Phi^{D}_{o,\max})}\psi\big\}<0$ on $D$, then the definition of $\Phi^{D}_{o,\max}$ shows that 
	$$\Phi^{D}_{o,\max}\ge\max\left\{\Phi^{D}_{o,\max},\frac{1}{\sigma(\psi,\Phi^{D}_{o,\max})}\psi\right\}, $$
	which implies that 
	$$\psi\le\sigma(\psi,\Phi^{D}_{o,\max})\Phi^{D}_{o,\max}$$
	holds on $D$.
\end{proof}

We give a  global Zhou weight on the sublevel set $\{\Phi^{D}_{o,\max}<-t\}$.

\begin{Lemma}\label{l:subset}
	Let  $\Phi^{D}_{o,\max}$ be a global Zhou weight related to $|f_0|^2e^{-2\varphi_0}$ on $D$, then for any $t>0$, $\Phi^{D}_{o,\max}+t$ is a global Zhou weight related to $|f_0|^2e^{-2\varphi_0}$ on $\{\Phi^{D}_{o,\max}<-t\}$.
\end{Lemma}
\begin{proof}
	By Remark \ref{r:exists}, there exists a global Zhou weight $\Phi_{o,\max}^{\{\Phi^{D}_{o,\max}<-t\}}$ related to $|f_0|^2e^{-2\varphi_0}$ on  $\{\Phi^{D}_{o,\max}<-t\}$ satisfying that
	\[\Phi_{o,\max}^{\{\Phi^{D}_{o,\max}<-t\}}\ge\Phi^{D}_{o,\max}+t\]
	on $\{\Phi^{D}_{o,\max}<-t\}$. 
	
	Let
	\begin{equation*}
		\phi:=\left\{
		\begin{array}{ll}
			\Phi_{o,\max}^{\{\Phi^{D}_{o,\max}<-t\}}-t & \ \text{on} \ \{\Phi^{D}_{o,\max}<-t\}, \\
			\Phi^{D}_{o,\max} & \ \text{on} \ \{\Phi^{D}_{o,\max}\ge-t\}.
		\end{array}
		\right.
	\end{equation*}
	Then we have $\phi\ge\Phi^{D}_{o,\max}$ on $D$. As $\Phi_{o,\max}^{\{\Phi^{D}_{o,\max}<-t\}}$ is plurisubharmonic on $\{\Phi^{D}_{o,\max}<-t\}$ and $\Phi^{D}_{o,\max}$ is plurisubharmonic on $D$, we have
	\[\limsup_{\tilde z\rightarrow z} \phi(\tilde{z})\le\phi(z), \  \forall z\in D\backslash\partial\{\Phi^{D}_{o,\max}<-t\}.\]
	For any $z\in \partial\{\Phi^{D}_{o,\max}<-t\}\cap D$,  $\phi(z)=\Phi^{D}_{o,\max}(z)\ge-t$, then
	\begin{flalign*}
		\begin{split}
			\limsup_{\tilde z\rightarrow z}\phi(\tilde{z})&=\max\left\{\limsup_{\{\Phi^{D}_{o,\max}<-t\}\ni\tilde z\rightarrow z}\phi(\tilde{z}),\limsup_{(D\backslash\{\Phi^{D}_{o,\max}<-t\})\ni\tilde z\rightarrow z}\phi(\tilde{z})\right\}\\
			&\le-t\le\phi(z).
		\end{split}
	\end{flalign*}
	Thus, $\phi$ is an upper semicontinuous function on $D$. For any $z\in \partial\{\Phi^{D}_{o,\max}<-t\}\cap D$, as $\phi\ge\Phi_{o,\max}^D$ on $D$, we have
	$$\phi(z)=\Phi_{o,\max}^D(z)\le\frac{1}{\lambda(\mathbb{B}(z,r))}\int_{B(z;r)}\Phi^D_{o,\max} \ d\lambda\le\frac{1}{\lambda(\mathbb{B}(z,r))}\int_{B(z;r)}\phi \ d\lambda,$$
	where $\mathbb{B}(z,r)\subset D$ is a ball with a radius of $r$ and $\lambda$ is the Lebesgue measure on $\mathbb{C}^n$. Hence, $\phi$ is a plurisubharmonic function on $D$. As  $\Phi^{D}_{o,\max}$ is a global Zhou weight related to $|f_0|^2e^{-2\varphi_0}$ on $D$, we get $\Phi^{D}_{o,\max}\ge\phi$, which implies that $\Phi_{o,\max}^{\{\Phi^{D}_{o,\max}<-t\}}=\Phi_{o,\max}^D+t$ on $\{\Phi_{o,\max}^D<-t\}$. Lemma \ref{l:subset} has been proved. 
\end{proof}

\subsection{Proofs of Proposition \ref{l:max2} and Proposition \ref{thm:contin}}

In this section, we prove Proposition \ref{l:max2} and Proposition \ref{thm:contin}. 

Firstly, we recall a lemma, which will be used in the proof of Proposition \ref{l:max2}.

\begin{Lemma}[see \cite{Blo-note}, see also \cite{BT76,Blo93}]
	\label{l:max1}Let $\varphi\in \mathrm{PSH}(\Omega)\cap L^{\infty}_{\mathrm{loc}}(\Omega)$ on an open subset $\Omega$ of $\mathbb{C}^n$. If for any $u\in \mathrm{PSH}(\Omega)$ such that $\varphi\ge u$ outside a compact subset of $\Omega$ we have $\varphi\ge u$ on $\Omega$, then $(dd^c\varphi)^n=0$ on $\Omega$.
\end{Lemma}

Now, we prove Proposition \ref{l:max2}.
\begin{proof}[Proof of Proposition \ref{l:max2}]
	As $\tilde L_o\not=\emptyset$, it follows from Lemma \ref{l:max-loc} that $\Phi^{D}_{o,\max}$ is a local Zhou weight related to $|f_0|^2e^{-2\varphi_0}$ near $o$. Thus, $\Phi^{D}_{o,\max}\ge N\log|z|+O(1)$ for large enough $N\gg0$. Then there exist $r>0$ such that
	\[C:=\inf_{\{|z|=r\}}\Phi_{o,\max}^D>-\infty.\]
	Let 
	\begin{equation*}
		\tilde\varphi:=\left\{
		\begin{array}{ll}
			\Phi_{o,\max}^D & \ \text{on} \ \{|z|<r\}, \\
			\max\{\Phi_{o,\max}^D, C-1\} & \ \text{on} \ D\backslash\{|z|<r\}.
		\end{array}
		\right.
	\end{equation*}
	Then $\tilde\varphi$ is a negative plurisubharmonic function on $D$. By the definition of $\Phi_{o,\max}^D$, $\Phi_{o,\max}^D=\tilde\varphi$ on $D$, which implies that $\Phi_{o,\max}^D\in L_{\mathrm{loc}}^{\infty}(D\backslash\{o\})$. Using Lemma \ref{l:max1}, we have $\big(dd^c\Phi^{D}_{o,\max}\big)^n=0$ on $D\backslash\{o\}$.
\end{proof}

Finally, we prove Proposition \ref{thm:contin}.

\begin{proof}[Proof of Proposition \ref{thm:contin}]
	Note that there exists $N_0\gg0$ such that $\Phi^{D}_{o,\max}\ge N_0\log|z|+O(1)$ near $o$. As $D$ is a hyperconvex domain, there exists a continuous exhaustion plurisubharmonic function $\varphi_1<0$ on $D$. Then there exist $r_1>r_2>0$ and $N_1>0$ such that $$\inf_{\{r_2<|z|<r_1\}}\Phi^{D}_{o,\max}(z)\ge  N_1\sup_{\{r_2<|z|<r_1\}}\varphi_1(z).$$
	Let
	\begin{equation*}
		\varphi_2:=\left\{
		\begin{array}{ll}
			\Phi^{D}_{o,\max} & \ \text{on} \ \{|z|<r_1\}, \\
			\max\{\Phi^{D}_{o,\max},N_1\varphi_1\} & \ \text{on} \D\backslash\{|z|<r_1\}.
		\end{array}
		\right.
	\end{equation*}
	Then $\varphi_2$ is a negative plurisubharmonic function on $D$ and $\varphi_2\ge \Phi^{D}_{o,\max}$. By the definition of $\Phi^{D}_{o,\max}$, we have $\varphi_2= \Phi^{D}_{o,\max}$, i.e. $\Phi^{D}_{o,\max}\ge N_1\varphi_1$ on $D\backslash\{|z|<r_1\}$, which implies that $\Phi^{D}_{o,\max}(z)\rightarrow0$ when $z\rightarrow\partial D$.
	
	For any positive integer $m$, let $\{\sigma_{m,k}\}_{k=1}^{\infty}$ be an orthonormal basis of
	\[A^2\big(D,2m\Phi^{D}_{o,\max}\big):=\left\{f\in\mathcal{O}(D):\int_{D}|f|^2e^{-2m\Phi^{D}_{o,\max}}d\lambda<+\infty\right\}.\]
	Denote 
	$$\varphi_m:=\frac{1}{2m}\log\sum_{k=1}^{\infty}|\sigma_{m,k}|^2$$
	on $D$. 
	
	Fixed any $z\in D$, there exists a holomorphic function $f_{z,m}$ on $D$ such that
	\[\int_{D}|f_{z,m}|^2e^{-2m\Phi^{D}_{o,\max}}d\lambda=1 \ \&  \ \frac{1}{2m}\log|f_{z,m}(z)|^2=\varphi_m(z).\]
	For any $t>0$ satisfying $z\in\{\Phi^{D}_{o,\max}<-t\}$, since $\Phi^{D}_{o,\max}(z)\rightarrow0$ when $z\rightarrow\partial D$ and $\int_{D}|f_{z,m}|^2e^{-2m\Phi^{D}_{o,\max}}d\lambda=1$, there exists  $M\gg0$ such that 
	$|f_{z,m}|<1$ on $\{\Phi^{D}_{o,\max}<-t\}$ for any $m>M$. It follows from Lemma \ref{l:subset} and Lemma \ref{r:value-inequa}  that
	$$\varphi_m(z)=\frac{1}{m}\log|f_{z,m}(z)|\le\frac{\sigma\big(\log|f_{z,m}|,\Phi^{D}_{o,\max}\big)}{m}(\Phi^{D}_{o,\max}+t)(z)$$
	for $z\in\{\Phi^{D}_{o,\max}<-t\}$. Note that $c_o^{f_{z,m}}(\Phi^{D}_{o,\max})\ge m$, then  Theorem \ref{thm:valu-jump} shows that $\sigma\big(\log|f_{z,m}|,\Phi^{D}_{o,\max}\big)\ge m+c_3$, where $c_3$ is a constant independent of $m$. Hence, we have
	\begin{equation}
		\label{eq:0911a}
		\frac{\varphi_m(z)}{1+\frac{c_3}{m}}-t\le \Phi^{D}_{o,\max}(z).
	\end{equation}
	Using inequality \eqref{eq:0911b} in Lemma \ref{l:appro-Berg}, we get 
	\begin{equation}
		\label{eq:0911c}
		\Phi^{D}_{o,\max}\le \varphi_m+\frac{c_1}{m}
	\end{equation}
	on $D$. Note that $e^{\varphi_m}$ is smooth on $D$, then it follows from inequality \eqref{eq:0911a} and inequality \eqref{eq:0911c} that $e^{\Phi^{D}_{o,\max}}$ is continuous on $D$.
\end{proof}

\section{Approximations of global Zhou weights: proofs of Theorem \ref{thm-approximation}, Corollary \ref{cor-approximation} and Corollary \ref{cor-approximation2}}\label{Approximation}

In this section, we discuss the approximations of global Zhou weights, and prove Theorem \ref{thm-approximation}, Corollary \ref{cor-approximation} and Corollary \ref{cor-approximation2}.

\subsection{Preparations}

Let $D$ be a hyperconvex domain in $\mathbb{C}^n$ such that the origin $o\in D$. Let $\Phi_{o,\max}^{f_0,\varphi_0,D}$ ($\Phi^D_{o,\max}$ for short in this section) be a global Zhou weight related to some $|f_0|^2e^{-2\varphi_0}$   on $D$ near $o$ defined in this section, where $f_0$ is a holomorphic vector on $D$, and $\varphi_0$ is a plurisubharmonic function near $o$ such that $|f_0|^2e^{-2\varphi_0}$ is integrable near $o$. Denote $\sigma(\cdot, \Phi^D_{o,\max})$ be the Zhou number with respect to $\Phi^D_{o,\max}$.

For any $m\in\mathbb{N}_+$, we recall two compact subsets of $\mathcal{O}(D)$ as follows:
\[\mathscr{E}_m(D):=\big\{f\in\mathcal{O}(D) : \sup_{z\in D}|f(z)|\le 1, (f,o)\in\mathcal{I}(m\Phi^D_{o,\max})_o\big\},\]
\[\mathscr{A}^2_m(D):=\big\{f\in\mathcal{O}(D) : \|f\|_D\le 1, (f,o)\in\mathcal{I}(m\Phi^D_{o,\max})_o\big\},\] 
where $\|f\|_D^2:=\int_D|f|^2$. The compactness of $\mathscr{E}_m(D)$ and $\mathscr{A}^2_m(D)$ is due to the closedness property of coherent sheaves. We also recall the definitions of $\phi_m$ and $\varphi_m$ for any $m$:
\begin{equation}\label{eq-phim}
	\phi_m(z):=\sup_{f\in\mathscr{E}_m(D)} \frac{1}{m}\log |f(z)|, \ \forall z\in D,
\end{equation}

\begin{equation}\label{eq-varphim}
	\varphi_m(z):=\sup_{f\in\mathscr{A}_m^2(D)}\frac{1}{m}\log|f(z)|, \ \forall z\in D.
\end{equation}

Suppose that $D$ is a strictly hyperconvex domain with the function $\varrho$ defined in Definition \ref{def-strhpconvex}. For any $j\in\mathbb{N}_+$, denote $D_j:=\{z\in \Omega : \varrho(z)<1/j\}$, which is a decreasing sequence of bounded hyperconvex domains. Define
\[\Phi^{D_j} _{o,\max}(z):=\sup\bigg\{\phi(z): \phi\in \mathrm{PSH}^-(D_j), \ (f_0,o)\notin \mathcal{I}(\varphi_0+\phi)_o, \  \phi\ge \Phi^D_{o,\max}+O(1)\bigg\}\]
for any $z\in D_j$, where the inequality $\phi\ge \Phi^D_{o,\max}+O(1)$ means that it holds near $o$ (similar inequalities in this section all mean that they hold near $o$). Then $\Phi^{D_j} _{o,\max}$ is a global Zhou weight related to $|f_0|^2e^{-2\varphi_0}$ on $D_j$ (see Lemma \ref{l:local-glabal}).

\begin{Lemma}\label{lem-ImDjImD}
	For any $1\le j<j'$, the following statements hold:
	
	$(1)$
	\[\Phi_{o,\max}^{D_j}\le \Phi_{o,\max}^{D_{j'}} \ \text{on} \ D_{j'},\]
	and
	\[\Phi_{o,\max}^{D_j}\le\Phi_{o,\max}^D \ \text{on} \ D.\]
	
	$(2)$
	\[\Phi_{o,\max}^{D_j}=\Phi_{o,\max}^D+O(1) \ \text{near} \ o.\]
	
	$(3)$ For any $t>0$,
	\[\mathcal{I}\big(t\Phi^{D_j} _{o,\max}\big)_o=\mathcal{I}\big(t\Phi^D_{o,\max}\big)_o.\]
	
	$(4)$ There exists $N>0$, which is independent of $j$, such that
	\[\Phi_{o,\max}^{D_j}(z)\ge N G_{D_j}(o,z), \ \forall z\in D,\]
	where $G_{D_j}(o,\cdot)$ is the pluricomplex Green function of $D_j$ with a pole at $o$.
\end{Lemma}

\begin{proof}
	The statement (1) just follows from the definitions of $\Phi^{D_j} _{o,\max}$ and $\Phi^D_{o,\max}$. It shows that $\big( \Phi_{o,\max}^{D_j}\big)_{j\ge 1}$ is an increasing sequence on $D$.
	
	The statement (2) follows from that global Zhou weights are also local Zhou weights.
	
	The	statement (3) is a direct result of the statement (2).
	
	To show the statement (4), let us look back to the proof of Lemma \ref{l:max>green}. We can see that the sufficiently large constant $N>0$ is completely determined by the local property of the global Zhou weight $\Phi_{o,\max}^{D_j}$ near $o$. Thus by the statement (2), the constant $N$ here can be chosen independent of $j$.	
\end{proof}

To simplify the proof of the next lemma (Lemma \ref{lem-DjconvergetoD}), we recall a lemma in \cite{Ni95} about the pluricomplex Green functions on $D_j$.

\begin{Lemma}[\cite{Ni95}]\label{lem-Greenfunction}
	The sequence of pluricomplex Green functions $\big(G_{D_j}(o,\cdot)\big)_{j\ge 1}$ converges uniformly to $G_D(o,\cdot)$ on $\overline{D}$.
\end{Lemma}

\begin{Lemma}
	\label{lem-DjconvergetoD}
	The sequence $\big(\Phi^{D_j} _{o,\max}\big)_{j\ge 1}$ converges uniformly to $\Phi^D_{o,\max}$ on $\overline{D}$, where $\Phi^D_{o,\max}|_{\partial D}$ is defined to be $0$.
\end{Lemma}

\begin{proof}
	According to Lemma \ref{lem-ImDjImD} and Lemma \ref{lem-Greenfunction}, we have that, restricted on the boundary of $D$,  the sequence $\big(\Phi_{o,\max}^{D_j}\big)$ increasingly converges uniformly to $\Phi_{o,\max}^D$, i.e. to $0$. Denote
	\[C_j:=\inf_{z\in\partial D}\Phi_{o,\max}^{D_j}(z)\in (-\infty,0),\]
	then $\lim_j C_j=0$. For any $j\ge 1$, consider a function $\Phi_j$ defined on $D_j$ by
	\begin{equation*}
		\Phi_j(z):=\left\{
		\begin{array}{ll}
			\Phi_{o,\max}^{D_j}(z) & z\in D_j\setminus D, \\
			\max\left\{\Phi_{o,\max}^{D_j}(z), \Phi_{o,\max}^D(z)+C_j\right\} & z\in\overline{D}.
		\end{array}
		\right.
	\end{equation*}
	Then $\Phi_j\in \mathrm{PSH}^-(D_j)$, and continuous near $\partial D$. In addition, $\Phi_j\ge \Phi_{o,\max}^{D_j}$ on $D_j$. On the other hand, we have
	\[\sigma\big(\Phi_j,\Phi_{o,\max}^{D_j}\big)=\min\big\{\sigma\big(\Phi_{o,\max}^{D_j},\Phi_{o,\max}^{D_j}\big), \sigma\big(\Phi_{o,\max}^D+C_j,\Phi_{o,\max}^{D_j}\big)\big\}=1,\]
	which implies $\Phi_j\le \Phi_{o,\max}^{D_j}$. Consequently, $\Phi_{o,\max}^D+C_j\le \Phi_{o,\max}^{D_j}$ on $\overline{D}$. Thus, we get
	\[\Phi_{o,\max}^{D_j}\le \Phi_{o,\max}^D\le \Phi_{o,\max}^{D_j}-C_j \ \text{on} \ \overline{D}.\]
	Since $\lim_j C_j=0$, the lemma is proved.	
\end{proof}

We also have the following property that shows the approximation of Zhou weights from inside domains. Suppose $D$ is a bounded hyperconvex domain containing $o$. Let $\{\mathscr{D}_l\}_{l\ge 1}$ be an increasing sequence of hyperconvex domains containing $o$  such that $\bigcup_l\mathscr{D}_l=D$. Define
\[\Phi^{\mathscr{D}_l} _{o,\max}(z):=\sup\bigg\{\phi(z): \phi\in \mathrm{PSH}^-(\mathscr{D}_l), \ (f_0,o)\notin \mathcal{I}(\varphi_0+\phi)_o,\ \phi\ge \Phi^D_{o,\max}+O(1)\bigg\}\]
for any $z\in \mathscr{D}_l$. 
\begin{Lemma}\label{lem-approxfrominside}
	The sequence $\big(\Phi^{\mathscr{D}_l} _{o,\max}\big)_{l\ge 1}$ converges on $D$ to $\Phi_{o,\max}^D$ pointwisely.
\end{Lemma}

\begin{proof}
	Observe that $\big(\Phi^{\mathscr{D}_l} _{o,\max}\big)_{l\ge 1}$ is a decreasing sequence of negative plurisubharmonic functions on $D$, then the pointwise limit $\varPhi:=\lim_l \Phi^{\mathscr{D}_l} _{o,\max}$ exists, and $\varPhi$ is negative and plurisubharmonic on $D$. In addition, we have $\varPhi\ge\Phi_{o,\max}^D$ since $\Phi_{o,\max}^{\mathscr{D}_l}\ge\Phi_{o,\max}^D$ on $\mathscr{D}_l$ for every $l$, and $(f_0, o)\notin \mathcal{I}(\varphi_0+\varPhi)_o$ since $\varPhi\le \Phi_{o,\max}^{\mathscr{D}_1}$. Then according to the definition of $\Phi_{o,\max}^D$, we get $\varPhi=\Phi_{o,\max}^D$ on $D$.
\end{proof}

The following lemma gives a relation between the plurisubharmonic functions $\phi_m$ and $\varphi_m$.

\begin{Lemma}\label{lem-phimlevarphim+}
	Let $D$ be a bounded hyperconvex domain, and for any $m\in\mathbb{N}_+$ let the functions $\phi_m$ and $\varphi_m$ be defined by equation (\ref{eq-phim}) and (\ref{eq-varphim}). Then for any $m\in\mathbb{N}_+$, we have
	
	$(1)$ $\phi_m$ and $\varphi_m$ are continuous and plurisubharmonic functions on $D$ (may taking the value $-\infty$), and $\phi_m$ takes values in $[-\infty, 0)$;
	
	$(2)$
	\begin{equation}\label{eq-phimlevarphim+}
		\phi_m(z)\le\varphi_m(z)+\frac{1}{2m}\log \lambda(D), \ \forall z\in D,
	\end{equation}
	where $\lambda(D)$ is the Lebesgue measure of $D$.
	
\end{Lemma}

\begin{proof}
	(1) For every $f\in\mathscr{E}_m(D)$ or $\mathscr{A}^2_m(D)$, $(\log|f|)/m$ is continuous on $D$, so $\phi_m$ and $\varphi_m$ are lower-semicontinuous on $D$. Since $\mathscr{E}_m(D)$ and $\mathscr{A}^2(D)$ are compact subsets of $\mathcal{O}(D)$, using Montel's Theorem, we can see that $\phi_m$ and $\varphi_m$ are also upper-semicontinuous on $D$. It follows that $\phi_m$ and $\varphi_m$ are continuous on $D$. 
	
	Now the plurisubharmonicity of $\phi_m$ and $\varphi_m$ is a consequence of the definitions of $\phi_m$ and $\varphi_m$.
	
	(2) Observe that for any $f\in\mathscr{E}_m(D)$, we have $(\lambda(D))^{-1/2}f\in\mathscr{A}^2_m(D)$, which implies inequality (\ref{eq-phimlevarphim+}).
\end{proof}

Now we suppose $D$ is a bounded strictly hyperconvex domain in $\mathbb{C}^n$ with the function $\varrho$ defined in Definition \ref{def-strhpconvex} and $D_j=\{z\in\Omega : \varrho(z)<1/j\}$.

\begin{Lemma}\label{lem-hmDjphim}
	For any $j\ge 1$ and $m\ge 1$, we have
	\[h_{m,D_j}(z)\le \varphi_{m,D_j}(z)\le \phi_m(z)+\frac{\log(c\delta(j)^{-n})}{m}, \ \forall z\in D,\]
	where $c$ is a constant only depending on $n$, $\delta(j)=\mathrm{dist}(\overline{D},\partial D_j)$, and for any $z\in D_j$,
	\[h_{m,D_j}(z):=\sup\left\{\frac{1}{m}\log|f(z)| : f\in \mathcal{O}(D_j), \ \int_{D_j}|f|^2e^{-2m\Phi^{D_j} _{o,\max}}\le 1\right\},\]
	\[\varphi_{m,D_j}(z):=\sup\left\{\frac{1}{m}\log|f(z)| : f\in \mathcal{O}(D_j), \  \int_{D_j}|f|^2\le 1, \ (f,o)\in\mathcal{I}\big(m\Phi_{o,\max}^{D_j}\big)_o\right\}.\]
\end{Lemma}

\begin{proof}
	Let $f\in \mathcal{O}(D_j)\cap L^2(D_j)$ with $(f,o)\in\mathcal{I}\big(m\Phi_{o,\max}^{D_j}\big)_o=\mathcal{I}\big(m\Phi_{o,\max}^{D}\big)_o$ (by Lemma \ref{lem-ImDjImD}(3)), then there exists some $w\in \overline{D}$ such that $f(w)=\sup_{\overline{D}}|f|$. According to the mean value inequality applied to the plurisubharmonic function $|f|^2$, we have
	\begin{equation*}
		|f(w)|^2\le\frac{c^2}{\delta(j)^{2n}}\int_{B(w,\delta(j))}|f(z)|^2\le\frac{c^2}{\delta(j)^{2n}}\int_{D_j}|f(z)|^2.
	\end{equation*}
	Thus
	\[\varphi_{m,D_j}(z)\le \phi_m(z)+\frac{\log(c\delta(j)^{-n})}{m}, \ \forall z\in D.\]
	
	Since $\Phi_{o,\max}^{D_j}$ is negative on $D_j$, for any $f\in A^2\big(D_j, 2m\Phi_{o,\max}^{D_j}\big)$, we have $f\in \mathcal{O}(D_j)\cap L^2(D_j)$, $(f,o)\in \mathcal{I}\big(m\Phi^{D_j} _{o,\max}\big)_o$ and
	\[\int_{D_j}|f|^2\le\int_{D_j}|f|^2e^{-2m\Phi_{o,\max}^{D_j}}.\]
	Thus we obtain $h_{m,D_j}\le\varphi_{m,D_j}$.
\end{proof}

For fixed $z\in D$, let $k$ be a negative integer with $|k|$ sufficiently large such that $z\in D_{k}:=\{z\in D : \varrho(z)<1/k\}$. Note that $D_{k}$ is a bounded strictly hyperconvex domain and $D_k\subset\subset D$ for $k<0$. Denote
\[\Phi^{D_k} _{o,\max}(w):=\sup\bigg\{\phi(w) : \phi\in \mathrm{PSH}^-(D_k), \ (f_0,o)\notin \mathcal{I}(\varphi_0+\phi)_o, \ \phi\ge \Phi^D_{o,\max}+O(1)\bigg\}\]
for $w\in D_k$. Then $\Phi_{o,\max}^{D_k}$ is a global Zhou weight related to $|f_0|^2e^{-\varphi_0}$ on $D_k$ near $o$ satisfying that $\Phi_{o,\max}^{D_k}=\Phi_{o,\max}^D+O(1)$ near $o$ (see Lemma \ref{l:local-glabal} and Lemma  \ref{l:max-loc}), and
\[\Phi^{D} _{o,\max}(w'):=\sup\bigg\{\phi(w') : \phi\in \mathrm{PSH}^-(D), \ (f_0,o)\notin \mathcal{I}(\varphi_0+\phi)_o, \ \phi\ge \Phi^{D} _{o,\max}+O(1)\bigg\}\]
for any $w'\in D$. Replacing $D$ by $D_k$ and $\varrho$ by $\varrho-1/k$ in Lemma \ref{lem-hmDjphim}, we get
\begin{Lemma}\label{lem-phimDk}
	For fixed $z\in D$ and $k<0$ with $z\in D_k$, we have
	\[\varphi_{m}(z)\le \phi_{m,D_k}(z)+\frac{\log(c\delta(k)^{-n})}{m}, \ \forall m\ge 1,\]
	where $c$ is a constant only depending on $n$, $\delta(k)=\mathrm{dist}(\overline{D_k},\partial D)$, and for any $w\in D_k$,
	\[\phi_{m,D_k}(w):=\sup\left\{\frac{1}{m}\log|f(w)| : f\in\mathcal{O}(D_k), \ \sup_{D_k}|f|\le 1, \ (f,o)\in\mathcal{I}(m\Phi^{D_k} _{o,\max})_o\right\}.\]
\end{Lemma}

\begin{Lemma}\label{lem-phimlePhimaxD}
	If $D$ is a bounded hyperconvex domain, then there exists a constant $\mathsf{C}$ independent of $m$ such that
	\begin{equation*}
		\phi_m\le \frac{m-\mathsf{C}}{m}\Phi^D_{o,\max}, \ \forall m\in\mathbb{N}_+.
	\end{equation*}
\end{Lemma}

\begin{proof}
	Let $f\in\mathscr{E}_m(D)$, then $\log|f|$ is negative on $D$ and $c_{o}^{f}(\Phi_{o,\max}^D)\ge m$. Theorem \ref{thm:valu-jump} shows that
	\[\sigma(\log|f|,\Phi_{o,\max}^D)\ge m-\mathsf{C},\]
	where $\mathsf{C}$ is a constant independent of $f$ and $m$. Now we get
	\[\log|f|\le \sigma(\log|f|,\Phi_{o,\max}^D)\Phi_{o,\max}^D\le (m-\mathsf{C})\Phi_{o,\max}^D, \ m>\mathsf{C}\]
	(the case $m\le \mathsf{C}$ is trivial), which implies
	\[\phi_m\le \frac{m-\mathsf{C}}{m}\Phi_{o,\max}^D, \forall m\in\mathbb{N}_+.\]
\end{proof}

\subsection{Proofs of Theorem \ref{thm-approximation}, Corollary \ref{cor-approximation} and Corollary \ref{cor-approximation2}}

Now we give the proofs of Theorem \ref{thm-approximation}, Corollary \ref{cor-approximation} and Corollary \ref{cor-approximation2}.

\begin{proof}[Proof of Theorem \ref{thm-approximation}]
	
	Firstly, we prove the statement (2) in Theorem \ref{thm-approximation}.
	It follows from Lemma \ref{lem-ImDjImD}(2), Lemma \ref{l:appro-Berg}, Remark \ref{r:0930a} and Lemma \ref{lem-hmDjphim} that
	\[\phi_m\ge h_{m,D_j}+O(1)\ge \Phi^{D_j} _{o,\max}+O(1)=\Phi^D_{o,\max}+O(1)\]
	near $o$. We also have $\varphi_m\ge\Phi^D_{o,\max}+O(1)$ near $o$ by Lemma \ref{lem-phimlevarphim+}. Thus,
	\[\sigma(\phi_m,\Phi^D_{o,\max})\le 1, \ \sigma(\varphi_m,\Phi_{o,\max}^D)\le 1.\]	
	In addition, we get
	\[\sigma(\phi_m, \Phi_{o,\max}^D)\ge 1-\frac{\mathsf{C}}{m}, \ \forall m\in\mathbb{N}_+\]
	according to Lemma \ref{lem-phimlePhimaxD} and
	\[\sigma(\varphi_m, \Phi_{o,\max}^D)\ge 1-\frac{\mathsf{C}}{m}, \ \forall m\in\mathbb{N}_+\]
	according to Lemma \ref{lem-phimDk} for some constant $\mathsf{C}$ independent of $m$.

	Next, we prove
	\[\lim_{m\to\infty}\phi_m(z)=\Phi_{o,\max}^D(z), \ \forall z\in D.\]
	According to Lemma \ref{l:appro-Berg} and Lemma \ref{lem-hmDjphim}, for any $z\in D$ we have
	\[\Phi^{D_j} _{o,\max}(z)-\frac{C_1}{m}\le h_{m,D_j}(z)\le \phi_m+\frac{\log(c\delta(j)^{-n})}{m},\]
	where $C_1, c$ are constants independent of $m$. Letting $m\to\infty$, we deduce that
	\[\Phi_{o,\max}^{D_j}(z)\le \liminf_{m\to\infty}\phi_m(z).\]
	Letting $j\to\infty$ in the above inequality, by Lemma \ref{lem-DjconvergetoD}, we get
	\[\Phi_{o,\max}^D(z)\le\liminf_{m\to\infty}\phi_m(z).\]
	On the other hand, letting $m\to\infty$ in Lemma \ref{lem-phimlePhimaxD}, we obtain
	\[\limsup_{m\to\infty}\phi_m(z)\le\Phi_{o,\max}^D(z).\]
	Consequently,
	\[\lim_{m\to\infty}\phi_m(z)=\Phi_{o,\max}^D(z), \ \forall z\in D.\]
	
	Finally, we prove that for any $z\in D$,
	\[\lim_{m\to\infty}\varphi_m(z)=\Phi_{o,\max}^D(z).\]
	Lemma \ref{lem-phimlevarphim+} gives
	\[\liminf_{m\to\infty}\varphi_m(z)\ge\lim_{m\to\infty}\phi_m(z)=\Phi_{o,\max}^D(z).\]
	On the other hand, for $k<0$ such that $z\in D_k=\{\varrho<1/k\}$, by Lemma \ref{lem-phimDk},
	\[\limsup_{m\to\infty}\varphi_m(z)\le\liminf_{m\to\infty}\phi_{m,D_k}(z).\]
	It follows the previous result of this proof that
	\[\lim_{m\to\infty}\phi_{m,D_k}(z)=\Phi_{o,\max}^{D_k}(z),\]
	thus
	\[\limsup_{m\to\infty}\varphi_m(z)\le\Phi_{o,\max}^{D_k}(z)\]
	for all $k<0$ with $|k|\gg1$. Letting $k\to -\infty$ and using Lemma \ref{lem-approxfrominside}, we obtain
	\[\limsup_{m\to\infty}\varphi_m(z)\le\Phi_{o,\max}^D(z).\]
	Thus,
	\[\lim_{m\to\infty}\varphi_m(z)=\Phi_{o,\max}^D(z), \ \forall z\in D,\]
	which completes the proof.
\end{proof}

In the following, we prove Corollary \ref{cor-approximation} and Corollary \ref{cor-approximation2}.

\begin{proof}[Proof of Corollary \ref{cor-approximation}]
	Denote
	\[\mathscr{E}(D):=\big\{f\in\mathcal{O}(D) : \sup_D|f|\le 1, \ f(o)=0, \ f\not\equiv 0\big\}.\]
	Then for any $f\in \mathscr{E}(D)$, $\log|f|$ is negative on $D$. As $\Phi_{o,\max}^D$ is a global Zhou weight, we have
	\[\frac{\log|f(w)|}{\sigma(\log|f|, \Phi^D_{o,\max})}\le \Phi_{o,\max}^D(w),\]
	for any $f\in\mathscr{E}(D)$ and any $w\in D$.
	
	On the other hand, Theorem \ref{thm-approximation} gives
	\[\lim_{m\to\infty}\phi_m(w)=\Phi^D_{o,\max}(w), \ \forall w\in D.\]
	By the compactness of $\mathscr{E}_m(D)$ and Montel's theorem, for fixed $w\in D$, there exists $F_{m,w}\in\mathscr{E}_m(D)\setminus\{0\}\subset\mathscr{E}(D)$ such that
	\[\phi_m(w)=\frac{1}{m}\log|F_{m,w}(w)|.\]
	Since $c_o^{F_{m,w}}(\Phi_{o,\max}^D)\ge m$, Theorem \ref{thm:valu-jump} shows that
	\[\sigma(\log|F_{m,w}|,\Phi_{o,\max}^D)\ge m-\mathsf{C},\]
	where $\mathsf{C}$ is a constant independent of $m$. Thus
	\[\limsup_{m\to\infty}\frac{\log|F_{m,w}(w)|}{\sigma(\log|F_{m,w}|,\Phi_{o,\max}^D)}\ge \limsup_{m\to\infty}\frac{m}{m-\mathsf{C}}\phi_m(w)=\Phi_{o,\max}^D(w),\]
	which completes the proof.
\end{proof}

\begin{proof}[Proof of Corollary \ref{cor-approximation2}]
	Denote
	\[\mathcal{S}\big(\Phi_{o,\max}^D\big):=\big\{\phi\in\mathrm{PSH}^-(D): \phi\sim_{\mathcal{I}}\Phi^D_{o,\max} \ \text{at} \ o\big\}.\]
	Then $\Phi_{o,\max}^D\in \mathcal{S}\big(\Phi_{o,\max}^D\big)$. Take any $\varphi\in\mathcal{S}\big(\Phi_{o,\max}^D\big)$. For any $w\in D$, denote
	\[\tilde{\varphi}_m(w):=\sup\left\{\frac{1}{m}\log|f(w)| : f\in\mathcal{O}(D), \ \|f\|_D\le 1, \ (f,o)\in\mathcal{I}(m\varphi)_o\right\}.\]
	Since $\varphi\sim_{\mathcal{I}}\Phi_{o,\max}^D$ at $o$, using Theorem \ref{thm-approximation}, we can get
	\[\lim_{m\to\infty}\tilde{\varphi}_m(w)=\Phi_{o,\max}^D(w), \ \forall w\in D.\]
	Since $\varphi$ is negative, we can see
	\[\tilde{\varphi}_m(w)\ge h_m(w), \ \forall w\in D,\]
	where
	\[h_{m}(w):=\sup\left\{\frac{1}{m}\log|f(w)| : f\in \mathcal{O}(D), \ \int_{D}|f|^2e^{-2m\varphi}\le 1\right\}.\]
	Lemma \ref{l:appro-Berg} gives
	\[\tilde{\varphi}_m(z)\ge h_m(z)\ge\varphi(z)-\frac{C_1}{m}\]
	for some $C_1>0$ independent of $m$. Letting $m\to\infty$, we deduce that $\Phi_{o,\max}^D(z)\ge \varphi(z)$ for any $z\in D$. Thus we complete the proof.
\end{proof}

\section{Appendix: proofs of Lemma \ref{l:0921-3} and Lemma \ref{l:sigma-analytic}}
\label{sec:app}

We prove Lemma \ref{l:0921-3} and Lemma \ref{l:sigma-analytic} in this section. Firstly, we present some lemmas.

\begin{Lemma}
	\label{l:0921-2}Let $\varphi$ be a plurisubharmonic function near $o$,  and let $f$ be a holomorphic function near $o$. Assume that $c_o^{f}(\varphi)=1$, then $\lim_{m\rightarrow+\infty}c_o^f(\varphi_m)=1$, where
	\[\varphi_m:=\frac{1}{2m}\log\sum_{1\le l\le k_m}|f_{m,l}|^2,\]
	and $\{f_{m,1},\ldots,f_{m,k_m}\}$ is the generators set of $\mathcal{I}(m\varphi)_o$. 
\end{Lemma}
\begin{proof}As $\{f_{m,1},\ldots,f_{m,k_m}\}$ is the generators set of $\mathcal{I}(m\varphi)_o$, there exists a neighborhood $U$ of $o$ such that
	\begin{equation}
		\nonumber
		\begin{split}
			&\int_{U}e^{-2\varphi}-e^{-2\max\big\{\varphi,\frac{m+1}{m}\varphi_{m+1}\big\}}\\
			\le&\int_{U\cap\big\{\frac{m+1}{m}\varphi_{m+1}>\varphi\big\}}e^{-2\varphi}\\
			\le&\int_U e^{2(m+1)\varphi_{m+1}-2(m+1)\varphi}\\
			=&\int_U\sum_{1\le l\le k_{m+1}}|f_{m+1,l}|^2e^{-2(m+1)\varphi}\\
			<&+\infty.
		\end{split}
	\end{equation}
	Theorem \ref{thm:SOC} shows that $|f|^2e^{-2\varphi}$ is not integrable near $o$. Then $|f|^2e^{-2\max\{\varphi,\frac{m+1}{m}\varphi_{m+1}\}}$ is not integrable near $o$, i.e., $c_o^{f}\big(\max\{\varphi,\frac{m+1}{m}\varphi_{m+1}\}\big)\le 1$, which implies that $$c_o^{f}\left(\frac{m+1}{m}\varphi_{m+1}\right)\le 1.$$
	It follows from Demailly's approximation theorem (Lemma \ref{l:appro-Berg}) that $\varphi_m\ge\varphi+O(1)$ near $o$, which implies that $$c_o^{f}(\varphi_m)\ge c_o^{f}(\varphi)=1$$
	for any $m$.  Thus, we have $\lim_{m\rightarrow+\infty}c_o^f(\varphi_m)=1$.
\end{proof}

The following lemma will be used in the proof of Lemma \ref{l:0921-1}.
\begin{Lemma}
	[see \cite{guan-remapprox}]\label{l:III}
	$(z_1^{\alpha_1}\cdots z_n^{\alpha_n},o)\in \mathcal{I}\big(\log\max\{|z_1|^{b_1},\ldots,|z_n|^{b_n}\}\big)_o$ if and only if $\sum\limits_{1\le j\le n}\frac{\alpha_j+1}{b_j}>1$.
\end{Lemma}

We recall the following desingularization theorem due to Hironaka.
\begin{Theorem}[\cite{Hironaka}, see also \cite{BM-1991}]\label{thm:desing}
	Let $X$ be a complex manifold, and $M$ be an analytic sub-variety in $X$.  Then there is a local finite sequence of blow-ups $\mu_j:X_{j+1}\rightarrow X_j$ $(X_1:=X,j=1,2,...)$ with smooth centers $S_j$ such that:
	
	$(1)$ Each component of $S_j$ lies either in $(M_j)_{\rm{sing}}$ or in $M_j\cap E_j$, where $M_1:=M$, $M_{j+1}$ denotes the strict transform of $M_j$ by $\mu_j$, $(M_j)_{\rm{sing}}$ denotes the singular set of $M_j$, and $E_{j+1}$ denotes the exceptional divisor $\mu^{-1}_j(S_j\cup E_j)$;
	
	$(2)$ Let $M'$ and $E'$ denote the final strict transform of $M$ and the exceptional divisor respectively. Then:
	
	\quad$(a)$ The underlying point-set $|M'|$ is smooth;
	
	\quad$(b)$ $|M'|$and $E'$ simultaneously have only normal crossings.
\end{Theorem}

The $(b)$ in the above theorem means that, local, there is a coordinate system in which $E'$ is a union of coordinate hyperplanes and $|M'|$ is a coordinate subspace.

We prove Lemma \ref{l:sigma-analytic} by using Theorem \ref{thm:desing}.

\begin{proof}[Proof of Lemma \ref{l:sigma-analytic}]
	Let $D$ be a small neighborhood of $o$. Using Theorem \ref{thm:desing}, there is a proper holomorphic map $\mu:\widetilde D\rightarrow D$, which is local a finite composition of blow-ups with smooth centers, such that 
	$$u\circ\mu(w)=c_1\log|w^a|+\tilde u(w)$$
	on $W$,
	where  $a=(a_1,\ldots,a_n)\in\mathbb{Z}_{\ge0}^n$ and $(W;w_1,\ldots,w_n)$ is a coordinate ball centered at $\tilde z\in\widetilde{D}$ satisfying that $W\Subset\widetilde{D}$.  As $v\circ\mu\le(1-\epsilon)u\circ\mu+O(1)$ near $\tilde z$ for any $\epsilon>0$, it follows from Siu's decomposition theorem (see \cite{siu74,demailly-book}) that 
	$$v\circ\mu\le c_1\log|w^a|+O(1)=u\circ\mu+O(1)$$
	near $\tilde z$, which implies that $v\le u+O(1)$ near $o$ as $\mu$ is proper. 
	Thus, Lemma \ref{l:sigma-analytic} holds.
\end{proof}

We prove a special case of Lemma \ref{l:0921-3}.

\begin{Lemma}\label{l:0921-1}
	Let $\varphi:=c\log\sum_{1\le j\le m}|f_{j}|^2$ near $o$, where $c>0$ and $f_j$ is a holomorphic function near $o$ for any $1\le j\le m$.  Then
	\[\lim_{N\rightarrow+\infty}c_o^f\big(\max\{\varphi,N\log|z|\}\big)=c_o^{f}(\varphi)\]
	for any holomorphic function $f$ near $o$.
\end{Lemma}

\begin{proof}
	Without loss of generality, assume that $c_o^f(\varphi)<1$. Let $D$ be a small neighborhood of $o$, and 
	$$Y:=\bigcap_j\big\{z\in D:f_j(z)=0\big\}$$ be an analytic sub-variety in $D$.  Using Theorem \ref{thm:desing}, we get a proper holomorphic map $\mu:\widetilde D\rightarrow D$ (local a finite composition of blow-ups with smooth centers), which satisfies that, locally, there is a coordinate system in which the strict transform $\tilde Y$ of $Y$ is a coordinate hyperplane and the exceptional divisor $E'$  is a  union of coordinate hyperplanes. Denote that $X=\mu^{-1}(o)$.
	As $c_o^f(\varphi)<1$, there exists	 $\tilde z\in X$ and $c_1<1$ such that
	\begin{equation}
		\label{eq:0921c}
		\int_{U_{\tilde z}}\mu^*\big(|f|^2e^{-2c_1\varphi}\wedge_{1\le j\le n}\sqrt{-1}dz_j\wedge d\bar z_j\big)=+\infty
	\end{equation} 
	for any neighborhood $U_{\tilde z}$ of $\tilde z$.
	Let $(W;w_1,\ldots,w_n)$ be a coordinate ball centered at $\tilde z$ satisfying that $W\Subset\tilde M$, $w^b=0$ is the zero divisor of the Jacobian $J_{\mu}$ (of $\mu$) and
	$$\varphi\circ\mu(w)=c_2\log|w^a|^2+\tilde u(w)$$
	on $W$,
	where $\tilde u\in\mathcal{C}^{\infty}(\overline W)$, $w^a:=\prod_{j=1}^nw_j^{a_j}$ and $w^b:=\prod_{j=1}^nw_j^{b_j}$. 
	Note that $\mu^*(f)=\sum_{\alpha\in\mathbb{Z}_{\ge0}^n}a_{\alpha}w^{\alpha}$ near $\tilde z$. 
	It follows from	inequality \eqref{eq:0921c}  that there exists $a_{\hat\alpha}\not=0$ such that 
	$$|w^{\hat\alpha+b}|^2e^{-2c_1\varphi\circ\mu}$$
	is not integrable near $\tilde z$, which implies that there exists $j_0\in\{1,\ldots,n\}$ (without loss of generality, assume that $j_0=1$) such that 
	\begin{equation}
		\nonumber
		\hat\alpha_1+b_1+1\le c_1c_2a_1.
	\end{equation}
	Noting that $c_1<1$,  there exists $N\gg0$ such that 
	$$\frac{	\hat\alpha_1+b_1+1}{c_2a_1}+\sum_{2\le j\le n}\frac{\hat\alpha_j+b_j+1}{N}<1,$$
	then it follows from Lemma \ref{l:III} that 
	$$|w^{\hat\alpha+b}|^2e^{-2\log\max\{|w_1|^{c_2a_1},|w_2|^{N},|w_3|^N,\ldots,|w_n|^N\}}$$
	is not integrable near $\tilde z$. Denote that $U_0:=\{|w|<r\}$ ($r>0$ is sufficiently small). We have 
	\begin{equation}
		\nonumber
		\begin{split}
			&\int_{U_0}\mu^*\big(|f|^2e^{-2\max\{\varphi,N\log|z|\}}\big)|J_{\mu}|^2\\
			\ge& C_1\int_{U_0}|w^{\hat\alpha+b}|^2e^{-2\max\{c_2\log|w^a|^2,N\log|w|\}}\\
			\ge& C_2\int_{U_0}|w^{\hat\alpha+b}|^2e^{-2\log\max\{|w_1|^{c_2a_1},|w_2|^{N},|w_3|^N,\ldots,|w_n|^N\}}\\
			=&+\infty,
		\end{split}
	\end{equation} 
	which implies that $|f|^2e^{-2\max\{\varphi,N\log|z|\}}$ is not integrable near $o$, i.e., 
	$$c_o^{f}\big(\max\{\varphi,N\log|z|\}\big)\le 1.$$
	Now, we get that: $c_o^{f}(\varphi)<1$ implies that there exists $N\gg0$ such that 
	$$c_o^{f}\big(\max\{\varphi,N\log|z|\}\big)\le 1.$$ Note that $c_o^{f}\big(\max\{\varphi,N\log|z|\}\big)\ge c_o^{f}(\varphi)$, then we have $$\lim_{N\rightarrow+\infty}c_o^f\big(\max\{\varphi,N\log|z|\}\big)=c_o^{f}(\varphi)$$
	for any holomorphic function $f$ near $o$.
\end{proof}

Now, we prove Lemma \ref{l:0921-3}.

\begin{proof}[Proof of Lemma \ref{l:0921-3}]
	Let $f\not\equiv0$ be a holomorphic function near $o$.
	If $c_o^{f}(\varphi)=0$, then $\varphi\equiv-\infty$ near $o$ and Lemma \ref{l:0921-3} holds clearly. In the following, we assume that $c_o^{f}(\varphi)=1$. 
	
	Lemma \ref{l:0921-2} tells us that $$\lim_{m\rightarrow+\infty}c_o^f(\varphi_m)=1,$$ where $\varphi_m:=\frac{1}{2m}\log\sum_{1\le l\le k_m}|f_{m,l}|^2$ and $\{f_{m,1},\ldots,f_{m,k_m}\}$ is the generators set of $\mathcal{I}(m\varphi)_o$.
	Note that $\varphi_m\ge\varphi+O(1)$ near $o$ by Lemma \ref{l:appro-Berg}, which shows that 
	$$c_o^{f}(\varphi_m)\ge c_o^{f}(\varphi)=1$$
	for any $m$. For any $\epsilon>0$, there exists $m>0$ such that $$c_o^{f}(\varphi_m)\in[1,1+\epsilon),$$ then it follows from Lemma \ref{l:0921-1} that there exists $N_m>0$ such that
	$$c_o^{f}\big(\max\{\varphi_m,N_m\log|z|\}\big)\in[1,1+2\epsilon).$$
	As $\varphi_m\ge\varphi+O(1)$ near $o$, we have 
	$$c_o^{f}\big(\max\{\varphi,N_m\log|z|\}\big)\in[1,1+2\epsilon).$$
	Thus, we have
	$$\lim_{N\rightarrow+\infty}c_o^{f}\big(\max\{\varphi,N\log|z|\}\big)=1=c_o^f(\varphi).$$
	Lemma \ref{l:0921-3} has been proved.
\end{proof}

\vspace{.1in} {\em Acknowledgements}.
	The authors would like to thank Professor Xiangyu Zhou for helpful discussions	and sincerely encouragements.
The second author would like to thank Professor Mattias Jonsson for helpful discussions, and
Professor Chenyang Xu for sharing his recent work and helpful discussions.

The second author was supported by National Key R\&D Program of China 2021YFA1003100, NSFC-12425101, NSFC-11825101, NSFC-11522101 and NSFC-11431013,
and the National Science Foundation Grant No. DMS-163852 and the Ky Fan and Yu-Fen Fan Membership Fund. The third author was supported by the Talent Fund of Beijing Jiaotong University 2024-004. The fourth author was supported by China Postdoctoral Science Foundation BX20230402 and 2023M743719.

\bibliographystyle{references}
\bibliography{xbib}

\end{document}